\documentclass[12pt]{article}

\input xy

\xyoption{all}

\usepackage{amssymb,amsbsy,amsthm,amsmath,amscd,graphicx,mathrsfs,epsfig}

\usepackage{times}

\usepackage{mathabx}

\usepackage{sectsty}

\sectionfont{\large\center}

\partfont{\Large}

\subsectionfont{\normalsize}

\newtheorem{thm}{Theorem}[section]
\newtheorem*{thm*}{Theorem}
\newtheorem{lem}[thm]{Lemma}
\newtheorem{prop}[thm]{Proposition}
\newtheorem{cor}[thm]{Corollary}
\newtheorem{conj}[thm]{Conjecture}
\newtheorem{dfn}[thm]{Definition}

\newtheorem{ques}[thm]{Question}
\theoremstyle{remark}
\newtheorem{ex}[thm]{Example}
\newtheorem*{rmk}{Remark}

\renewcommand{\bf}[1]{\mathbf{#1}}
\renewcommand{\rm}[1]{\mathrm{#1}}

\newcommand{\bbN}{\mathbb{N}}
\newcommand{\bbP}{\mathbb{P}}

\newcommand{\bbR}{\mathbb{R}}

\newcommand{\rmH}{\mathrm{H}}

\newcommand{\rmP}{\mathrm{P}}

\renewcommand{\d}{\mathrm{d}}
\newcommand{\rme}{\mathrm{e}}
\newcommand{\rmh}{\mathrm{h}}

\newcommand{\calN}{\mathcal{N}}
\newcommand{\calO}{\mathcal{O}}

\newcommand{\X}{\mathcal{X}}
\newcommand{\Y}{\mathcal{Y}}

\newcommand{\G}{\Gamma}

\renewcommand{\O}{\Omega}
\renewcommand{\S}{\Sigma}

\renewcommand{\a}{\alpha}

\newcommand{\eps}{\varepsilon}
\newcommand{\g}{\gamma}
\renewcommand{\l}{\lambda}

\newcommand{\s}{\sigma}
\renewcommand{\phi}{\varphi}
\renewcommand{\k}{\kappa}

\newcommand{\Sym}{\mathrm{Sym}}

\renewcommand{\hat}[1]{\widehat{#1}}

\newcommand{\ul}[1]{\underline{#1}}

\newcommand{\fin}{\nolinebreak\hspace{\stretch{1}}$\lhd$}

\newcommand{\actson}{\curvearrowright}
\renewcommand{\to}{\longrightarrow}
\newcommand{\onto}{\twoheadrightarrow}

\renewcommand{\t}{\widetilde}
\newcommand{\lws}{\stackrel{\rm{lw}\ast}{\to}}

\newcommand{\q}{\stackrel{\rm{q}}{\to}}
\newcommand{\dq}{\stackrel{\rm{dq}}{\to}}
\newcommand{\aL}{\stackrel{\rm{aL}}{\to}}
\newcommand{\cov}{\rm{cov}}
\newcommand{\pack}{\rm{pack}}
\newcommand{\ps}{\rm{ps}}
\renewcommand{\Pr}{\mathrm{Prob}}

\begin{document}


\title{\textbf{\Large{Additivity properties of sofic entropy and measures on model spaces}}}
\author{Tim Austin\thanks{Research supported partly by a fellowship from the Clay Mathematics Institute and partly by the Simons Collaboration on Algorithms and Geometry}}

\date{}

\maketitle

\begin{abstract}
Sofic entropy is an invariant for probability-preserving actions of sofic groups.  It was introduced a few years ago by Lewis Bowen, and shown to extend the classical Kolmogorov-Sinai entropy from the setting of amenable groups.  Some parts of Kolmogorov-Sinai entropy theory generalize to sofic entropy, but in other respects this new invariant behaves less regularly.

This paper explores conditions under which sofic entropy is additive for Cartesian products of systems.  It is always subadditive, but the reverse inequality can fail.  We define a new entropy-notion in terms of probability distributions on the spaces of good models of an action.  Using this, we prove a general lower bound for the sofic entropy of a Cartesian product in terms of separate quantities for the two factor systems involved.  We also prove that this lower bound is optimal in a certain sense, and use it to derive some sufficient conditions for the strict additivity of sofic entropy itself.  Various other properties of this new entropy notion are also developed.

\vspace{7pt}

\end{abstract}



%


%
%
\setcounter{tocdepth}{1}
\tableofcontents

\section{Introduction}

Let $G$ be a discrete sofic group, $(X,\mu)$ a standard probability space and ${T:G\actson X}$ a measurable action which preserves $\mu$.  The triple $(X,\mu,T)$ is called a \textbf{$G$-system} or just a \textbf{system}.

Fix a sofic approximation $\S = (\s_n)_{n\geq 1}$ to the group $G$. For a system $(X,\mu,T)$ which has a finite generating partition, Lewis Bowen defined the `sofic entropy relative to $\S$', denoted by $\rmh_\S(\mu,T)$~(\cite{Bowen10}).  An alternative definition which does not require a finite generating partition, and so generalizes Bowen's, was given by Kerr and Li in~\cite{KerLi11b}.  That definition was in terms of operator algebras, but they later gave a more elementary construction of the same invariant:~\cite[Section 3]{KerLi13}.

If $G$ is amenable, then sofic entropy agrees with the classical Kolmogorov-Sinai entropy $\rmh_{\rm{KS}}$ for any choice of sofic approximation (see~\cite{Bowen12,KerLi13}).  If $G$ is not amenable, then sofic entropy can serve as a substitute for Kolmogorov-Sinai entropy for some purposes.  Bowen's original motivation was to show that isomorphic Bernoulli shifts over a sofic group $G$ must have one-dimensional distributions of equal Shannon entropy: this was accomplished for shifts with finite alphabet in~\cite{Bowen10} and then completed in~\cite{KerLi11a}.  (The sufficiency of that condition is also known in many cases, including for any $G$ that contains an infinite amenable subgroup.)  However, it is still uncertain how much of classical entropy theory generalizes to sofic groups, or what modifications are necessary.

This paper considers how $\rmh_\S$ behaves under forming Cartesian products of systems. If $G$ is amenable, then Kolmogorov-Sinai entropy is additive under Cartesian products:
\[\rmh_{\rm{KS}}(\mu\times \nu,T\times S) = \rmh_{\rm{KS}}(\mu,T) + \rmh_{\rm{KS}}(\nu,S).\]
Sofic entropy is subadditive for Cartesian products, and indeed for arbitrary joinings; the easy proof of this is given in Subsection~\ref{subs:subadd}.  However, strict additivity can fail: examples showing this will also be given in that subsection.

Suppose that $(X,\mu,T)$ is a $G$-system. The main innovation of this paper is a new quantity, denoted by $\rmh^\rm{dq}_\S(\mu,T)$, with the property that
\[\rmh_\S(\mu\times\nu,T\times S) \geq \rmh^\rm{dq}_\S(\mu,T) + \rmh_\S(\nu,S)\]
whenever $(Y,\nu,S)$ is another $G$-system for which $\rmh_\S(\nu,S)$ has a certain regularity (it is equal to the `lower' sofic entropy $\ul{\rmh}_\S(\nu,S)$, which is recalled below).

The sofic entropy of $(X,\mu,T)$ is obtained from a certain family of metric spaces consisting of `finitary models' for $(X,\mu,T)$.  The new quantity $\rmh^\rm{dq}_\S(\mu,T)$ is also defined in terms of these spaces, but together with another kind of structure: sequences of probability measures $\mu_n$ on those model spaces.

\subsection*{Model spaces and sofic entropy}

In order to formulate our main results precisely, we first recall the construction of model spaces and the definition of sofic entropy.

Our definition is very close to that in~\cite[Section 3]{KerLi13} but it is adapted slightly better to the purposes of this paper.  The following is only a sketch; full details are given in Section~\ref{sec:model-spaces}.  The equivalence with the Kerr-Li definition is shown in Subsection~\ref{subs:KerrLi}.

First, a \textbf{$G$-process} is a $G$-system in which $X = \X^G$ for some other standard measurable space $\X$ and $S$ is the right-shift action of $G$ on $X$:
\[S^g((x_h)_{h \in G}) := (x_{hg})_{h \in G}.\]
This is close to the probabilistic notion of a `$G$-stationary process': formally, that would be the collection of coordinate projections $\X^G\to \X$, regarded as a $G$-indexed family of $\X$-valued random variables on the probability space $(\X^G,\mu)$.

When we deal with isomorphism-invariant properties of systems, no generality is lost by confining our attention to $G$-processes.  Indeed, for any system $(X,\mu,T)$, the map
\begin{equation}\label{eq:system-process}
\Phi:X\to X^G:x \mapsto (T^hx)_{h \in G}
\end{equation}
intertwines $T$ with the right-shift action $S$ on $X^G$, and converts $\mu$ into the shift-invariant measure $\Phi_\ast\mu$. This map is clearly injective, so it defines a measure-theoretic isomorphism from $(X,\mu,T)$ to the $G$-process $(X^G,\Phi_\ast\mu,S)$.

If $(\X^G,\mu,S)$ is a $G$-process and $F \subseteq G$, then $\mu_F$ denotes the marginal of $\mu$ on $\X^F$.  Also, whenever $F\subseteq F' \subseteq G$, we let $\pi^{F'}_F$ denote the coordinate projection $\X^{F'} \to \X^F$. Thus,
\[\mu_F = (\pi^{F'}_F)_\ast\mu_{F'} = (\pi^G_F)_\ast\mu.\]
If $x \in \X^{F'}$ then we often write $x|_F$ as a shorthand for $\pi^{F'}_F(x)$.

Next, since $\X$ is standard, its $\s$-algebra may be generated as the Borel sets for some compact metric $d$.  We refer to such a $d$ as a \textbf{compact generating metric} for $\X$. Although this metric is far from unique, it is a key auxiliary object in the constructions that follow.  A \textbf{metric $G$-process} is a quadruple $(\X^G,\mu,S,d)$ in which $(\X^G,\mu,S)$ is a $G$-process and $d$ is such a metric on $\X$.  Sofic entropy is initially defined for metric $G$-processes, and then one shows that it does not depend on the choice of metric.  One can extend this fact to allow more general Polish generating metrics on $\X$~\cite{Hay14}, but we do not do so here.  Once a metric $d$ has been chosen, it is always implicit that $\X^G$ has the resulting compact metrizable product topology, and similarly for other products of standard spaces for which we have chosen compact generating metrics.

Now let $V$ be a nonempty finite set.  The space $\X^V$ carries an associated metric defined by
\[d^{(V)}(\bf{x},\bf{x}') = \frac{1}{|V|}\sum_{v \in V}d(x_v,x'_v) \quad \hbox{for}\ \bf{x} = (x_v)_{v \in V},\ \bf{x}' = (x'_v)_{v \in V}.\]
We call this the \textbf{Hamming average} of $d$ over $V$.  It generalizes the classical normalized Hamming metrics, which arise in this way when $(\X,d)$ is a finite set with the discrete metric.

Now consider again a sofic approximation $\S = (\s_n)_{n\geq 1}$ to $G$, where each $\s_n$ is a map from $G$ to $\rm{Sym}(V_n)$ for some finite set $V_n$ (see Section~\ref{sec:model-spaces} for a complete definition). Given $(\X^G,\mu,S,d)$ and $\S$, we define a sequence of subsets
\[\O(\calO,\s_n) \subseteq \X^{V_n}, \quad n \in \bbN,\]
for each choice of a weak$^\ast$-neighbourhood $\calO$ of $\mu$ in $\Pr(\X^G)$.  Note that the weak$^\ast$ topology invoked here depends on the topology of $\X^G$, which in turns depends on the choice of $d$.  The elements of $\O(\calO,\s_n)$ are the `$\calO$-good models of $\mu$ over $\s_n$': the elements of $\X^{V_n}$ whose empirical distributions lie in $\calO$, hence `close' to $\mu$.  Empirical distributions are defined in Section~\ref{sec:model-spaces}.  The `quality' required of these models improves as $\calO$ is reduced, and it is clear that
\[\calO' \subseteq \calO \quad \Longrightarrow \quad \O(\calO',\s) \subseteq \O(\calO,\s)\]
for any finite $V$ and map $\s:G\to \Sym(V)$.

For any compact metric space $(Y,d_Y)$, subset $Z \subseteq Y$, and $\delta > 0$, we let $\rm{cov}_\delta(Z,d_Y)$ be the minimum cardinality among $\delta$-dense subsets of $Z$.  The sofic entropy $\rmh_\S(\mu)$ is defined to be
\[\sup_{\delta > 0}\ \inf_{\calO}\ \limsup_{n\to\infty}\ \frac{1}{|V_n|}\log \rm{cov}_\delta\big(\O(\calO,\s_n),d^{(V_n)}\big),\]
where $\calO$ ranges over all weak$^\ast$-neighbourhoods of $\mu$. Heuristically, this is approximately the exponential growth rate of the covering numbers
\[\rm{cov}_\delta\big(\O(\calO,\s_n),d^{(V_n)}\big)\]
as $n\to\infty$, for sufficiently small $\delta$ and then for sufficiently small $\calO$ depending on $\delta$.

In general, this sequence of covering numbers need not grow at a well-defined exponential rate, so one takes the supremum of those rates over subsequences.  It can be important to know when the covering numbers have different asymptotics along other subsequences.  To capture this possibility, one also defines the \textbf{lower sofic entropy}
\[\ul{\rmh}_\S(\mu) := \sup_{\delta > 0}\ \inf_{\calO}\ \liminf_{n\to\infty}\ \frac{1}{|V_n|}\log \rm{cov}_\delta\big(\O(\calO,\s_n),d^{(V_n)}\big).\]

It can happen that $\ul{\rmh}_\S(\mu) < \rmh_\S(\mu)$: see the end of Subsection~\ref{subs:defs}.  If $\ul{\rmh}_\S(\mu) = \rmh_\S(\mu)$, then this asserts the following: for every $\eps > 0$ there is a $\delta_0 > 0$ such that for every $\delta \in (0,\delta_0)$ there is a weak$^\ast$-neighbourhood $\calO_\delta$ such that for every weak$^\ast$-neighbourhood $\calO \subseteq \calO_\delta$ we have
\[\Big|\frac{1}{|V_n|}\log \rm{cov}_\delta\big(\O(\calO,\s_n),d^{(V_n)}\big) - \rmh_\S(\mu)\Big| < \eps\]
for all sufficiently large $n$.  More heuristically: if $\delta$ is sufficiently small and then $\calO$ is sufficiently small depending on $\delta$, the covering numbers do grow at an approximately well-defined exponential rate.  Since it suffices to check this for rational $\delta$ and for a countable basis of  neighbourhoods $\calO$, a simple diagonal argument can always provide a subsequence of $(\s_n)_{n\geq 1}$ for which this is the case.

The quantity $\rmh_\S(\mu)$ generally depends on the choice of sofic approximation $\S$.  But, crucially, it does not depend on the compact metric $d$ that one uses to generate the $\s$-algebra of $\X$.  In fact, $\rmh_\S(\mu)$ and likewise $\ul{\rmh}_\S(\mu)$ are invariants of the system $(\X^G,\mu,S)$ up to measure-theoretic isomorphism, for a fixed choice of the sofic approximation $\S$.  One may therefore define $\rmh_\S(\mu,T)$ and similarly $\ul{\rmh}_\S(\mu,T)$ for an arbitrary $G$-system $(X,\mu,T)$, for instance by using the isomorphism~(\ref{eq:system-process}).

\subsection*{Measures on model spaces}

Our new invariant is defined in terms of sequences of measures $\mu_n$ on $\X^{V_n}$ which are asymptotically supported on these model spaces and which locally resemble $\mu$ at most points of $V_n$, which we refer to as `vertices'.

Section~\ref{sec:loc-and-quench} will consider three senses in which a sequence of probability measures $\mu_n$ on $\X^{V_n}$ can converge to a measure $\mu$ on $\X^G$: local weak$^\ast$ convergence, quenched convergence, and doubly-quenched convergence.  All three senses are relative to a particular choice of sofic approximation $\S$; we refer to convergence `over $\S$' if we need to make that choice explicit.  They are also relative to a particular choice of compact generating metric $d$ for $\X$.

Local weak$^\ast$ convergence asserts that, once $n$ is large, the marginals of $\mu_n$ around most points of $V_n$ resemble the corresponding marginal of $\mu$ in the weak$^\ast$ topology.  Here we use that the sofic approximation $\s_n$ gives a way to copy a fixed finite subset of $G$ to a corresponding `patch' around any vertex of $V_n$, perhaps with errors for a few vertices.  This convergence is denoted by $\mu_n \lws \mu$. This notion already has an important role in the study of various statistical physics models on random graphs.

Quenched convergence strengthens local weak$^\ast$ convergence by imposing a second condition: that $\mu_n$ be mostly supported on individual good models of $\mu$. The term `quenched' is also taken from statistical physics, where it indicates a property that holds among most instances in an ensemble, not just on average.  This convergence is denoted by $\mu_n \q \mu$.  Quenched convergence is strictly stronger than local weak$^\ast$ convergence in general, and the difference between them has a simple characterization in terms of certain random measures constructed from the $\mu_n$ (Lemma~\ref{lem:lwa}).  However, using this characterization, it follows that the two notions are equivalent if $\mu$ is ergodic (Corollary~\ref{cor:lws-lwa}).

A closely related notion of convergence for a sequence of measures on model spaces already appears in~\cite{Bowen11}.  Using this notion, that paper gives a new formula for sofic entropy for certain special examples of probability-preserving systems and sofic approximations.

For any compact metric space $(Y,d_Y)$, Borel probability measure $\nu$ on $Y$, and $\eps,\delta > 0$, we write
\[\cov_{\eps,\delta}(\nu,d_Y):= \min\big\{|F|:\ F\subseteq Y,\ \nu(B_\delta(F)) > 1 - \eps\big\},\]
where $B_\delta(F)$ is the $\delta$-neighbourhood of $F$ according to the metric $d_Y$.  Using this quantity in place of the covering numbers of spaces themselves, we define the following analog of $\rmh_\S$:
\begin{multline*}
\rmh^{\rm{q}}_\S(\mu) := \sup\Big\{\sup_{\delta,\eps > 0}\limsup_{i\to\infty}\frac{1}{|V_{n_i}|}\log\rm{cov}_{\eps,\delta}\big(\mu_i,d^{(V_{n_i})}\big):\\ n_i \uparrow \infty\ \hbox{and}\ \mu_i \q \mu\ \hbox{over}\ (\s_{n_i})_{i\geq 1}\Big\}.
\end{multline*}
The outer supremum here is over all subsequences $(\s_{n_i})_{i\geq 1}$ of the sofic approximation $\S$, and over all sequences of measures $\mu_i$ on $\X^{V_{n_i}}$ that quenched-converge to $\mu$ over that subsequence.  We must allow this supremum over subsequences, because it may be that there is no sequence of measures $\mu_n$ such that $\mu_n \q \mu$ over the original sofic approximation at all.  This will be explained more carefully at the beginning of Section~\ref{sec:mod-meas-sof-ent}.

We call $\rmh^\rm{q}_\S(\mu)$ the \textbf{model-measure sofic entropy of $\mu$ rel $\S$}.   Like sofic entropy, it is an isomorphism-invariant of the $G$-process (Theorem~\ref{thm:iso-invar}), and so in fact it does not depend on the choice of the generating metric $d$. As a result, its definition can be extended unambiguously to arbitrary $G$-systems.  Since quenched convergence $\mu_n \q \mu$ requires that $\mu_n$ be mostly supported on good models for $\mu$ once $n$ is large, it follows easily that $\rmh_\S \geq \rmh^\rm{q}_\S$ (see Lemma~\ref{lem:simple-ineq}). This inequality can be strict.

Like sofic entropy, $\rmh^\rm{q}_\S$ is always subadditive under Cartesian products, but may not be strictly additive.  However, this defect can be repaired by further restricting the sequences of measures on model spaces that we allow.

If $\mu_n \lws \mu$, then it follows easily that $\mu_n \times \mu_n \lws \mu\times \mu$.  However, the same implication may fail for quenched convergence: even if $\mu_n$ is asymptotically mostly supported on good models for $\mu$, the product $\mu_n \times \mu_n$ may not be mostly supported on good models for $\mu\times \mu$.  This phenomenon is responsible for cases in which $\rmh^\rm{q}_\S(\mu^{\times 2}) < 2\rmh_\S^\rm{q}(\mu)$.  However, if we simply \emph{require} the convergence of Cartesian squares
\[\mu_n \times \mu_n \q \mu\times \mu,\]
then it turns out that this implies good behaviour for all other Cartesian products with the measures $\mu_n$.

\vspace{7pt}

\noindent\textbf{Theorem A.}\quad \emph{Suppose that $\mu_n \q \mu$ over $\S$.  The following are equivalent:
\begin{itemize}
\item[i)] $\mu_n\times \mu_n \q \mu\times \mu$;
\item[ii)] if $(\Y^G,\nu,S,d_\Y)$ is another metric $G$-process and $\calN$ is a weak$^\ast$ neighbourhood of $\mu\times \nu$, then there is a weak$^\ast$ neighbourhood $\calO$ of $\nu$ such that
\[\inf_{\bf{y} \in \O(\calO,\s_{n_i})}\mu_{n_i}\big\{\bf{x} \in \X^{V_{n_i}}:\ (\bf{x},\bf{y}) \in \O(\calN,\s_{n_i})\big\} \to 1 \quad \hbox{as}\ i \to\infty\]
for any subsequence $n_i \uparrow \infty$ such that $\O(\calO,\s_{n_i}) \neq \emptyset$ for all $i$ (we regard this as vacuously true if there is no such subsequence $n_i$);
\item[iii)] if $(\Y^G,\nu,S,d_\Y)$ is another metric $G$-process, $n_i \uparrow \infty$, and $\nu_i \in \Pr(\Y^{V_{n_i}})$ is a sequence such that $\nu_i \q \nu$ over $(\s_{n_i})_{i\geq 1}$, then $\mu_{n_i}\times \nu_i \q \mu\times \nu$ over $(\s_{n_i})_{i\geq 1}$.
\end{itemize}}

\vspace{7pt}

Under any of the above equivalent conditions, we say that $( \mu_n)_{n\geq 1}$ \textbf{doubly-quenched converges} to $\mu$, and denote this by $\mu_n \dq \mu$.

Theorem A is analogous to the equivalence among various standard characterizations of weak mixing.  Furthermore, since local weak$^\ast$ convergence implies quenched convergence when $\mu$ is ergodic, one can deduce that quenched convergence implies doubly-quenched convergence when $\mu$ is weakly mixing (Lemma~\ref{lem:q-dq-and-wm}).

Doubly-quenched convergence finally leads to the new invariant we need:
\begin{multline*}
\rmh^{\rm{dq}}_\S(\mu) := \sup\Big\{\sup_{\delta,\eps > 0}\limsup_{i \to\infty}\frac{1}{|V_{n_i}|}\log\rm{cov}_{\eps,\delta}\big(\mu_i,d^{(V_{n_i})}\big):\\ n_i \uparrow \infty\ \hbox{and}\ \mu_i \dq \mu\ \hbox{over}\ (\s_{n_i})_{i\geq 1}\Big\}.
\end{multline*}
This is called the \textbf{doubly-quenched model-measure sofic entropy of $\mu$ rel $\S$}.  One sees easily that $\rmh^\rm{dq}_\S \leq \rmh^\rm{q}_\S$ (Lemma~\ref{lem:simple-ineq}). The proof that $\rmh^\rm{q}_\S$ is isomorphism-invariant (and hence independent of the choice of $d$) gives the same result for $\rmh^\rm{dq}_\S$, and so the definition of $\rmh_\S^{\rm{dq}}$ can be extended unambiguously to any $G$-system.

We can now state our main result for Cartesian products.

\vspace{7pt}

\noindent\textbf{Theorem B.}\quad \emph{Suppose that $(X,\mu,T)$ and $(Y,\nu,S)$ are $G$-systems such that
\begin{equation}\label{eq:upper=lower}
\rmh_\S(\nu,S) = \ul{\rmh}_\S(\nu,S).
\end{equation}
Then
\begin{equation}\label{eq:add1}
\rmh_\S(\mu\times\nu,T\times S) \geq \rmh^\rm{dq}_\S(\mu,T) + \rmh_\S(\nu,S).
\end{equation} }

\vspace{7pt}

This follows fairly easily from conclusion (ii) of Theorem A.

Of course, by symmetry, one also has the analogous conclusion with the roles of $(X,\mu,T)$ and $(Y,\nu,S)$ reversed.

One cannot hope for anything like~(\ref{eq:add1}) without some assumption such as~(\ref{eq:upper=lower}), since in general the relevant entropies $\rmh_\S^\rm{dq}(\mu,T)$ and $\rmh_\S(\nu,S)$ could be obtained as limit suprema along disjoint subsequences.

Since $\rmh_\S^\rm{dq} \leq \rmh_\S$ and $\rmh_\S$ is always subadditive, the following is an immediate corollary.

\vspace{7pt}

\noindent\textbf{Corollary B$'$.}\quad \emph{If $(X,\mu,T)$ and $(Y,\nu,S)$ satisfy
\[\rmh_\S^\rm{dq}(\mu,T) = \rmh_\S(\mu,T) \quad \hbox{and} \quad \ul{\rmh}_\S(\nu,S) = \rmh_\S(\nu,S),\]
or vice-versa, then
\[\rmh_\S(\mu\times\nu,T\times S) = \rmh_\S(\mu,T) + \rmh_\S(\nu,S).\]}

\vspace{7pt}

For instance, this condition on $(X,\mu,T)$ is satisfied by Bernoulli systems, so we recover the known result~\cite[Section 8]{Bowen10} that forming products with Bernoulli systems always makes the obvious additive contribution to sofic entropy.

It can happen that $\rmh^\rm{q}_\S > \rmh^\rm{dq}_\S$, and there are cases in which one cannot replace $\rmh^\rm{dq}_\S(\mu,T)$ with $\rmh^\rm{q}_\S(\mu,T)$ in Theorem B: see Example~\ref{ex:hq-neq-hdq}.  However, this cannot occur if $(X,\mu,T)$ is weakly mixing, simply because quenched convergence itself implies doubly-quenched convergence for weakly mixing systems.

Unlike the other notions, doubly-quenched model-measure sofic entropy does enjoy a general additivity result for Cartesian products.  The only caveat is that we must still assume some analog of condition~(\ref{eq:upper=lower}) in Theorem B.  This is conveniently expressed in terms of a `lower' version of doubly-quenched model-measure sofic entropy, denoted $\ul{\rmh}_\S^\rm{dq}$.

The additivity result for $\rmh^\rm{dq}_\S$ has an analog for $\rmh^\rm{q}_\S$ under an additional ergodicity assumption.

\vspace{7pt}

\noindent\textbf{Theorem C.}\quad \emph{For any $G$-systems $(X,\mu,T)$ and $(Y,\nu,S)$, it holds that
\[\rmh^{\rm{dq}}_\S(\mu\times \nu,T\times S) \leq \rmh^{\rm{dq}}_\S(\mu,T) + \rmh^{\rm{dq}}_\S(\nu,S),\]
and similarly with $\rmh^\rm{dq}_\S$ replaced by $\rmh_\S^\rm{q}$.}

\emph{If $\rmh^{\rm{dq}}_\S(\nu,S) = \ul{\rmh}^{\rm{dq}}_\S(\nu,S)$, then in fact
\begin{equation}\label{eq:add2}
\rmh^{\rm{dq}}_\S(\mu\times \nu,T\times S) = \rmh^{\rm{dq}}_\S(\mu,T) + \rmh^{\rm{dq}}_\S(\nu,S),
\end{equation}
where we interpret the right-hand side as $-\infty$ if either of its terms is $-\infty$. }

\emph{If $\mu\times \nu$ is ergodic, then the analogous result holds with $\rmh^\rm{dq}_\S$ and $\ul{\rmh}_\S^\rm{dq}$ replaced by $\rmh_\S^\rm{q}$ and $\ul{\rmh}_\S^\rm{q}$ throughout.  In particular, this is the case if both $\mu$ and $\nu$ are ergodic and one of them is weakly mixing.}

\vspace{7pt}

Theorems B and C will be proved in Section~\ref{sec:prod}.  The first conclusion of Theorem C (subadditivity) actually holds for arbitrary joinings: see Proposition~\ref{prop:h-q-subadd}.

By applying Theorem C to copies of a single system $(X,\mu,T)$, we can show that $\rmh^\rm{dq}_\S$ is stable under Cartesian powers:
\[\rmh^{\rm{dq}}_\S(\mu^{\times k}) = k\cdot \rmh^{\rm{dq}}_\S(\mu) \quad \forall k \geq 1\]
(see Corollary~\ref{cor:dq-is-stable}).  In this case we can do without the assumption that $\rmh^{\rm{dq}}_\S(\mu,T) = \ul{\rmh}^{\rm{dq}}_\S(\mu,T)$.

\subsection*{Processes with finite state spaces}

In case $\X$ is a finite set and $(\X^G,\mu,S)$ is a $G$-process, we are able to prove another relation between the new invariant $\rmh_\S^{\rm{dq}}$ and sofic entropies.  In view of isomorphism-invariance, this will actually apply to any system $(X,\mu,T)$ which has a finite generating partition. For an ergodic system, this, in turn, is equivalent to finiteness of the Rokhlin entropy $\rmh^\rm{Rok}(\mu,T)$, by the results of~\cite{Seward--KriI}.

For a general system $(X,\mu,T)$, consider the sequence of values
\[\frac{1}{k}\rmh_\S(\mu^{\times k},T^{\times k}), \quad k\geq 1.\]
Since $\rmh_\S$ is subadditive, the sequence $(\rmh_\S(\mu^{\times k},T^{\times k}))_{k\geq 1}$ is subadditive.  Therefore the limit
\[\rmh_\S^\ps(\mu,T) := \lim_{k\to\infty}\frac{1}{k}\rmh_\S(\mu^{\times k},T^{\times k})\]
exists and satisfies $\rmh_\S^\ps(\mu,T) \leq \rmh_\S(\mu,T)$, by Fekete's Lemma.  We call it the \textbf{power-stabilized sofic entropy rel $\S$}.  It is clearly an isomorphism-invariant.  If $\rmh_\S^\ps(\mu,T) < \rmh_\S(\mu,T)$, then this gap quantifies the failure of additivity of sofic entropy among the Cartesian powers of $(X,\mu,T)$ itself.

\vspace{7pt}

\noindent\textbf{Theorem D.}\quad \emph{It is always the case that
\[\rmh_\S^\ps(\mu,T) \geq \rmh^{\rm{dq}}_\S(\mu,T),\]
and this is an equality if $(X,\mu,T)$ has a finite generating partition.}

\vspace{7pt}

It is fairly easy to prove that $\rmh^\ps_\S \geq \rmh^\rm{dq}_\S$, so most of the work goes into proving the reverse inequality. This will be done by showing that individual good models for large Cartesian powers $(\mu^{\times k},T^{\times k})$ can be converted into measures that doubly-quenched converge to $\mu$.

Theorem D leads to the following sense in which Theorem B is optimal for systems that have finite generating partitions.

\vspace{7pt}

\noindent\textbf{Corollary D$'$.}\quad \emph{If $(X,\mu,T)$ has a finite generating partition and $\rmh_\S(\mu^{\times k},T^{\times k}) = \ul{\rmh}_\S(\mu^{\times k},T^{\times k})$ for all $k$, then
\begin{multline*}
\rmh^\rm{dq}_\S(\mu,T) = \inf\big\{\rmh_\S(\mu\times \nu,T\times S) - \rmh_\S(\nu,S):\\ (Y,\nu,S)\ \hbox{another $G$-system with}\ \rmh_\S(\nu,S) = \ul{\rmh}_\S(\nu,S)\big\}.
\end{multline*} }

\vspace{7pt}

Thus, if $(X,\mu,T)$ has a finite generating partition and $\rmh_\S(\mu^{\times k},T^{\times k}) = \ul{\rmh}_\S(\mu^{\times k},T^{\times k})$ for all $k$, then no other quantity which depends only on $(X,\mu,T)$ can improve on $\rmh^\rm{dq}_\S(\mu,T)$ in Theorem B.

\subsection*{Equality in the case of co-induced systems}

For an infinite sofic group $G$ and a Bernoulli process $(\X^G,\nu^{\times G},S)$, it is fairly easy to show that $\rmh_\S(\nu^{\times G})$, $\rmh^\rm{q}_\S(\nu^{\times G})$ and $\rmh^\rm{dq}_\S(\nu^{\times G})$ are all just equal to the Shannon entropy of $\nu$.  (In the case of $\rmh_\S$, this is the calculation that gives the classification of Bernoulli systems in~\cite{Bowen10,KerLi11a}.)  This supplies some examples for which the three entropy-notions coincide.

The final topic of this paper is a generalization of this result to a class of co-inductions.  Suppose now that $G$ and $H$ are two sofic groups, and let
\[\S = (\s_n:G\to \rm{Sym}(V_n))_{n\geq 1} \quad \hbox{and} \quad \rm{T} = (\tau_n:H\to \rm{Sym}(W_n))_{n\geq 1}\]
be respective sofic approximations for them. Then $G\times H$ has a \textbf{product} sofic approximation $\S\times \rm{T} = (\s_n\times \tau_n)_{n\geq 1}$, where
\[(\s_n\times \tau_n)^{(g,h)} := \s_n^g\times \tau_n^h \in \rm{Sym}(V_n\times W_n).\]

Let $(X,\mu,T)$ be a $G$-system.  Then \textbf{co-induction} gives a functorial way to construct from it a $(G\times H)$-system: the new probability space is $(X^H,\mu^{\times H})$, and the action $\rm{CInd}_G^{G\times H}T$ is defined by setting
\[(\rm{CInd}_G^{G\times H}T)^g := (T^g)^{\times H} \quad \hbox{for}\ g \in G\]
and
\[(\rm{CInd}_G^{G\times H}T)^h((x_k)_{k \in H}) := (x_{kh})_{k \in H} \quad \hbox{for}\ h \in H\]
(so $H$ acts simply by the right-shift on $X^H$). Since these commute, they define an action of $G\times H$.  See~\cite[Subsection II.10.(G)]{Kec10} or~\cite{DooZha12} for a more general discussion and some previous uses of co-induction in ergodic theory.

If $G$ is the trivial group, then this just produces an $H$-Bernoulli shift.  The following result therefore generalizes our observation about Bernoulli shifts.

\vspace{7pt}

\noindent\textbf{Theorem E.}\quad \emph{Let $G$, $H$, $\S$, $\rm{T}$, and $(X,\mu,T)$ be as above.  Assume that $H$ is infinite. Then
\[\rmh_{\S\times \rm{T}}\big(\mu^{\times H},\rm{CInd}_G^{G\times H}T\big) = \rmh^\rm{q}_{\S\times \rm{T}}\big(\mu^{\times H},\rm{CInd}_G^{G\times H}T\big) = \rmh^\rm{dq}_{\S\times \rm{T}}\big(\mu^{\times H},\rm{CInd}_G^{G\times H}T\big).\]}

\vspace{7pt}

For a $G$-system $(X,\mu,T)$, a different condition which is sufficient for $\rmh_\S(\mu,T) = \rmh^{\rm{q}}_\S(\mu,T)$ has appeared previously as~\cite[Theorem 4.1]{Bowen11}.

Combining Theorem E with Theorems B and D immediately yields the following.

\vspace{7pt}

\noindent\textbf{Corollary E$'$.}\quad \emph{Consider again the setting of Theorem E.
\begin{enumerate}
\item If $(Y,\nu,S)$ is another $(G\times H)$-system such that
\[\rmh_{\S\times \rm{T}}(\nu,S) = \ul{\rmh}_{\S\times \rm{T}}(\nu,S),\]
then
\[\rmh_{\S\times \rm{T}}\big(\mu^{\times H}\times \nu,\rm{CInd}_G^{G\times H}T \times S\big)\\ = \rmh_{\S\times \rm{T}}\big(\mu^{\times H},\rm{CInd}_G^{G\times H}T\big) + \rmh_{\S\times \rm{T}}(\nu,S).\]
\item If $(X,\mu,T)$ has a finite generating partition, then
\[\rmh_{\S\times \rm{T}}\big((\mu^{\times H})^{\times k},(\rm{CInd}_G^{G\times H}T)^{\times k}\big) = k\cdot \rmh_{\S\times \rm{T}}(\mu^{\times H},\rm{CInd}_G^{G\times H}T)\]
for all $k \geq 1$.
\qed
\end{enumerate} }

\vspace{7pt}

\subsection*{Overview of the paper}

The rest of this paper is divided into two parts.

Part I concerns the spaces of good models for a $G$-process.  After collecting some background material in Section~\ref{sec:background}, model spaces and sofic entropy are defined carefully and studied in Section~\ref{sec:model-spaces}.

The most substantial results of this part are in Section~\ref{sec:factors}, which describes how a factor map between $G$-processes can be converted into `approximately Lipschitz' maps between their model spaces, endowed with suitable Hamming-like metrics.  This is delicate, because an arbitrary measurable factor map must first be approximated by `almost continuous' maps, and then these can be modified to act on the model-spaces.

Part II introduces measures on the spaces of models of a $G$-process.  Section~\ref{sec:loc-and-quench} defines locally weak$^\ast$, quenched and doubly-quenched convergence, and establishes some basic properties, including Theorem A.  Then Section~\ref{sec:mod-meas-sof-ent} uses such convergent sequences of measures to define the associated entropy-notions and prove some preparatory results for them.  These then lead to proofs of all the main theorems stated above: Theorems B and C in Section~\ref{sec:prod}; Theorem D in Section~\ref{sec:finite-state-spaces}; and Theorem E in Section~\ref{sec:when-equal}.

When a factor map between $G$-processes is converted into maps between the model spaces, those can then be used to convert a convergent sequence of measures for the domain process into a convergent sequence of measures for the target process.  This construction is crucial for many of the proofs in Part II, but, as might be expected from Part I, it requires careful control of several different approximations.  This is the most technical aspect the paper.

Section~\ref{sec:open} collects some open questions about this new notion of entropy.

Many of the ideas in Part I are just small variations on~\cite{Bowen10,KerLi11a,KerLi11b,KerLi13}.  Certainly none of Section~\ref{sec:model-spaces} is really original.  The principal difference from those works is our explicit development of maps between model spaces corresponding to factor maps between systems.  This gives us some very versatile tools for passing between model spaces, as illustrated by their use in the proofs of the main theorems.  We develop the somewhat new formalism for sofic entropy in Section~\ref{sec:model-spaces} in order to make this analysis of maps simpler and more natural, and to have the same effect on our definitions of model-measure sofic entropies in Part II.  This is why we include this new formalism, rather than just using the definitions from~\cite{KerLi13}, for example.

\subsubsection*{Acknowledgements}

I am grateful for several discussions with Mikl\'os Ab\'ert, Lewis Bowen, David Fisher, Alex Lubotzky and Brandon Seward. After the first versions of this paper appeared, I learned that Mikl\'os Ab\'ert and Benjamin Weiss have been studying some quite similar definitions using measures on model spaces.  I also thank the referee for correcting several small errors.

This research was supported partly by a fellowship from the Clay Mathematics Institute and partly by the Simons Collaboration on Algorithms and Geometry

\part{Model spaces and maps}

\section{Some notation and preliminaries}\label{sec:background}

\subsection{Elementary analysis}

We use Landau (`big-$O$' and `little-$o$') notation without further comment. Among real numbers, we sometimes write `$a \approx_\eps b$' in place of `$|a-b| < \eps$'.  The notation `$\eps_n \downarrow 0$' means that $(\eps_n)_{n\geq 1}$ is a non-increasing sequence of strictly positive real numbers which tends to $0$.

Now let $\rmP$ be a property that holds for some non-decreasing sequences of non-negative integers (that is, a subset of the set of non-decreasing members of $\bbN^\bbN$).  We will say that $\rmP$ holds \textbf{whenever $m_1 \leq m_2 \leq \dots$ grows sufficiently slowly} if there is a fixed non-decreasing sequence $(m_n^\circ)_{n\geq 1}$ with $m_n^\circ \uparrow \infty$ such that, for any other non-decreasing sequence $(m_n)_{n\geq 1} \in \bbN^\bbN$, we have
\[[m_n \to\infty\ \hbox{and}\ m_n \leq m_n^\circ\ \forall n\ ] \quad \Longrightarrow \quad (m_n)_{n \geq 1}\ \hbox{has}\ \rmP. \]

The following nomenclature from probabilistic combinatorics will be very convenient.  If $(\O_n,\bbP_n)_{n \geq 1}$ is a sequence of probability spaces, and $\rm{P}$ is a property that holds for some elements of $\O_n$ for each $n$, then $\rm{P}$ holds \textbf{with high probability} (`\textbf{w.h.p.}') if
\[\bbP_n\{\omega \in \O_n:\ \rm{P}\ \hbox{holds for}\ \omega\} \to 1 \quad \hbox{as}\ n\to\infty.\]
If $\rm{P}$ is written out explicitly in terms of $\omega$, then we may also write that $\rm{P}$ holds \textbf{w.h.p. in $\omega$}.  If each $\O_n$ is finite and $\bbP_n$ is not specified, then this is to be understood with $\bbP_n$ equal to the uniform measure on $\O_n$.

\subsection{Metric spaces and almost Lipschitz maps}

Let $(X,d_X)$ and $(Y,d_Y)$ be metric spaces, let $\eps > 0$, and let $L < \infty$.  A map $\phi:X \to Y$ is \textbf{$\eps$-almost $L$-Lipschitz} if it is Borel measurable and satisfies
\[d_Y(\phi(x),\phi(x')) \leq \eps + Ld_X(x,x') \quad \forall x,x' \in X.\]
A map is \textbf{$\eps$-almost Lipschitz} if it is so for some $L$.

The following requires only an immediate check.

\begin{lem}\label{lem:almost-Lip-with-Lip}
If $\phi:X\to Y$ is $\eps$-almost $L$-Lipschitz and $f:Y\to \bbR$ is $K$-Lipschitz, then $f\circ \phi$ is $(K\eps)$-almost $(KL)$-Lipschitz. \qed
\end{lem}

Maps that are $\eta$-almost Lipschitz for arbitrarily small $\eta$ have the following simple characterization.

\begin{lem}\label{lem:cts-alm-Lip}
If $(X,d_X)$ and $(Y,d_Y)$ are compact and $f:X\to Y$, then $f$ is continuous if and only if it is $\eta$-almost Lipschitz for every $\eta > 0$.
\end{lem}

\begin{proof}
The reverse implication is simple, so we focus on the forward implication.

Suppose for simplicity that $d_Y$ has diameter at most $1$. Given $\eta > 0$, let $\delta > 0$ be so small that if $x,x' \in X$ then
\[d_X(x,x') < \delta \quad \Longrightarrow \quad d_Y(f(x),f(x')) < \eta.\]
Now define $L := 1/\delta$.  For any $x,x' \in X$, we obtain
\[d_Y(f(x),f(x')) \leq \left\{\begin{array}{lll}\eta &\quad \hbox{if}\ d_X(x,x') < \delta\\
1 < \eta + L\delta &\quad \hbox{if}\ d_X(x,x') \geq \delta\end{array}\right\} \leq \eta + Ld_X(x,x').\]
\end{proof}

For a general target space, an $\eps$-almost Lipschitz map need not be close to a truly Lipschitz map, even if $\eps$ is very small. However, this does hold among $\bbR$-valued maps.

\begin{lem}\label{lem:almost-Lip-near-Lip}
If $(X,d)$ is a metric space, $U \subseteq X$ is nonempty, and $f:X \to \bbR$ is a function such that $f|U$ is $\eps$-almost $L$-Lipschitz, then there is an $L$-Lipschitz map $g:X \to \bbR$ such that $|f(x) - g(x)| \leq \eps$ for all $x \in U$.
\end{lem}

\begin{proof}
The following standard construction gives a suitable approximant:
\[g(x) := \inf_{x' \in U}(f(x') + Ld(x,x')).\]
\end{proof}

As in the Introduction, if $V$ is a nonempty finite set and $(X,d)$ is a metric space, then $d^{(V)}$ denotes the Hamming average metric on $X^V$ defined by
\[d^{(V)}(x,x') := \frac{1}{|V|}\sum_{v\in V}d(x_v,x'_v).\]
If $V \subseteq U$ with $V$ nonempty and finite, then $d^{(V)}(x,x')$ is still well-defined for pairs $x,x' \in X^U$.  It defines a pseudometric on $X^U$.  In particular, if $G$ is a group, then this gives a pseudometric $d^{(F)}$ on $X^G$ for every nonempty finite $F\subseteq G$.

Probability measures on metric spaces will always be defined on their Borel $\s$-algebras.  The set of Borel probability measures on a metric space $X$ is denoted by $\Pr(X)$.  If $X$ is compact then $\Pr(X)$ is given the weak$^\ast$ topology.

\subsection{Approximating Borel maps by almost Lipschitz maps}

In order to study factor maps between systems, we will need to approximate Borel maps by maps that have a fairly explicit kind of `approximate continuity'.  For our purposes, the best-adapted approximants seem to be almost Lipschitz maps.  In this subsection we prove the existence of such approximants using Lusin's Theorem.

The following definition is not standard, but will be very convenient.  It is a prelude to Definition~\ref{dfn:eta-approx}, which is a dynamical version.

\begin{dfn}\label{dfn:eps-good}
Let $(X,d_X)$ and $(Y,d_Y)$ be compact metric spaces, let $\mu \in \Pr(X)$, let $\phi:X\to Y$ be Borel, and let $\eta > 0$.  Then an \textbf{$\eta$-almost Lipschitz} (or \textbf{$\eta$-AL}) \textbf{approximation to $\phi$ rel $\mu$} is a Borel map $\psi:X\to Y$ with the following properties:
\begin{itemize}
\item[i)] the map $\psi$ approximates $\phi$ in the sense that
\[\int d_Y(\phi(x),\psi(x))\,\mu(\d x) < \eta;\]
\item[ii)] there is an open subset $U \subseteq X$ such that $\mu(U) > 1 - \eta$, and such that $\psi|U$ is $\eta$-almost Lipschitz from $d_X$ to $d_Y$.
\end{itemize}
\end{dfn}

\begin{lem}\label{lem:Lusin1}
Let $(X,d_X)$ and $(Y,d_Y)$ be compact metric spaces, let $\mu \in \Pr(X)$, and let $\phi:X\to Y$ be Borel.  Then $\phi$ has $\eta$-AL approximations rel $\mu$ for all $\eta > 0$.
\end{lem}

\begin{proof}
We may assume for simplicity that $d_Y$ has diameter at most $1$.

Lusin's Theorem gives a compact subset $K\subseteq X$ such that $\mu(K) > 1 - \eta$ and $\phi|K$ is continuous.  Then Lemma~\ref{lem:cts-alm-Lip} gives $L < \infty$ such that $\phi|K$ is $(\eta/3)$-almost $L$-Lipschitz.

Now choose $\eps$ so small that $L\eps < \eta/3$, and let $U := B_{\eps}(K)$.  Certainly $\mu(U) > 1 - \eta$.  Let $\xi: X \to K$ be a Borel map such that $\xi|K = \rm{id}_K$ and $d_X(x,\xi(x)) < \eps$ for all $x \in U$, and define $\psi:= \phi\circ \xi$.  This gives $\phi(x) = \psi(x)$ for all $x \in K$, and hence
\[\int d_Y(\phi(x),\psi(x))\,\mu(\d x) \leq \mu(X\setminus K) < \eta.\]

Finally, for any $x,x' \in U$, we have $\xi(x),\xi(x') \in K$, and therefore
\begin{multline*}
d_Y(\psi(x),\psi(x')) = d_Y\big(\phi(\xi(x)),\phi(\xi(x'))\big) \leq \eta/3 + Ld_X(\xi(x),\xi(x'))\\
\leq \eta/3 + 2L\eps + Ld_X(x,x') < \eta + Ld_X(x,x').
\end{multline*}
So $\psi|U$ is $\eta$-almost $L$-Lipschitz.
\end{proof}

\begin{rmk}
The key difference between Lusin's Theorem itself and Lemma~\ref{lem:Lusin1} is that $\psi$ is almost Lipschitz on an \emph{open} set of large measure.  This tweak will be important for some applications of the Portmanteau Theorem later. \fin
\end{rmk}

\subsection{Covering and packing numbers}

If $(X,d)$ is a metric space and $\delta > 0$, then a subset $F \subseteq X$ is \textbf{$\delta$-separated} if any distinct $x,y \in F$ satisfy $d(x,y) \geq \delta$. The \textbf{$\delta$-covering} and \textbf{$\delta$-packing numbers} of the space are defined by
\[\cov_\delta(X,d) := \min\{|F|:\ B_\delta(F) = X\}\]
and
\[\pack_\delta(X,d) := \max\big\{|F|:\ F\ \hbox{is $\delta$-separated in}\ (X,d)\big\},\]
where either value may be $+\infty$. More generally, if $Y \subseteq X$ then we abbreviate
\[\cov_\delta\big(Y,\,d|Y\times Y\big) =: \cov_\delta(Y,d) \quad \hbox{and} \quad \pack_\delta\big(Y,\,d|Y\times Y\big) =: \pack_\delta(Y,d).\]
These definitions lead quickly to the standard inequalities
\begin{equation}\label{eq:pre-cov-and-pack}
\cov_{\delta/2}(Y,d) \geq \pack_\delta(Y,d) \geq \cov_\delta(Y,d) \quad \forall \delta > 0.
\end{equation}

Now suppose in addition that $\mu \in \Pr(X)$.  For $\eps,\delta > 0$, the \textbf{$(\eps,\delta)$-covering number of $\mu$ according to $d$} is
\[\cov_{\eps,\delta}(\mu,d) := \min\big\{\cov_\delta(V,d):\ V \subseteq X\ \hbox{such that}\ \mu(V) > 1 - \eps\big\}.\]
Similarly, the \textbf{$(\eps,\delta)$-packing number of $\mu$ according to $d$} is
\[\pack_{\eps,\delta}(\mu,d) := \min\big\{\pack_\delta(V,d):\ V \subseteq X\ \hbox{such that}\ \mu(V) > 1 - \eps\big\}.\]

From these definitions, the inequalities~(\ref{eq:pre-cov-and-pack}) translate immediately into
\begin{equation}\label{eq:cov-and-pack}
\cov_{\eps,\delta/2}(\mu,d) \geq \pack_{\eps,\delta}(\mu,d) \geq \cov_{\eps,\delta}(\mu,d) \quad \forall \eps,\delta > 0.
\end{equation}

Now let $(X,d_X)$ and $(Y,d_Y)$ be metric spaces and let $\mu$ and $\nu$ be Borel probabilities on $X$ and $Y$ respectively.  Assume that $X$ and $Y$ are separable, so that the Borel $\s$-algebra of $X\times Y$ agrees with the product of their separate Borel $\s$-algebras.  Let $d$ be the \textbf{Hamming average} on $X\times Y$ of the metrics $d_X$ and $d_Y$:
\[d\big((x,y),(x',y')\big) := \frac{1}{2}d_X(x,x') + \frac{1}{2}d_Y(y,y').\]
The following is quite close to standard results for covering and packing numbers, but the switch to the setting of probability measures requires a little extra work.

\begin{lem}\label{lem:sum-cov}
The following hold for all $\eps,\delta > 0$:
\begin{itemize}
\item[i)] any coupling $\l$ of $\mu$ and $\nu$ satisfies
\[\cov_{\eps,\delta}(\l,d) \leq \cov_{\eps/2,\delta}(\mu,d_X) \cdot \cov_{\eps/2,\delta}(\nu,d_Y);\]
\item[ii)] the product measure satisfies
\[\pack_{\eps,\delta/2}(\mu\times \nu,d) \geq \pack_{\sqrt{\eps},\delta}(\mu,d_X) \cdot \pack_{\sqrt{\eps},\delta}(\nu,d_Y).\]
\end{itemize}
\end{lem}

\begin{proof}\emph{(i).}\quad If $E\subseteq X$ and $F\subseteq Y$ are such that $\mu(B_\delta(E)) > 1 - \eps/2$ and $\nu(B_\delta(F)) > 1 - \eps/2$ (where these neighbourhoods are taken according to the metrics $d_X$ and $d_Y$, respectively), then $B_\delta(E\times F) \supseteq B_\delta(E)\times B_\delta(F)$, and so
\begin{align*}
\l\big((X\times Y)\setminus B_\delta(E\times F)\big) &\leq \l\big((X\times Y)\setminus (B_\delta(E)\times B_\delta(F))\big)\\ &\leq \l\big((X\setminus B_\delta(E))\times Y\big) + \l\big(X\times (Y\setminus B_\delta(F))\big)\\ &= \mu(X\setminus B_\delta(E)) + \nu(Y\setminus B_\delta(F))\\ &< 2(\eps/2) = \eps.
\end{align*}

\vspace{7pt}

\emph{(ii).}\quad Let $m := \pack_{\sqrt{\eps},\delta}(\mu,d_X)$ and $n := \pack_{\sqrt{\eps},\delta}(\nu,d_Y)$.

Suppose that $V \subseteq X\times Y$ has $(\mu\times \nu)(V) > 1 - \eps$.  For each $x \in X$ let
\[V_x := \{y \in Y:\ (x,y) \in V\}.\]
Fubini's Theorem and Chebyshev's Inequality imply that the set
\[U := \{x:\ \nu(V_x) > 1 - \sqrt{\eps}\}\]
has $\mu(U) > 1 - \sqrt{\eps}$.  The latter inequality implies that there is some $F\subseteq U$ with cardinality $m$ and which is $\delta$-separated according to $d_X$.  On the other hand, for each $x \in F$, the inequality $\nu(V_x) > 1 - \sqrt{\eps}$ implies that there is some $E_x \subseteq V_x$ with cardinality $n$ and which is $\delta$-separated according to $d_Y$.  Now the set of pairs
\[\{(x,y):\ x \in F\ \hbox{and}\ y \in E_x\}\]
is contained in $V$, has cardinality $nm$, and is $(\delta/2)$-separated according to $d$.
\end{proof}

Combining Lemma~\ref{lem:sum-cov} with the inequalities in~(\ref{eq:cov-and-pack}) gives the following.

\begin{cor}\label{cor:sum-cov}
For all $\eps,\delta > 0$ one has
\[\cov_{\eps,\delta/4}(\mu\times \nu,d) \geq \cov_{\sqrt{\eps},\delta}(\mu,d_X) \cdot \cov_{\sqrt{\eps},\delta}(\nu,d_Y).\]
\qed
\end{cor}

The next lemma gives some control over the covering numbers of the pushforward of a measure under an almost Lipschitz map.  After that, Lemma~\ref{lem:comparing-nearby-covnos} shows that the covering numbers of pushforward measures are somewhat stable under small changes to the maps.  These results will be used for the proofs that our new entropy-notions are isomorphism-invariant in Subsection~\ref{subs:iso-invar}.

\begin{lem}\label{lem:img-covnos}
Let $(X,d_X)$ and $(Y,d_Y)$ be metric spaces, let $\mu \in \Pr(X)$, and let $\phi:X\to Y$ be $\eta$-almost $L$-Lipschitz.  Then
\[\cov_{\eps,\eta + L\delta}(\phi_\ast\mu,d_Y) \leq \cov_{\eps,\delta}(\mu,d_X).\]
\end{lem}

\begin{proof}
Choose $F \subseteq X$ such that $\mu(B_\delta(F)) > 1 - \eps$.  Since $\phi$ is $\eta$-almost $L$-Lipchitz, we have
\[B_{\eta + L\delta}(\phi(F)) \supseteq \phi(B_\delta(F)),\]
and hence
\[\phi^{-1}\big(B_{\eta + L\delta}(\phi(F))\big) \supseteq B_\delta(F).\]
Therefore
\[(\phi_\ast\mu)\big(B_{\eta + L\delta}(\phi(F))\big) = \mu\big(\phi^{-1}\big(B_{\eta + L\delta}(\phi(F))\big)\big) \geq \mu(B_\delta(F)) > 1 - \eps.\]
\end{proof}

\begin{lem}\label{lem:comparing-nearby-covnos}
Let $(X,\mu)$ be a probability space, let $(Y,d)$ be a separable metric space, and let $\phi,\psi:X\to Y$ be measurable functions.  Also let $\eps,\eps',\delta > 0$.  If
\[\mu\{d(\phi(\cdot),\psi(\cdot)) > \delta/2\} < \eps'\]
then
\[\cov_{\eps + \eps',\delta}(\phi_\ast\mu,d) \leq \cov_{\eps,\delta/2}(\psi_\ast\mu,d).\]
\end{lem}

As before, the separability of $Y$ ensures that the Borel $\s$-algebra of $Y\times Y$ agrees with the product $\s$-algebra, and hence that $d$ is measurable with respect to the latter.  This is needed for the first displayed inequality above to make sense.

\begin{proof}
Let $U := \{d(\phi(\cdot),\psi(\cdot)) \leq \delta/2\}$, and let $F \subseteq Y$ be such that
\[\mu\{x:\ \psi(x) \in B_{\delta/2}(F)\} = \psi_\ast\mu(B_{\delta/2}(F)) > 1 - \eps.\]
Then
\begin{multline*}
\phi_\ast\mu(B_\delta(F)) \geq \mu\{x:\ x \in U\ \hbox{and}\ \phi(x) \in B_\delta(F)\}\\
\geq \mu\{x:\ x \in U\ \hbox{and}\ \psi(x) \in B_{\delta/2}(F)\} > 1 - \eps - \eps'.
\end{multline*}
\end{proof}

\section{Empirical distributions, good models, and sofic entropy}\label{sec:model-spaces}

\subsection{Definitions}\label{subs:defs}

Suppose that $G$ is a countable discrete group.  Suppose further that $V$ is a finite set and that $\s:G\to \Sym(V)$ is any map.  Think of this $\s$ as an `attempt' at a representation of $G$ by permutations of $V$.  Given $g,h \in G$ and $v \in V$, it may not be the case that
\begin{equation}\label{eq:gp-action}
\s^g(\s^h(v)) = \s^{gh}(v).
\end{equation}
The `quality' of $\s$ as an attempt at a representation can be quantified by the number of $v$ at which~(\ref{eq:gp-action}) holds, say for some finite list of groups elements $g$, $h$ of interest.

A \textbf{sofic approximation} to $G$ is a sequence of finite sets $V_n$ and maps
\[\s_n:G\to \Sym(V_n), \quad n\geq 1,\]
such that
\begin{equation}\label{eq:sofic1}
\big[\quad \s_n^g(\s_n^h(v)) = \s_n^{gh}(v) \quad \hbox{w.h.p. in}\ v \quad \big] \quad \forall g,h \in G
\end{equation}
and
\begin{equation}\label{eq:sofic2}
\big[ \quad \s_n^g(v) \neq v \quad \hbox{w.h.p. in}\ v \quad \big]\quad \forall g \in G\setminus \{e_G\}, 
\end{equation}
both as $n\to\infty$.  Note the order of the quantifiers: we certainly do \emph{not} ask that
\[[\ \s_n^g(\s_n^h(v)) = \s_n^{gh}(v) \quad \forall g,h \in G\ ] \quad \hbox{w.h.p. in}\ v.\]
The group itself is \textbf{sofic} if it has a sofic approximation.  This definition essentially follows Weiss~\cite{Weiss00} and (with a different nomenclature) Gromov~\cite{Gro99}.

Now consider also a compact metric space $(\X,d)$, and let $V$ and $\s$ be as above.  Elements of $\X^V$ will be denoted by boldface letters, to distinguish them from elements of shift-spaces such as $\X^G$.  For $\bf{x} = (x_v)_{v \in V}\in \X^V$ and $v \in V$, the \textbf{pullback name of $\bf{x}$ at $v$} is defined by
\[\Pi^\s_v(\bf{x}) := (x_{\s^g(v)})_{g \in G} \in \X^G.\]
This defines a map $\Pi^\s_v:\X^V\to \X^G$ for each $v$.  Let $S$ be the right-shift action of $G$ on $\X^G$. Properties~(\ref{eq:sofic1}) and~(\ref{eq:sofic2}) have the following simple but important consequence.

\begin{lem}\label{lem:approx-equiv}
If $(\s_n)_{n\geq 1}$ is a sofic approximation to $G$, $F\subseteq G$ is finite, and $g \in G$, then the following holds w.h.p. in $v \in V_n$:
\[\Pi^{\s_n}_{\s_n^g(v)}(\cdot)|_F = (S^g(\Pi^{\s_n}_v(\cdot)))|_F.\]
\end{lem}

More fully, this conclusion asserts that
\[[\quad \Pi^{\s_n}_{\s_n^g(v)}(\bf{x})|_F = \Pi^{\s_n}_v(\bf{x})|_{Fg} \quad \forall \bf{x} \in \X^{V_n}\quad ] \quad \hbox{w.h.p. in}\ v,\]
where we identify elements of $\X^{Fg}$ with elements of $\X^F$ in the obvious way.

\begin{proof}
It suffices to prove this when $F$ is an arbitrary singleton, say $\{h\}$.  Then it holds w.h.p. in $v$ that
\[\s_n^h(\s_n^g(v)) = \s_n^{hg}(v).\]
If $v$ satisfies this, and $\bf{x} \in \X^{V_n}$, then
\[\big(\Pi^{\s_n}_{\s_n^g(v)}(\bf{x})\big)_h = x_{\s_n^h(\s_n^g(v))} = x_{\s_n^{hg}(v)} = \big(\Pi^{\s_n}_v(\bf{x})\big)_{hg}.\]
\end{proof}

\begin{rmk}
The map $\s$ also gives rise to an `adjoint' map $\rho:G\to \Sym(\X^V)$:
\[\rho^g((x_v)_{v \in V}) := (x_{\s^{g^{-1}}(v)})_{v\in V}.\]
Similarly to the proof above, one can show that, if $( \s_n)_{n\geq 1}$ is a sofic approximation, then for most $v$ the pullback-name map $\Pi^{\s_n}_v$ approximately intertwines $\rho_n^g$ with the left-shift action $\t{S}$ of $G$ on $\X^G$, defined by
\[\t{S}^g((x_h)_{h \in G}) := (x_{g^{-1}h})_{h\in G}.\]
This observation does not seem to be useful unless our measure $\mu$ on $\X^G$ is invariant under $\t{S}$ as well as $S$.  In that case, each $\t{S}^g$ defines an isomorphism from the system $(\X^G,\mu,S)$ to itself, and the above relationship to $\rho_n^g$ becomes a special case of a result for general factor maps: see Lemma~\ref{lem:psi-sig-compatible} below. \fin
\end{rmk}

Each $\bf{x} \in \X^{V_n}$ has an associated probability measure on $\X^G$ called its \textbf{empirical distribution}:
\begin{eqnarray}\label{eq:emp}
P^\s_\bf{x} := \frac{1}{|V|}\sum_{v \in V}\delta_{\Pi^\s_v(\bf{x})}.
\end{eqnarray}

Lemma~\ref{lem:approx-equiv} will mostly be used through the following consequence, which asserts an approximate invariance for empirical distributions.

\begin{lem}\label{lem:approx-invar}
Let $F\subseteq G$ be finite and $g \in G$.  Then
\[\sup_{\bf{x} \in \X^{V_n}}\|(P^{\s_n}_\bf{x})_F - (S^g_\ast P^{\s_n}_\bf{x})_F\|_{\rm{TV}} \to 0 \quad \hbox{as} \ n\to\infty.\]
\end{lem}

\begin{proof}
Lemma~\ref{lem:approx-equiv} gives that
\[(S^g\Pi^{\s_n}_v(\cdot))|_F = \Pi^{\s_n}_{\s_n^gv}(\cdot)|_F \quad \hbox{w.h.p. in}\ v,\]
and therefore
\begin{eqnarray*}
(P^{\s_n}_\bf{x})_F - (S^g_\ast P^{\s_n}_\bf{x})_F &=& \frac{1}{|V_n|}\sum_{v \in V_n}(\delta_{\Pi^{\s_n}_v(\bf{x})|_F} - \delta_{(S^g\Pi^{\s_n}_v(\bf{x}))|_F})\\
&=& \frac{1}{|V_n|}\sum_{v \in V_n}(\delta_{\Pi^{\s_n}_v(\bf{x})|_F} - \delta_{\Pi^{\s_n}_{\s_n^gv}(\bf{x})|_F}) + o(1).
\end{eqnarray*}
But the sum on this last line vanishes, because $\s_n^g$ is a permutation of $V_n$.
\end{proof}

Now suppose that $\mu$ is a right-shift-invariant probability measure on $\X^G$.  For any $F \subseteq G$, let $\mu_F$ denote the marginal of $\mu$ on $X^F$, as previously.

For $G$ and $\s$ as above, and for any $\calO \subseteq \Pr(\X^G)$, let
\begin{eqnarray}\label{eq:basic-neigh}
\Omega(\calO,\s) := \big\{\bf{x} \in \X^V:\ P_{\bf{x}}^\s \in \calO \big\}.
\end{eqnarray}
In case $\calO$ is a w$^\ast$-neighbourhood of $\mu$, elements of $\O(\calO,\s)$ are called \textbf{$\calO$-good models} for $\mu$ over $\s$.

This definition is easiest to visualize in case $\calO$ is of the form $\{\nu:\ \nu_F \in \calO'\}$ for some finite $F\subseteq G$ and some w$^\ast$-neighbourhood $\calO'$ of $\mu_F$.  Neighbourhoods $\calO$ of this form are a basis of neighbourhoods around $\mu$, so this assumption does not lose much generality.  For this $\calO$ we can write
\[\O(\calO,\s_n) = \Big\{\bf{x} \in \X^{V_n}:\ \frac{1}{|V_n|}\sum_{v \in V_n}\delta_{\Pi^{\s_n}_v(\bf{x})|_F} \in \calO' \Big\}.\]
Given the finite set $F$, it follows from conditions~(\ref{eq:sofic1}) and~(\ref{eq:sofic2}) that the mapping
\[g\mapsto \s_n^g(v)\]
defines a map from $F$ to a corresponding `window' $\s_n^F(v)$ around $v$ in $V_n$, and this map is a bijection w.h.p. in $v$.  Consequently, one may regard the map $\bf{x} \mapsto \Pi^{\s_n}_v(\bf{x})|_F$ as restricting $\bf{x}$ to $\bf{x}|_{\s_n^F(v)}$, and then making a copy of this restriction indexed by $F$ itself.  The good models $\O(\calO,\s_n)$ are those $\bf{x}$ such that, on average over $v$, the frequency with which one sees a particular restriction through this window is close to the $F$-marginal of $\mu$ itself.  This is the sense in which $\bf{x}$ is `modeling' $\mu$.

In many examples $\X$ is a finite alphabet.  In that case one can instead work with total-variation neighbourhoods of finite-dimensional marginals of $\mu$. Specifically, if $\X$ is finite then the \textbf{$(F,\eps)$-good models} for $\mu$ over $\s$ are the elements of
\[\Omega_\mu(F,\eps,\s) := \Big\{\bf{x} \in \X^V:\ \Big\|\frac{1}{|V|}\sum_{v \in V}\delta_{\Pi^\s_v(\bf{x})|_F} - \mu_F \Big\|_{\rm{TV}} < \eps \Big\}.\]

Finally, as in the Introduction, we define the sofic entropy of a metric $G$-process $(\X^G,\mu,S,d)$ to be
\[\rmh_\S(\mu) := \sup_{\delta > 0}\ \inf_{\calO}\ \limsup_{n\to\infty}\frac{1}{|V_n|}\log \rm{cov}_\delta\big(\O(\calO,\s_n),d^{(V_n)}\big),\]
where $\calO$ ranges over w$^\ast$-neighbourhoods of $\mu$. Similarly, the lower sofic entropy is
\[\ul{\rmh}_\S(\mu) := \sup_{\delta > 0}\ \inf_{\calO}\ \liminf_{n\to\infty}\ \frac{1}{|V_n|}\log \rm{cov}_\delta\big(\O(\calO,\s_n),d^{(V_n)}\big).\]
We do not record $d$ in the notation for these quantities because it turns out that they depend only on the measure-theoretic structure of the process $(\X^G,\mu,S)$, as will be shown in the next subsection.

Examples in which $\ul{\rmh}_\S(\mu) < \rmh_\S(\mu)$ can be obtained from examples in which $\rmh_\S(\mu) \neq \rmh_{\S'}(\mu)$ for two different sofic approximations $\S$ and $\S'$ to $G$.  Interleaving $\S$ and $\S'$ into a single sofic approximation then gives the former inequality. For instance,~\cite[Subsection 8.3]{Bowen10b} includes examples of free-group Markov chains whose f-invariant is finite and negative.  Since (i) sofic entropies can take values only in $\{-\infty\} \cup [0,+\infty]$, and (ii) the f-invariant can be expressed as a kind of average of sofic entropies over randomly-chosen sofic approximations~\cite{Bowen10c}, it follows that these systems must have some sofic approximations which give non-negative real values for the sofic entropy, and others which give $-\infty$.

However, it is an important open problem whether one can obtain two different finite real values for $\rmh_\S(\mu)$ and $\rmh_{\S'}(\mu)$ using two different sofic approximations.

\subsection{Agreement with the Kerr-Li definition}\label{subs:KerrLi}

Bowen defined sofic entropy for systems with finite generating partitions in~\cite{Bowen10}.  Kerr and Li give a new definition in~\cite{KerLi11b} in terms of approximate homomorphism between commutative von Neumann algebras, and showed that it gives the same values as Bowen's if there is a finite generating partition.  Then, in~\cite[Section 3]{KerLi13}, Kerr and Li gave another, more elementary definition of general sofic entropy, and proved its equivalence to their previous definition.

In this subsection we show that our definition gives the same values as the entropy of~\cite[Definition 3.3]{KerLi13}.  We refer to that paper for a careful introduction to their definition.  If $(X,\mu,T)$ is a $G$-system, let us write $\t{\rmh}_\S(\mu,T)$ for the entropy defined there.

It suffices to show that any metric $G$-process $(\X^G,\mu,S,d)$ satisfies
\[\rmh_\S(\mu) = \t{\rmh}_\S(\mu,S).\]
That is, rather than analyze a $G$-action on an arbitrary probability space $(X,\mu)$, we may restrict our attention to the case of processes, so $X = \X^G$ for another standard measurable space $\X$ which is equipped with a particular compact generating metric $d$.  We may also assume that $d$ has diameter at most $1$.

On $X = \X^G$, define the pseudometric
\[\rho(x,y) = d(x_e,y_e) \quad \hbox{for}\ x = (x_g)_g,\ y = (y_g)_g \in \X^G.\]
Let $\S = (\s_n:G\to \rm{Sym}(V_n))_{n\geq 1}$ be the sofic approximation.  If $F$ is a finite subset of $G$, $L$ is a finite subset of $C(X)$, and $\eps > 0$, then define $\rm{Map}_\mu(\rho,F,L,\eps,\s_n)$ to be the set of those $x \in X^{V_n}$ such that
\begin{itemize}
 \item[i)] we have
\[\sqrt{\frac{1}{|V_n|}\sum_{v \in V_n}\rho(x_{\s_n^g(v)},S^g(x_v))^2} < \eps \quad \forall g \in F\]
\item[ii)] and
\[\frac{1}{V_n}\sum_{v \in V_n}f(x_v) \approx_\eps \int f\,\d\mu \quad \forall f\in L.\]
\end{itemize}

The pseudometric $\rho$ is clearly dynamically generating (see~\cite[Section 2]{KerLi13} or~\cite[Section 4]{Li12}).  Therefore~\cite[Proposition 3.4]{KerLi13} gives that
\begin{equation}\label{eq:KL-def}
\t{\rmh}_\S(\mu,S) = \sup_{\delta > 0}\inf_{F,L,\eps}\limsup_{n\to\infty}\frac{1}{|V_n|}\log \rm{pack}_\delta\big(\rm{Map}_\mu(\rho,F,L,\eps,\s_n),\rho^{2,(V_n)}\big),
\end{equation}
where $F$, $L$ and $\eps$ are as above, and $\rho^{2,(V_n)}$ is the $\ell_2$-analog of the Hamming-average metric on $X^{V_n}$:
\[\rho^{2,(V_n)}(x,y) := \sqrt{\frac{1}{|V_n|}\sum_{v \in V_n}\rho(x_v,y_v)^2} \quad \hbox{for}\ (x_v)_v,\ (y_v)_v \in X^{V_n}.\]
(I have adjusted some of the notation from~\cite{KerLi13} to match the present paper.)

\begin{prop}\label{prop:equals-KL}
 In the setting above we have $\rmh_\S(\mu) = \t{\rmh}_\S(\mu,S)$.
\end{prop}

\begin{proof}
\emph{Step 1.}\quad Let
\[\rho^{1,(V_n)}(x,y) := \frac{1}{|V_n|}\sum_{v \in V_n}\rho(x_v,y_v) \quad \hbox{for}\ (x_v)_v,\ (y_v)_v \in X^{V_n}.\]
Since $\rho$ is bounded by $1$, we have
\[\rho^{1,(V_n)} \leq \rho^{2,(V_n)} \leq \sqrt{\rho^{1,(V_n)}}.\]
Since the right-hand side of~(\ref{eq:KL-def}) takes a supremum over $\delta > 0$, we may therefore replace $\rho^{2,(V_n)}$ with $\rho^{1,(V_n)}$ in that equation.  Also, recalling the inequalities~(\ref{eq:pre-cov-and-pack}), the same reasoning lets us replace $\pack_\delta$ with $\cov_\delta$ in~(\ref{eq:KL-def}).

\vspace{7pt}

\emph{Step 2.}\quad For each $n$, define
\[\Phi_n:X^{V_n} \to \X^{V_n}:\big((x_{v,g})_{g\in G}\big)_{v\in V_n} \mapsto (x_{v,e})_{v \in V_n}.\]
This is an isometry from the pseudometric $\rho^{1,(V_n)}$ to the true metric $d^{(V_n)}$.

Now suppose that $\calO$ is a w$^\ast$-neighbourhood of $\mu$.  Then there are a finite $L \subseteq C(X)$, say consisting of $[0,1]$-valued functions, and $\eps_1 > 0$ such that
\[\calO \supseteq \calO' := \Big\{\nu \in \Pr(X):\ \int f\,\d\nu \approx_{3\eps_1} \int f\,\d\mu\ \forall f \in L\Big\}.\]
Since $L$ is finite, there are a finite set $F \subseteq G$ and an $\eps_2 > 0$ such that, for any $x,y \in X$, we have
\[\big[\ d(x_g,y_g) < \eps_2 \quad \forall g \in F\ \big] \quad \Longrightarrow \quad \big[\ f(x) \approx_{\eps_1} f(y) \quad \forall f\in L\ \big].\]

Given $\eps_2$, we may now choose $\eps \in (0,\eps_1)$ so small that the following holds.  If
\[x = \big((x_{v,g})_{g\in G}\big)_{v\in V_n} \in \rm{Map}_\mu(\rho,F,L,\eps,\s_n),\]
then condition (i) implies that
\[\big|\big\{v \in V_n:\ d(x_{\s_n^g(v),e},x_{v,g}) < \eps_2 \ \forall g \in F\big\}\big| > (1 - \eps_1)|V_n|.\]
It follows that
\[\int f\,\d P^{\s_n}_{\Phi_n(x)} = \frac{1}{|V_n|}\sum_{v \in V_n}f\big((x_{\s_n^g(v),e})_{g\in G}\big) \approx_{2\eps_1} \frac{1}{|V_n|}\sum_{v \in V_n}f\big((x_{v,g})_{g\in G}\big),\]
and now condition (ii) gives that this is within $\eps_1$ of $\int f\,\d\mu$.  This shows that
\[\Phi_n\big(\rm{Map}_\mu(\rho,F,L,\eps,\s_n)\big) \subseteq \O(\calO',\s_n) \subseteq \O(\calO,\s_n).\]
Since $\Phi_n$ is an isometry, it preserves covering numbers.  Taking infima over $\calO$ (on the right-hand side) or $F$, $L$ and $\eps$ (on the left-hand side), this shows that
\[\t{\rmh}_\S(\mu,S) \leq \rmh_\S(\mu).\]

\vspace{7pt}

\emph{Step 3.}\quad The proof of the reverse inequality is very similar.  Now we define
\[\Psi_n:\X^{V_n}\to X^{V_n}:\bf{x} \mapsto \big(\Pi^{\s_n}_v(\bf{x})\big)_{v\in V_n},\]
which is an isometry from $d^{(V_n)}$ to $\rho^{1,(V_n)}$.  For any $F$, $L$, and $\eps > 0$, there is a w$^\ast$-neighbourhood $\calO$ of $\mu$ such that
\[\Psi_n\big(\O(\calO,\s_n)\big) \subseteq \rm{Map}_\mu(\rho,F,L,\eps,\s_n).\]
This $\calO$ can be obtained by reversing the construction in Step 2.  The only point worth remarking is that condition (ii) in the definition of $\rm{Map}_\mu(\rho,F,L,\eps,\s_n)$ is obtained as a consequence of Lemma~\ref{lem:approx-equiv}.

Taking infima, this leads to
\[\rmh_\S(\mu) \leq \t{\rmh}_\S(\mu,S).\]
\end{proof}

Since Kerr and Li have already proved isomorphism-invariance of sofic entropy using their definitions (see~\cite{KerLi11b}), that property follows for ours.  This justifies the omission of $d$ from the notation $\rmh_\S(\mu)$. A proof of this invariance using our definition would be similar to the proof that model-measure sofic entropy is isomorphism-invariant, which is given in Subsection~\ref{subs:iso-invar} below.

\begin{rmk}
The proof of Proposition~\ref{prop:equals-KL} uses a particular generating pseudometric on $X = \X^G$, obtained from a metric on the single coordinate-space $\X$.  However, it turns out that for any abstract $G$-system $(X,\mu,T)$, any totally bounded and Borel measurable pseudometric $\rho$ on $X$ arises in this way, up to isomorphism.  Indeed, letting $(\X_0,d_0)$ be the quotient of $X$ by the relation $\{\rho = 0\}$, and letting $(\X,d)$ be the completion of $(\X_0,d_0)$, it is easy to check that the quotient map $q:X\to \X$ is Borel, and now this can be extended to a factor map $q^G:X\to \X^G$ similarly to~(\ref{eq:system-process}).  If $\rho$ is dynamically generating, then $q^G$ is an isomorphism, and $\rho$ is now obtained by pulling back the identity-coordinate metric $d$ through this isomorphism. \fin
\end{rmk}

\subsection{Subadditivity and failure of additivity}\label{subs:subadd}

\begin{prop}\label{prop:subadd}
Let $(\X^G,\mu,S)$ and $(\Y^G,\nu,S)$ be two $G$-processes, and let $\l$ be a joining of them.  Then
\[\rmh_\S(\l) \leq \rmh_\S(\mu) + \rmh_\S(\nu).\]
\end{prop}

\begin{proof}
Let $d_\X$ and $d_\Y$ be compact generating metrics on $\X$ and $\Y$ respectively, and let $d$ be their Hamming average on $\X\times \Y$.  All subsequent topologies are those determined by these metrics.

For any w$^\ast$-neighbourhood $\calO_1$ of $\mu$ and $\calO_2$ of $\nu$, there is a w$^\ast$-neighbourhood $\calN$ of $\l$ such that every $\theta \in \calN$ has first marginal in $\calO_1$ and second marginal in $\calO_2$.  This implies that
\[\O(\calN,\s_n) \subseteq \O(\calO_1,\s_n)\times \O(\calO_2,\s_n) \quad \forall n\geq 1,\]
and now the inequality
\[\cov_\delta\big(\O(\calN,\s_n),d^{(V_n)}\big) \leq \cov_\delta\big(\O(\calO_1,\s_n),d_\X^{(V_n)}\big)\cdot \cov_\delta\big(\O(\calO_2,\s_n),d_\Y^{(V_n)}\big)\]
completes the proof.
\end{proof}

In particular,
\[\rmh_\S(\mu\times \nu) \leq \rmh_\S(\mu) + \rmh_\S(\nu).\]
We now describe examples in which this inequality is strict.

\begin{ex}\label{ex:coind-ex}
Some standard probabilistic estimates are required to justify this example carefully, but we omit these for brevity.

Let $H = \langle a,b\rangle$ be the free group on two generators, let $H' = \langle a',b'\rangle$ be a copy of $H$, and let $G = H\ast H'$.  Then $G$ is a free group on four generators, and we may regard $H$ as a subgroup of $G$.

Let $T_0$ be the trivial $H$-action on the set $\X = \{0,1\}$, and endow this set with its discrete metric.  Let
\[\mu_0 := \frac{3}{4}\delta_0 + \frac{1}{4}\delta_1 \in \Pr(\X).\]
(It will be clear in what follows that `$\frac{3}{4}$' could be replaced with any value in $(1/2,1)$.)

Now co-induce $T_0$ to the $G$-action $\rm{CInd}_H^G T_0$ on the space
\[(X,\mu) := (\X^{H\backslash G},\mu_0^{\times H\backslash G}).\]
This co-induced system is isomorphic to a $G$-process $(\X^G,\nu,S)$ where $\nu$ is defined by the following three properties:
\begin{itemize}
\item every one-dimensional marginal of $\nu$ equals $\mu_0$;
 \item if $Hg = Hg'$ for some $g,g' \in G$, then $x_g = x_{g'}$ for $\nu$-a.e. $x$;
\item if the cosets $Hg_1$, \dots, $Hg_k$ are distinct, and $x$ is drawn at random from $\nu$, then the coordinates $x_{g_1}$, \dots, $x_{g_k}$ are independent.
\end{itemize}
See~\cite[Subsection II.10.(G)]{Kec10} or~\cite{DooZha12} for the definition and basic properties of co-induction.

Now for each $n$ let $U_n := \{1,\dots,3n\}$, $W_n := \{3n+1,\dots,4n\}$ and $V_n := U_n \cup W_n$. Choose four elements of $\Sym(V_n)$ in the following randomized way:
\begin{itemize}
 \item Let $\s_n^{a'}$ and $\s_n^{b'}$ be independent, uniformly random elements of $\Sym(V_n)$.
\item Let $\tau_{n,0}^a$ and $\tau_{n,0}^b$ be uniformly random elements of $\Sym(U_n)$ and let $\tau_{n,1}^a$ and $\tau_{n,1}^b$ be uniformly random elements of $\Sym(W_n)$, all independent.  Let $\s_n^a := \tau_{n,0}^a \cup \tau_{n,1}^a$ and $\s_n^b := \tau_{n,0}^b \cup \tau_{n,1}^b$.  Thus, $\s_n^a$ and $\s_n^b$ are chosen uniformly and independently from among those elements of $\Sym(V_n)$ that preserve the partition $\{U_n,W_n\}$.
\end{itemize}
For each $n$, the four permutations $\s_n^a$, $\s_n^b$, $\s_n^{a'}$ and $\s_n^{b'}$ generate a random homomorphisms $\s_n:G\to \Sym(V_n)$.  Standard arguments show that the resulting sequence $\S = (\s_n)_{n\geq 1}$ is a sofic approximation to $G$ with high probability.

For each $n$, let $\bf{x}_n \in \X^{V_n}$ be the indicator function $1_{W_n}$.  It is now easily checked that $P^{\s_n}_{\bf{x}_n} \stackrel{\rm{weak}^\ast}{\to} \nu$, and so for every w$^\ast$-neighbourhood $\calO$ of $\nu$ we have $\O(\calO,\s_n) \neq \emptyset$ for all sufficiently large $n$.  Therefore $\rmh_\S(\nu,S) \geq 0$.

The random Schreier graph on $V_n$ generated by the random permutations $\s_n^a$ and $\s_n^b$ is an expander within each of $U_n$ and $W_n$ with high probability; this follows by the usual counting argument (see, for instance,~\cite[Proposition 1.2.1]{Lubot--book}).  In this case, any other partition of $V_n$ with small edge boundary in this Schreier graph must be very close to the partition $\{U_n,W_n\}$.  This implies that any other good model $\bf{y} \in \X^{V_n}$ of $\nu$ must be very close to $\bf{x}_n$ in normalized Hamming distance as $n\to\infty$.  It follows that in fact $\rmh_\S(\nu,S) = 0$.

(Note that at this point, we are using the fact that the atom-sizes of the measure $\mu_0$ correspond to the ratios $|U_n|/|V_n|$ and $|W_n|/|V_n|$.  If these were all equal to $1/2$, instead of $3/4$ and $1/4$, then both $\bf{x}_n$ and $1_{U_n}$ would be good models of $\nu$, and all other good models would lie close to one of these two in normalized Hamming distance.  In this case, the rest of the argument below can still be completed, but there would be slightly more work to do.)

However, we can now show that $\rmh_\S(\nu\times \nu,S\times S) = -\infty$.  If $(\bf{y},\bf{y'}) \in (\X\times \X)^{V_n}$ were a good model for $\nu\times \nu$, then both $\bf{y}$ and $\bf{y'}$ would be good models for $\nu$, hence close to $\bf{x}_n$ in Hamming distance.  But this would imply that
\[P^{\s_n}_{(\bf{y},\bf{y}')}\{(1,0)\} = \frac{1}{4n}|\{v \in V_n:\ y_v = 1\ \hbox{and}\ y'_v = 0\}|\]
is close to $0$, whereas a good model for $\nu\times \nu$ should have this probability close to $\frac{1}{4}\cdot \frac{3}{4}$. So $\S$ does not provide arbitrarily good models for $\nu\times \nu$ as $n\to\infty$.

Note that this argument is really only about the $H$-subaction of $S$.  The only reason to co-induce to $G$ is to make an example which is ergodic overall and free.  It would be interesting to know whether one can produce such an example which is totally ergodic by starting with a more subtle choice of $H$-system.

If one replaces each $\S$ with $\S' := (\s_n^{\times k}:G\to V_n^{\times k})_{n\geq 1}$ for some fixed $k \geq 1$, then similar reasoning shows that
\[\rmh_\S(\mu^{\times \ell},S^{\times \ell}) = \left\{\begin{array}{ll}0& \quad \hbox{for}\ \ell \leq k\\ -\infty& \quad \hbox{for}\ \ell > k. \end{array}\right.\]
\fin
\end{ex}

\begin{ex}\label{ex:planted-bis}
Let us speculate about a second example.  The details required for its analysis are not available in full, but it would arguably be more natural than Example~\ref{ex:coind-ex}.

Let $H = \langle a,b\rangle$ be the free group and let $\X = \{0,1\}$, as above.  We start by constructing some finite quotients of $H$ as a variant of the `planted bisection model'.  This classical model has a long history in statistical physics and computer science: see~\cite{MosNeeSly15} for its definition and some references.

Let $\a \in (0,1)$ be a small parameter.  Let $V_n = U_n \cup W_n$ be as in Example~\ref{ex:coind-ex}, but now construct $\s_n^a,\s_n^b \in \Sym(V_n)$ as follows.

First let $\G_n$ be a random $4$-regular graph on $V_n$ drawn uniformly from those graphs that have roughly $6n(1-\a)$ edges within $U_n$, $2n(1-\a)$ edges within $W_n$, and $8n\a$ edges between $U_n$ and $W_n$.

Using this random graph $\G_n$, one can construct the pair of permutations $\s_n^a$ and $\s_n^b$ as follows. First, a simple greedy algorithm finds a disjoint union of cycles in $\G_n$ that contains all vertices in $V_n$.  Choose an orientation for each of these cycles, and let those directed edges define the permutation $\s_n^a$.  After removing these edges from $\G_n$, the remaining $2$-regular graph decomposes into another disjoint union of cycles; orienting those gives the permutation $\s_n^b$.  Let $\s_n:G\to \Sym(V_n)$ be the homomorphism generated by $\s_n^a$ and $\s_n^b$.  Now $\G_n$ is the Schreier graph of the homomorphism $\s_n$ and generating set $\{a^{\pm 1},b^{\pm 1}\}$.

As with other simple random-graph models, it should hold that $\G_n$ looks like a tree in a large neighbourhood around most points of $V_n$, and this would imply that $\S = (\s_n)_{n\geq 1}$ is a sofic approximation to $H$.

Finally, let $\bf{x}_n = 1_{W_n}$ as in Example~\ref{ex:coind-ex}, and now use w$^\ast$-compactness to choose a subsequence $n_1  < n_2 < \dots$ such that $P_{\bf{x}_{n_i}}^{\s_{n_i}} \stackrel{\rm{weak}^\ast}{\to} \mu$ for some $\mu \in \Pr(\X^H)$.  Of course, this guarantees that for any w$^\ast$-neighbourhood $\calO$ of $\mu$ we have $\bf{x}_{n_i} \in \O(\calO,\s_{n_i})$ for all sufficiently large $i$, and so $\rmh_\S(\mu) \geq 0$.

On the other hand, our intuition is that, if $\a$ is extremely small, so the graph $\G_n$ has sufficiently few of its edges crossing from $U_n$ to $W_n$, then any other partition of $V_n$ into subsets of sizes roughly $3n$ and $n$ and with so few edges between must be very close to $\{U_n,W_n\}$ (up to an error depending on $\a$).  Some hope for a proof of this is offered by the recent work~\cite{MosNeeSly15} on the original planted bisection model, which shows that if the two edge-densities in the model are sufficiently well-separated, then one can reconstruct the values of those edge-densities with high probability if one is given only the output graph $\G_n$.

If this prediction is correct, then the same argument as for Example~\ref{ex:coind-ex} will show that $\rmh_\S(\mu\times \mu) = -\infty$. \fin
\end{ex}

\section{Factor maps and maps between model spaces}\label{sec:factors}

Before introducing measures on model spaces, we need to consider how a factor map between systems can be approximated by a sequence of somewhat `regular' maps between their model spaces.

This section is rather technical, but it lays essential foundations for many of the arguments that follow.  In particular, it is the basis for the proof that $\rmh_\S^\rm{q}$ and $\rmh^\rm{dq}_\S$ are isomorphism-invariant.

\subsection{Approximating factor maps}

Suppose that $\phi: \X^G \to \Y$ is measurable, so it gives rise to the measurable equivariant map
\[\Phi := (\phi \circ S^g)_{g \in G}: \X^G\to \Y^G.\]
More generally, suppose that $E,F \subseteq G$ and that $\phi:\X^E\to \Y$.  Then one defines \[\phi^F:\X^{EF}\to \Y^F:(x_g)_{g \in EF} \mapsto \big(\phi((x_{hg})_{h\in E})\big)_{g \in F};\]
in this notation, $\Phi = \phi^G$.

In order to study such equivariant maps, we need the ability to approximate $\phi$ by a map which is `roughly continuous'.  This can be done in two steps.  The first is to replace $\phi$ with a map depending on only finitely many coordinates.

\begin{dfn}\label{dfn:local-fn}
If $\phi:\X^G\to \Y$ and $D \subseteq G$ is finite, then $\phi$ is \textbf{$D$-local} if it is measurable with respect to $\pi_D$.  A function is \textbf{local} if it is \textbf{$D$-local} for some $D$.

Similarly, a subset $U\subseteq \X^G$ is $D$-local if it equals $\pi_D^{-1}(V)$ for some $V \subseteq \X^D$.
\end{dfn}

If a function is described as `$D$-local', then it is always implied that $D$ is finite.

We now introduce a choice of metrics on $\X$ and $\Y$.  The next definition is a simple adaptation of Definition~\ref{dfn:eps-good}.

\begin{dfn}\label{dfn:eta-approx}
Let $(\X,d_\X)$ and $(\Y,d_\Y)$ be compact metric spaces, let $\mu \in \Pr^S(\X^G)$, let $\phi: \X^G\to \Y$ be a measurable function, and let $\eta > 0$.  An \textbf{$\eta$-almost Lipschitz} (or \textbf{$\eta$-AL}) \textbf{approximation to $\phi$ rel $(\mu,d_\X,d_\Y)$} is a measurable map $\psi:\X^G \to \Y$ with the following properties.
\begin{itemize}
\item[i)] The map $\psi$ approximates $\phi$ in the sense that
\begin{eqnarray}\label{eq:int-approx}
\int d_\Y(\phi(x),\psi(x))\,\mu(\d x) < \eta.
\end{eqnarray}
\item[ii)] There is a finite $D \subseteq G$ such that $\psi$ is $D$-local.
\item[iii)] There is a $D$-local open subset $U \subseteq \X^G$ such that $\mu(U) > 1 -\eta$ and such that $\psi|U$ is $\eta$-almost Lipschitz from $d_\X^{(D)}$ to $d_\Y$.
\end{itemize}
\end{dfn}

In this definition, since $\psi$ and $U$ are both $D$ local, we may regard $\psi|U$ as a function on $\X^D$.  Part (iii) of the definition can be understood this way, or by considering almost Lipschitz functions with respect to the pseudometric $d_\X^{(D)}$ on $\X^G$.

Definition~\ref{dfn:eta-approx} really does depend on the measure $\mu$ and on the specific metrics $d_\X$ and $d_\Y$.  However, we may sometimes drop the qualifier `rel $(\mu,d_\X,d_\Y)$' when these data are clear from the context.

Formally, the $\psi$ in this definition is a $D$-local function on the whole space $\X^G$.  We sometimes commit the abuse of writing $\psi(x|_D)$ in place of $\psi(x)$ when the local nature of the function is important.

\begin{lem}\label{lem:approx-by-Lip}
There exist $\eta$-AL approximations to $\phi$ for all $\eta > 0$.
\end{lem}

\begin{proof}
Let $\eta > 0$.  Standard measure theory gives some finite $D \subseteq G$ and a measurable function $\phi':\X^D\to \Y$ such that
\[\int d_\Y(\phi(x),\phi'(x|_D))\,\mu(\d x) < \eta.\]
Since $\eta$ is arbitrary, it now suffices to approximate $\phi'$ instead of $\phi$; or, equivalently, to assume that $\phi$ itself is a function on $\X^D$.  Having done so, we apply Lemma~\ref{lem:Lusin1} to this map and the metric spaces $(\X^D,d_\X^{(D)})$ and $(\Y,d_\Y)$.
\end{proof}

In case $\X$ is a finite set, all finite-dimensional Cartesian powers of $\X$ are finite, and so one can use a simplified form of Definition~\ref{dfn:eta-approx}. In that case, an $\eta$-AL approximation to $\phi:\X^G\to \Y$ is simply a local map $\psi: \X^G\to \Y$ such that $\mu\{\phi = \psi\} > 1 - \eta$.

It will be helpful to know that Definition~\ref{dfn:eta-approx} behaves well in relation to Hamming averages.  The next lemma describes this.

\begin{lem}\label{lem:good-approx-Ham-sum}
Suppose that $d_\Y$ has diameter at most $1$.  If $\psi$ is an $\eta$-AL approximation to $\phi$ rel $(\mu,d_\X,d_\Y)$ for some $\eta \in (0,1)$, then $\psi^F$ is a $(3\sqrt{\eta})$-AL approximation to $\phi^F:\X^G\to \Y^F$ rel $(\mu,d_\X,d_\Y^{(F)})$ for every finite $F\subseteq G$.
\end{lem}

\begin{proof}
Let $\psi$ be $D$-local, and in this proof let us regard $\psi$ as a function on $\X^D$ itself.  Let $U\subseteq \X^D$ be an open subset with $\mu_D(U) > 1 - \eta$ and such that $\psi|U$ is $\eta$-almost $L$-Lipschitz from $d_\X^{(D)}$ to $d_\Y$.

Firstly, the shift-invariance of $\mu$ and inequality~(\ref{eq:int-approx}) imply that
\begin{equation}\label{eq:int-approx2}
\int d_\Y^{(F)}\big(\phi^F(x),\psi^F(x)\big)\,\mu(\d x) = \frac{1}{|F|} \sum_{g \in F}\int d_\Y(\phi(S^gx),\psi(S^gx))\,\mu(\d x) < \eta.
\end{equation}

Next, it is clear that $\psi^F$ is $(DF)$-local.  Let
\[U_F := \big\{x\in \X^{DF}:\ |\{g \in F:\ x|_{Dg} \not\in U\}| < \sqrt{\eta}|F|\big\}.\]
This set is open, and
\[\int |\{g \in F:\ x|_{Dg} \not\in U\}|\,\mu(\d x) = \sum_{g \in F}\mu\{x:\ x|_{Dg} \not\in U\} = |F|\cdot \mu_D(\X^D\setminus U) < \eta|F|,\]
so Chebyshev's Inequality proves that $\mu_{DF}(U_F) > 1 - \sqrt{\eta}$.

Finally, if $x,x' \in U_F$, then
\begin{eqnarray*}
d_\Y^{(F)}\big(\psi^F(x),\psi^F(x')\big) &=& \frac{1}{|F|}\sum_{g\in F}d_\Y\big(\psi(x|_{Dg}),\psi(x'|_{Dg})\big)\\
&\leq& \frac{|\{g \in F:\ x|_{Dg} \not\in U\ \hbox{or}\ x'|_{Dg} \not\in U\}|}{|F|}\\
&& + \frac{1}{|F|}\sum_{\scriptsize{\begin{array}{c}g\in F,\\ \,x|_{Dg},x'|_{Dg} \in U\end{array}}}\big(\eta + Ld_\X^{(D)}(x|_{Dg},x'|_{Dg})\big) \\
&\leq& 3\sqrt{\eta} + L\frac{1}{|D||F|}\sum_{g \in F,\ h\in D}d_\X(x_{hg},x'_{hg}).
\end{eqnarray*}
Any point of $DF$ can be represented as a product $hg$ with $h \in D$ and $g \in G$ in no more than $|D|$ ways, and so the last line above is at most
\[3\sqrt{\eta} + L|D|d_\X^{(DF)}(x,x').\]
This shows that $\psi^F|U_F$ is $(3\sqrt{\eta})$-almost $(L|D|)$-Lipschitz.
\end{proof}

Sometimes it is preferable to use the approximation~(\ref{eq:int-approx2}) through its consequence that
\begin{eqnarray}\label{eq:nonint-approx}
\mu\big\{x:\ d_\Y^{(F)}\big(\phi^F(x),\psi^F(x)\big) < \sqrt{\eta}\big\} > 1 - \sqrt{\eta},
\end{eqnarray}
which follows by Chebyshev's inequality.

Now suppose that $\Phi = \phi^G:(\X^G,\mu,S) \to (\Y^G,\nu,S)$ is a factor map, and that $d_\X$ and $d_\Y$ are generating compact metrics on $\X$ and $\Y$ with diameter at most $1$.  If $\phi$ is not continuous for the resulting topologies on $\X^G$ and $\Y$, then $\Phi_\ast:\Pr(\X^G) \to \Pr(\Y^G)$ cannot be w$^\ast$-continuous.  However, the next lemma shows that, if $\eta$ is sufficiently small, then an $\eta$-AL approximation to $\phi$ rel $(\mu,d_\X,d_\Y)$ acts approximately continuously on measures which are w$^\ast$-close to $\mu$.  This fact will be used several times later.

\begin{lem}\label{lem:pushfwd-approx}
For every w$^\ast$-neighbourhood $\calN$ of $\nu$ there is an $\eta > 0$ with the following property.  For any $\eta$-AL approximation $\psi$ to $\phi$, there is a w$^\ast$-neighbourhood $\calO$ of $\mu$ such that
\[(\psi^G)_\ast(\calO) \subseteq \calN.\]
\end{lem}

\begin{rmk}
It is very important that $\eta$ can be chosen depending only on $\calN$. \fin
\end{rmk}

\begin{proof}
It suffices to prove this for any sub-basis of w$^\ast$-neighbourhoods of $\nu$, so we may assume that
\[\calN = \Big\{\theta \in \Pr(\Y^G):\ \int h\,\d\theta \approx_\k \int h\,\d\nu\Big\},\]
where $h:\Y^G\to [0,1]$ is $F$-local for some finite $F\subseteq G$ and is $1$-Lipschitz according to $d_\Y^{(F)}$, and $\k > 0$.

In this case we will show that any $\eta < (\k/18)^2$ has the required property.  Let $\psi$ be an $\eta$-AL approximation to $\phi$, and let $D \subseteq G$ and $U\subseteq \X^G$ be as in Definition~\ref{dfn:eta-approx} for this $\psi$.  Let $U_F$ be the $(DF)$-local open subset of $\X^G$ constructed in the proof of Lemma~\ref{lem:good-approx-Ham-sum}, so $\psi^F|U_F$ is $(3\sqrt{\eta})$-almost Lipschitz according to $d_\X^{(DF)}$.  Let us abbreviate $3\sqrt{\eta} =: \a$.

The composition $h\circ \psi^G$ is $(DF)$-local. We may regard $h$ as a map $\Y^F \to \bbR$ and $\psi^F$ as a map $\X^{DF} \to \Y^F$.  Applying Lemma~\ref{lem:almost-Lip-with-Lip}, it follows that the restriction $(h\circ\psi^F)|U_F$ is $\a$-almost Lipschitz according to $d_\X^{(DF)}$.  Therefore Lemma~\ref{lem:almost-Lip-near-Lip} gives a $(DF)$-local function $f:\X^G \to [0,1]$ which is truly Lipschitz according to $d^{(DF)}_\X$ and satisfies
\[\|(h\circ \psi^F - f)|U_F\|_\infty \leq \a.\]
Let
\[\calO := \Big\{\g \in \Pr(\X^G):\ \g(U_F) > 1 - \a\ \hbox{and}\ \int f\,\d\g \approx_\a \int f\,\d\mu\Big\}.\]

Suppose that $\g \in \calO$.  Then
\begin{multline*}
\Big|\int h\,\d(\psi^G_\ast\g) - \int h\,\d\nu\Big| = \Big|\int h\circ \psi^G\,\d\g - \int h\circ \phi^G\,\d\mu\Big|\\
\leq \Big|\int h\circ \psi^G\,\d\g - \int h\circ \psi^G\,\d\mu\Big| + \int |h(\psi^G(x)) - h(\phi^G(x))|\,\mu(\d x)
\end{multline*}
We now bound these two terms separately.  The first is at most
\begin{multline*}
\Big|\int f\,\d\g - \int f\,\d\mu\Big| + \Big|\int h\circ \psi^G\,\d\g - \int f\,\d\g\Big| + \Big|\int f\,\d\mu - \int h\circ \psi^G\,\d\mu\Big|\\
< \a + (\a + \g(\X^G\setminus U_F)) + (\a + \mu(\X^G\setminus U_F)) < 5\a,
\end{multline*}
using the definition of $\calO$. Since $h$ is $1$-Lipschitz according to $d_\Y^{(F)}$, the second term is at most
\[\int d^{(F)}_\Y(\psi^F(x),\phi^F(x))\,\mu(\d x) = \int d_\Y(\psi(x),\phi(x))\,\mu(\d x),\]
and this is at most $\a$ according to Definition~\ref{dfn:eps-good}.  Adding these estimates gives
\[\Big|\int h\,\d(\psi^G_\ast\g) - \int h\,\d\nu\Big| < 6\a < \k.\]
\end{proof}

\begin{rmk}
The importance of AL approximations to $\Phi = \phi^G$ is that we can control their interactions with the w$^\ast$-neighbourhoods that appear in the definition of good models.  Almost Lipschitz maps are not the only way to do this, but they are very convenient.  On the one hand, we cannot use truly continuous maps in general, since there are choices of $\X$ and $\Y$ for which there are not enough continuous maps $\X^G\to \Y$.  This is why we use maps $\psi$ for which continuity can fail, but only up to an additive error that we control.  On the other hand, in the next subsection we will use such approximants $\psi$ to construct a family of maps acting between model-spaces, and it will be important to exert some uniform control over the `approximate continuity' of all these other maps.  A simple way is to show that they are all $\eta'$-almost $L'$-Lipschitz for some common $\eta'$ and $L'$: see Lemma~\ref{lem:Lip-still-Lip} below. \fin
\end{rmk}

The following lemma and corollary give some simple ways of combining AL approximations.

\begin{lem}\label{lem:combine}
Let $\eta > 0$. Suppose that
\begin{itemize}
 \item $(\X_i^G,\mu_i,S,d_i)$ for $i=1,2$ are metric $G$-process,
\item $(\Y_i,d'_i)$ for $i=1,2$ are compact metric spaces,
\item $\phi_i:\X_i^G\to \Y_i$ for $i=1,2$ are measurable functions,
\item and $\psi_i:\X_i^G\to \Y_i$ is an $\eta$-AL approximation to $\phi_i$ rel $(\mu,d_i,d'_i)$ for each $i=1,2$.
\end{itemize}
Let $d$ be the Hamming average metric of $d_1$ and $d_2$ on $\X_1\times \X_2$, and similarly let $d'$ be the Hamming averagee of $d'_1$ and $d'_2$.  Finally, let $\l$ be any joining of $\mu_1$ and $\mu_2$. Then the map
\[\psi_1\times \psi_2:(\X_1\times \X_2)^G\to \Y_1\times \Y_2\]
is a $(2\eta)$-AL approximation to $\phi_1\times \phi_2$ rel $(\l,d,d')$.
\end{lem}

\begin{proof}
For $i=1,2$, let $D_i \subseteq G$ be the finite subsets and $U_i \subseteq \X_i^G$ the open subsets promised by Definition~\ref{dfn:eta-approx} for the maps $\psi_i$.  Then~(\ref{eq:int-approx}) and the definition of $d$ give
\begin{multline*}
\int d'\big((\phi_1(x_1),\phi_2(x_2)),(\psi_1(x_1),\psi_2(x_2))\big)\,\l(\d x_1,\d x_2)\\ = \frac{1}{2}\Big(\int d'_1(\phi_1(x_1),\psi_1(x_1))\,\mu_1(\d x_1) + \int d'_2(\phi_2(x_2),\psi_2(x_2))\,\mu_2(\d x_2)\Big) < \eta.
\end{multline*}
The map $\psi_1\times \psi_2$ is $(D_1 \cup D_2)$-local, and so is the open set $U:= U_1\times U_2$.  This set $U$ has
\[\l(U) \geq 1 - \mu_1(\X_1^G\setminus U_1) - \mu_2(\X_2^G\setminus U_2) > 1 - 2\eta,\]
and the definition of $d$ implies that $(\psi_1\times \psi_2)|U$ is $\eta$-almost Lipschitz.
\end{proof}

\begin{cor}\label{cor:combine}
Let $\eta > 0$. Suppose that
\begin{itemize}
 \item $(\X^G,\mu,S,d)$ is a metric $G$-process,
\item $(\Y_i,d'_i)$ for $i=1,2$ are compact metric spaces,
\item $\phi_i:\X^G\to \Y_i$ for $i=1,2$ are measurable functions,
\item and $\psi_i:\X^G\to \Y_i$ is an $\eta$-AL approximation to $\phi_i$ rel $(\mu,d,d'_i)$ for each $i=1,2$.
\end{itemize}
Let $d'$ be the Hamming average metric of $d'_1$ and $d'_2$ on $\Y_1\times \Y_2$.  Then the map
\[(\psi_1,\psi_2):\X^G\to \Y_1\times \Y_2\]
is a $(2\eta)$-AL approximation to $(\phi_1,\phi_2)$ rel $(\mu,d,d')$.
\end{cor}

\begin{proof}
This is the special case of Lemma~\ref{lem:combine} in which both processes $(\X_i^G,\mu_i,S,d_i)$ are equal to $(\X^G,\mu,S,d)$ and $\l$ is the diagonal joining.
\end{proof}

\subsection{Action of approximations on good models}\label{subs:action-of-approx}

Now suppose that $\psi:\X^G\to \Y$, that $V$ is a finite set and that $\s:G\to \Sym(V)$.  We define a new mapping $\psi^\s:\X^V\to \Y^V$ by
\[\psi^\s(\bf{x}) := \big(\psi(\Pi^\s_v(\bf{x}))\big)_{v\in V}.\]

This is easily visualized if $\psi$ is $F$-local for some finite $F \subseteq G$, and if $v$ is such that the map $F\to V:g\mapsto \s^g(v)$ is injective.  In this case the tuple $\Pi^\s_v(\bf{x})|_F$ may be regarded as a copy of the restriction $\bf{x}|_{\s^F(v)}$, `pulled back' so that it is labeled by $F$ itself.  Then $\psi^\s(\bf{x})$ is simply the result of applying $\psi$ to this restriction around each $v \in V$.  For a fixed finite $F$, sofic approximations give that most points $v\in V_n$ satisfy that injectivity requirement once $n$ is large.  A general measurable function $\X^G \to \Y$ may not be local, but it can be approximated by local functions.  As a result, the general definition of $\psi^\s$ still resembles that special case, up to some errors that we have to control from time to time.

The following lemma gives a useful compatibility between $\psi^\s$ and $\Pi^\s_v$.

\begin{lem}\label{lem:psi-sig-compatible}
Let $\S$ be a sofic approximation, let $F \subseteq G$ be finite, and suppose that $\psi:\X^G \to \Y$ is local.  Then the following holds w.h.p. in $v$:
\[\Pi^{\s_n}_v(\psi^{\s_n}(\cdot))|_F = \psi^F(\Pi^{\s_n}_v(\cdot)).\]
\end{lem}

More fully, this conclusion asserts that
\[[\quad \Pi^{\s_n}_v(\psi^{\s_n}(\bf{x}))|_F = \psi^F(\Pi^{\s_n}_v(\bf{x})) \quad \forall \bf{x} \in \X^{V_n} \quad ] \quad \hbox{w.h.p. in}\ v.\]

\begin{proof}
This is similar to the proof of Lemma~\ref{lem:approx-equiv}.  It suffices to prove it when $F$ is an arbitrary singleton, say $\{h\}$.  If $\psi$ is $D$-local, then it holds w.h.p. in $v$ that
\[\s_n^g(\s_n^h(v)) = \s_n^{gh}(v) \quad \forall g \in D.\]
For such $v$ we have
\begin{multline*}
\big(\Pi^{\s_n}_v(\psi^{\s_n}(\bf{x}))\big)_h = (\psi^{\s_n}(\bf{x}))_{\s_n^h(v)} = \psi(\Pi^{\s_n}_{\s_n^h(v)}(\bf{x})) = \psi\big((x_{\s_n^g(\s_n^h(v))})_{g \in G}\big)\\
 = \psi\big((x_{\s_n^{gh}(v)})_{g \in D}\big) = \psi((\Pi^{\s_n}_v(\bf{x}))|_{Dh}) = \psi^{\{h\}}(\Pi^{\s_n}_v(\bf{x})).
\end{multline*}
\end{proof}

Now fix compact generating metrics $d_\X$ and $d_\Y$. The next result shows that the maps $\psi^\s$ inherit some regularity from $\psi$.

\begin{lem}\label{lem:Lip-still-Lip}
Suppose that $D\subseteq G$ is finite, that $\psi:\X^G\to \Y$ is $D$-local, and that $U\subseteq \X^G$ is a $D$-local open set such that $\mu(U) > 1 -  \eta$ and such that $\psi|U$ is $\eta$-almost $L$-Lipschitz from $d_\X^{(D)}$ to $d_\Y$. Then there is a w$^\ast$-neighbourhood $\calO$ of $\mu$ such that
\[\psi^{\s_n}|\,\O(\calO,\s_n)\]
is $3\eta$-almost $(L|D|)$-Lipschitz from $d_\X^{(V_n)}$ to $d_\Y^{(V_n)}$ for all sufficiently large $n$.
\end{lem}

\begin{proof}
This is similar to the almost-Lipschitz estimate in the proof of Lemma~\ref{lem:good-approx-Ham-sum}.  Assume that $d_\Y$ has diamater at most $1$ for simplicity.  Let
\[\calO = \{\theta\in \Pr(\X^G):\ \theta(U) > 1 - \eta\},\]
which is w$^\ast$-open by the Portmanteau Theorem. Suppose that $\bf{x},\bf{x}' \in \O(\calO,\s_n)$ for some $n$.  Then the definition of $\calO$ implies that
\[|\{v \in V:\ \Pi^{\s_n}_v(\bf{x}) \not\in U\}| = P^{\s_n}_\bf{x}(\X^G\setminus U) < \eta,\]
and similarly for $\bf{x}'$.  It follows that
\begin{align*}
&d_\Y^{(V_n)}(\psi^{\s_n}(\bf{x}),\psi^{\s_n}(\bf{x}'))\\
&\leq \frac{1}{|V_n|}\sum_{\scriptsize{\begin{array}{cc}v \in V_n,\\ \Pi^{\s_n}_v(\bf{x}),\Pi^{\s_n}_v(\bf{x}') \in U\end{array}}}d_\Y\big(\psi(\Pi^{\s_n}_v(\bf{x})),\psi(\Pi^{\s_n}_v(\bf{x}'))\big) + 2\eta\\
&\leq \frac{1}{|V_n|}\sum_{\scriptsize{\begin{array}{cc}v \in V_n,\\ \Pi^{\s_n}_v(\bf{x}),\Pi^{\s_n}_v(\bf{x}') \in U\end{array}}}Ld_\X^{(D)}(\Pi^{\s_n}_v(\bf{x}),\Pi^{\s_n}_v(\bf{x}')) + 3\eta\\
&\leq L|D|d_\X^{(V_n)}(\bf{x},\bf{x}') + 3\eta.
\end{align*}
\end{proof}

Now suppose that $\Phi = \phi^G:(\X^G,\mu,S) \to (\Y^G,\nu,S)$ is a factor map, and that $d_\X$ and $d_\Y$ are compact generating metrics on $\X$ and $\Y$ with diameter at most $1$. The next proposition shows that AL approximations to $\phi$ approximately preserve good models.  Related facts can be found within the proofs of~\cite[Theorem 2.6]{KerLi11b} and~\cite[Proposition 3.4]{KerLi13}, which show that the Kerr-Li definition of sofic entropy is independent of an underlying choice of a `dynamically generating' sequence of bounded functions or of a `dynamically generating' pseudometric.

\begin{prop}\label{prop:approx-by-Lip-maps}
For every w$^\ast$-neighbourhood $\calN$ of $\nu$ there is an $\eta > 0$ with the following property.  If $\psi$ is an $\eta$-AL approximation to $\phi$ rel $(\mu,d_\X,d_\Y)$, then there is a w$^\ast$-neighbourhood $\calO$ of $\mu$ such that
\[\psi^{\s_n}\big(\O(\calO,\s_n)\big) \subseteq \O(\calN,\s_n)\]
for all sufficiently large $n$.
\end{prop}

\begin{proof}
It suffices to prove this for a sub-basis of w$^\ast$-neighbourhoods $\calN$ of $\nu$, so we may assume that
\[\calN = \big\{\theta \in \Pr(\Y^G):\ \theta_E \in \calN_1\big\}\]
for some finite $E\subseteq G$ and w$^\ast$-neighbourhood $\calN_1$ of $\nu_E$.

Now let $\calN_1' \subseteq \calN_1$ be another w$^\ast$-neighbourhood of $\nu_E$ with the property that any measure which lies sufficiently close to $\calN'_1$ in total variation must lie inside $\calN_1$.  This is possible because the total variation norm is stronger than the w$^\ast$-topology.  Let
\[\calN' := \big\{\theta \in \Pr(\Y^G):\ \theta_E \in \calN_1'\big\}.\]

Let $\eta > 0$ be given by Lemma~\ref{lem:pushfwd-approx} so that, for any $\eta$-AL approximation $\psi$ to $\phi$, there is a w$^\ast$-neighbourhood $\calO$ of $\mu$ such that
\[(\psi^G)_\ast(\calO) \subseteq \calN'.\]
We will show that this $\eta$ also has the property required for the present proposition, and that we can use the same $\calO$ for the function $\psi$.

Indeed, since $\psi$ is a local function, Lemma~\ref{lem:psi-sig-compatible} gives that
\[\Pi^{\s_n}_v(\psi^{\s_n}(\cdot))|_E = \psi^E(\Pi^{\s_n}_v(\cdot)) \quad \hbox{w.h.p. in}\ v.\]
Using this, for any $\bf{x} \in \X^{V_n}$ the definition of empirical measures gives
\[(P^{\s_n}_{\psi^{\s_n}(\bf{x})})_E = \frac{1}{|V_n|}\sum_{v \in V_n} \delta_{\Pi^{\s_n}_v(\psi^{\s_n}(\bf{x}))|_E}
= \frac{1}{|V_n|}\sum_{v \in V_n}\delta_{\psi^E(\Pi^{\s_n}_v(\bf{x}))} + o(1)
= (\psi^E)_\ast P_\bf{x}^{\s_n} + o(1)\]
as $n\to\infty$, where this approximation is in total variation and is uniform in $\bf{x}$. Therefore, if $\bf{x} \in \O(\calO,\s_n)$, then $P^{\s_n}_\bf{x} \in \calO$, and so
\[(\psi^G)_\ast P^{\s_n}_\bf{x} \in \calN', \quad \hbox{i.e.} \quad (\psi^E)_\ast P^{\s_n}_\bf{x} \in \calN'_1.\]
Once $(P^{\s_n}_{\psi^{\s_n}(\bf{x})})_E$ lies close enough to $(\psi^E)_\ast P^{\s_n}_\bf{x}$ in total variation, it lies inside $\calN_1$.  Since the total variation estimate above was uniform in $\bf{x}$, this gives that $P^{\s_n}_{\psi^{\s_n}(\bf{x})} \in \calN$ for all $\bf{x} \in \O(\calO,\s_n)$, for all sufficiently large $n$.
\end{proof}

We will sometimes need to compare different approximations to the same factor map.

\begin{lem}\label{lem:comparing-approxs}
Fix $\eta > 0$, and let $\psi$ and $\psi'$ be two $\eta$-AL approximations to $\phi$ rel $(\mu,d_\X,d_\Y)$.  Then there is a w$^\ast$-neighbourhood $\calO$ of $\mu$ such that
\[d_\Y^{(V)}\big(\psi^\s(\bf{x}),(\psi')^\s(\bf{x})\big) < 10\eta \quad \forall \bf{x} \in \O(\calO,\s)\]
for any map $\s:G\to\Sym(V)$.
\end{lem}

\begin{proof}
Recall that we assume $d_\Y$ has diameter at most $1$.

Let $D$ and $U$ (respectively $D'$ and $U'$) be as in Definition~\ref{dfn:eta-approx} for the map $\psi$ (respectively $\psi'$). By replacing each of $D$ and $D'$ with their union, we may assume they are equal.  Having done so, both $U$ and $U'$ are $D$-local.  Consider the function $h:\X^D\to \bbR$ defined by
\[h(x) := d_\Y(\psi(x),\psi'(x)).\]
This is $D$-local, and an easy check shows that $h|U\cap U'$ is $(2\eta)$-almost Lipschitz according to $d_\X^{(D)}$.

Invoking Lemma~\ref{lem:almost-Lip-near-Lip}, let $f:\X^G \to [0,1]$ be a $D$-local function which is truly Lipschitz according to $d_\X^{(D)}$ and has the property that
\[\|(h - f)|U\cap U'\|_\infty \leq 2\eta.\]
It follows that
\begin{multline*}
\int f\,\d\mu \leq \mu(\X^G\setminus (U\cap U')) + 2\eta + \int_{U\cap U'}h\,\d\mu\\ < 4\eta + \int d_\Y(\phi(x),\psi(x))\,\mu(\d x) + \int d_\Y(\phi(x),\psi'(x))\,\mu(\d x) < 6\eta,
\end{multline*}
using~(\ref{eq:int-approx}), the definition of $h$ and the triangle inequality for $d_\Y$.

Now let
\[\calO := \Big\{\theta\in \Pr(\X^G):\ \theta(U\cap U') > 1 - 2\eta\ \hbox{and}\ \int f\,\d\theta < 6\eta\Big\}.\]
This is a w$^\ast$-open set which contains $\mu$ by construction.  The function $f$ was introduced for the sake of defining $\calO$, since $h$ itself may not be strictly continuous and so cannot be used to define a w$^\ast$-open set in the same way.  For $\bf{x} \in \O(\calO,\s)$, it follows that
\[d_\Y^{(V)}\big(\psi^\s(\bf{x}),(\psi')^\s(\bf{x})\big) = \int h\,\d P_{\bf{x}}^\s \leq \int f\,\d P_{\bf{x}}^\s + 2\eta + P_{\bf{x}}^\s(\X^G\setminus (U\cap U')) < 10\eta.\]
\end{proof}

\subsection{Formulations in terms of sequences}

In many of the arguments in Part 2, instead of working with a single AL-approximation to a factor map, it will be more convenient to work with a sequence of increasingly good approximations.  We therefore make the following relative of Definition~\ref{dfn:eta-approx}.

\begin{dfn}\label{dfn:eta-approx-seq}
An \textbf{almost Lipschitz} (or \textbf{AL}) \textbf{approximating sequence for $\phi$ rel $(\mu,d_\X,d_\Y)$} is a sequence $(\psi_m)_{m\geq 1}$ such that each $\psi_m$ is an $\eta_m$-AL approximation to $\phi$ rel $(\mu,d_\X,d_\Y)$ for some sequence $\eta_m \downarrow 0$. This situation is denoted by
\[\psi_m \stackrel{\rm{aL}}{\to} \phi \quad \hbox{rel}\ (\mu,d_\X,d_\Y).\]
\end{dfn}

As before, we sometimes drop the qualifier `rel $(\mu,d_\X,d_\Y)$' when these data are clear.

We now reformulate a few of the results above in terms of AL approximating sequences.  This will assist in their application later in the paper.  We will rely on the standard fact that the w$^\ast$-topology on $\Pr(\X^G)$ is first countable, since it is metrizable.  This will be used again later without further explanation.

We start with the reformulation of Lemma~\ref{lem:pushfwd-approx}.

\begin{cor}\label{cor:pushfwd-approx}
Let $\mu$, $\nu$ and $\phi$ be as above, let $(\psi_k)_{k\geq 1}$ be an AL approximating sequence to $\phi$ rel $(\mu,d_\X,d_\Y)$, and let $\calO_1 \supseteq \calO_2 \supseteq \dots$ and $\calN_1 \supseteq \calN_2 \supseteq \dots$ be bases for the w$^\ast$ topologies at $\mu$ and $\nu$ respectively.  Then whenever the sequence $(k_n)_{n\geq 1}$ grows sufficiently slowly, it holds that whenever the sequence $(m_n)_{n\geq 1}$ grows sufficiently slowly, we have
\[(\psi_{k_n}^G)_\ast(\calO_n) \subseteq \calN_{m_n}\]
for all sufficiently large $n$.
\end{cor}

\begin{rmk}
The conclusion here needs to be parsed carefully: both sequences $(k_n)_{n\geq 1}$ and $(m_n)_{n\geq 1}$ must be chosen to grow sufficiently slowly, but the bound on the growth of the second sequence may depend on the choice of the first sequence. \fin
\end{rmk}

\begin{proof}
The desired conclusion is not disrupted if we change $\calN_1$, so let us assume that $\calN_1 = \Pr(\Y^G)$.

Each $\psi_k$ is an $\eta_k$-AL approximation to $\phi$ for some parameters $\eta_k \downarrow 0$.  Therefore, for each $m$, Lemma~\ref{lem:pushfwd-approx} gives some $K(m)$ such that for every $k\geq K(m)$ there is an $N(k,m)$ for which
\[(\psi_k^G)_\ast(\calO_n) \subseteq \calN_m \quad \forall n\geq N(k,m).\]
Since $\calN_1 = \Pr(\Y^G)$, we may take $K(1) = 1$ and $N(k,1) = 1$ for all $k$.

By replacing each $N(k,m)$ with the value
\[\max\{N(k',m'):\ 1 \leq k' \leq k,\ 1 \leq m' \leq m\},\]
we may also assume that
\[N(k_1,m_1) \leq N(k_2,m_2) \quad \hbox{whenever}\ k_1 \leq k_2\ \hbox{and}\ m_1 \leq m_2.\]

Now for each $n$ define
\[s_n := \max\{s:\ N(s,s) \leq n\}.\]\
This is well-defined for all $n\geq 1$ because we have assumed that ${N(1,1) = 1}$. Clearly $s_1 \leq s_2 \leq \dots$ and these values tend to $\infty$.

Finally, assume that $(k_n)_{n\geq 1}$ is a sequence tending to $\infty$ for which $k_n \leq s_n$ for all $n$.   Having done so, assume that $(m_n)_{n\geq 1}$ is a sequence tending to $\infty$ which grows slowly enough that
\[m_n \leq s_n \quad \hbox{and} \quad K(m_n) \leq k_n \quad \forall n.\]
Then for any $n$ we have $k_n \geq K(m_n)$ and also
\[k_n,m_n \leq s_n \quad \Longrightarrow \quad N(k_n,m_n) \leq N(s_n,s_n) \leq n.\]
Therefore
\[(\psi_{k_n}^G)_\ast(\calO_n) \subseteq (\psi_{k_n}^G)_\ast(\calO_{N(k_n,m_n)}) \subseteq \calN_{m_n}.\]
\end{proof}

Proposition~\ref{prop:approx-by-Lip-maps} has an analogous reformulation in terms of AL approximating sequences.

\begin{cor}\label{cor:approx-by-Lip-maps}
Let $\mu$, $\nu$ and $\phi$ be as above, let $(\psi_k)_{k\geq 1}$ be an AL approximating sequence to $\phi$ rel $(\mu,d_\X,d_\Y)$, and let $\calO_1 \supseteq \calO_2 \supseteq \dots$ and $\calN_1 \supseteq \calN_2 \supseteq \dots$ be bases for the w$^\ast$ topologies at $\mu$ and $\nu$ respectively.  Then whenever the sequence $(k_n)_{n\geq 1}$ grows sufficiently slowly, it holds that whenever the sequence $(m_n)_{n\geq 1}$ grows sufficiently slowly, we have
\[\psi_{k_n}^{\s_n}\big(\O(\calO_n,\s_n)\big) \subseteq \O(\calN_{m_n},\s_n)\]
for all sufficiently large $n$.
\end{cor}

\begin{proof}
This follows from Proposition~\ref{prop:approx-by-Lip-maps} just as Corollary~\ref{cor:pushfwd-approx} follows from Lemma~\ref{lem:pushfwd-approx}.
\end{proof}

It will also be convenient to have a version of Lemma~\ref{lem:psi-sig-compatible} in terms of sequences.

\begin{cor}\label{cor:psi-sig-compatible}
Let $\S$ be a sofic approximation, let $F_1 \subseteq F_2 \subseteq \cdots$ be finite subsets of $G$, and let $\psi_1$, $\psi_2$, \dots be a sequence of local functions from $\X^G$ to $\Y$.  Then it holds that
\[\Pi^{\s_n}_v(\psi_{k_n}^{\s_n}(\cdot))|_{F_{m_n}} = \psi_{k_n}^{F_{m_n}}(\Pi^{\s_n}_v(\cdot)) \quad \hbox{w.h.p. in}\ v\ \hbox{as}\ n\to\infty\]
for any two sequences $k_1 \leq k_2 \leq \dots$ and $m_1 \leq m_2 \leq \dots$ which both grow sufficiently slowly.
\end{cor}

\begin{proof}
For each integer $\ell \geq 1$, Lemma~\ref{lem:psi-sig-compatible} gives some $N(\ell) \in \bbN$ such that
\begin{multline*}
\Big|\Big\{v \in V_n:\ \Pi^{\s_n}_v(\psi_k^{\s_n}(\cdot))|_{F_m} = \psi_k^{F_m}(\Pi^{\s_n}_v(\cdot))\ \forall k,m \in \{1,2,\dots,\ell\}\Big\}\Big|\\ \geq (1 - 2^{-\ell})|V_n| \quad \forall n \geq N(\ell).
\end{multline*}
Clearly we may assume that $N(1) \leq N(2) \leq \dots$.  Now the desired conclusion holds provided
\[k_n,m_n \leq \max(\{1\} \cup \{\ell:\ N(\ell) \leq n\}) \quad \forall n.\]
\end{proof}

\subsection{A categorial point of view}

Section 3 has shown how a metric $G$-process $(\X^G,\mu,S,d)$ may be converted into the sequences of good-model-spaces $\O(\calO,\s_n)$ for $n\geq 1$ and w$^\ast$-neighbourhoods $\calO$ of $\mu$.  Then, Section 4 has shown how a factor map
\[\Phi = \phi^G:(\X^G,\mu,S,d_\X) \to (\Y^G,\nu,S,d_\Y)\]
may be converted into the maps
\[\psi^{\s_n}:\X^{V_n}\to \Y^{V_n},\]
where $\psi$ is an $\eta$-AL approximation to $\phi$ for some small $\eta$.  These maps respect the subsets of good models in the sense of Proposition~\ref{prop:approx-by-Lip-maps}.

One can describe all this work as setting up a functor from the category of metric $G$-processes to another category.  The target category here should have objects that are families of sequences of subsets $\X^{V_n}$, such as our sets $\O(\calO,\s_n)$, or possibly equivalence classes of such families under a kind of `asymptotic equivalence'.  The morphisms should be (equivalence classes of) families of sequences of maps $\X^{V_n}\to \Y^{V_n}$ which respect those sequences of subsets, such as the maps $\psi^{\s_n}$ for the possible choices of $\eta$-AL approximation $\psi$ to $\phi$ as $\eta \to 0$.

\part{Measures on model spaces and associated entropies}

\section{Convergence of measures on model spaces}\label{sec:loc-and-quench}

\subsection{Local weak$^\ast$ convergence and distributions on measures}

Let $\S = (\s_n)_{n\geq 1}$ be a sofic approximation to $G$, and let $(\X^G,\mu,S,d)$ be a metric $G$-process.  It will be clear that all the notions and results of this subsection depend on $d$, as well as on the process $(\X^G,\mu,S)$.

Suppose also that $\mu_n \in \Pr(\X^{V_n})$ for each $n \geq 1$.

\begin{dfn}\label{dfn:localweak}
The sequence $(\mu_n)_{n\geq 1}$ \textbf{locally weak$^\ast$ converges} to $\mu$ if for every w$^\ast$-neighbourhood $\calO$ of $\mu$ it is the case that
\[(\Pi^{\s_n}_v)_\ast\mu_n \in \calO \quad \hbox{w.h.p. in}\ v\ \hbox{as}\ n\to\infty.\]
This is denoted by $\mu_n \lws \mu$.
\end{dfn}

That is, once $n$ is large, the local marginals of $\mu_n$ resemble those of $\mu$ around most vertices of $V_n$.  This kind of convergence depends on $d$ through the resulting w$^\ast$ topology on $\Pr(\X^G)$.

Bernoulli shifts give the obvious examples.

\begin{lem}\label{lem:Bern-q}
If $(\X^G,\nu^{\times G},S)$ is a Bernoulli process over $G$ and $d$ is any choice of compact generating metric on $\X$, then $\nu^{\times V_n} \lws \nu^{\times G}$. \qed
\end{lem}

Definition~\ref{dfn:localweak} and some relatives have an established role in statistical physics.  They appear naturally in analyses of the asymptotic behaviour of some classical statistical physics models, such as the Ising model, over sequences of sparse graphs.  In case those underlying graphs converge to a limiting infinite graph in a suitable sense, one can ask whether the Gibbs measures constructed over them converge to a Gibbs measure over that infinite graph.  See, for instance,~\cite[Subsection 2.2]{MontMosSly12}, where local weak$^\ast$ convergence is called `[weak] local convergence in probability'.  Such convergence for sequences of measures is in much the same spirit as Benjamini-Schramm convergence for sparse random graphs themselves~\cite{BenSch01}.

The next definition connects local weak$^\ast$ convergence and model spaces.

\begin{dfn}\label{dfn:approxs}
The sequence $(\mu_n)_{n\geq 1}$ \textbf{quenched converges} to $\mu$ if
\begin{itemize}
\item[i)] $\mu_n \lws \mu$, and
\item[ii)] $\mu_n(\O(\calO,\s_n)) \to 1$ as $n\to\infty$ for any w$^\ast$-neighbourhood $\calO$ of $\mu$.
\end{itemize}
This is denoted by $\mu_n \q \mu$.
\end{dfn}

One can strengthen Lemma~\ref{lem:Bern-q} to show that $\nu^{\times V_n} \q \nu^{\times G}$.  This can be proved directly using the Law of Large Numbers, but we will deduce it after developing some more general theory: see Corollary~\ref{cor:Bern-q} below.

It is sometimes more convenient to replace Definition~\ref{dfn:localweak} or~\ref{dfn:approxs} with the following variants.  The proofs are immediate, and are omitted.

\begin{lem}\label{lem:approxs-reform}
Let $\calO_1 \supseteq \calO_2 \supseteq \dots$ be a fixed basis of w$^\ast$-neighbourhoods of $\mu$.  Then $\mu_n \lws \mu$ if and only if it holds that
\begin{equation}\label{eq:localweak2}
(\Pi_v^{\s_n})_\ast\mu_n \in \calO_{k_n}  \quad \hbox{w.h.p. in}\ v\ \hbox{as}\ n\to\infty
\end{equation}
whenever the seqeuence $k_1 \leq k_2 \leq \dots$ grows sufficiently slowly.  Similarly, $\mu_n \q \mu$ if and only if we have both~(\ref{eq:localweak2}) and
\[\mu_n(\O(\calO_{k_n},\s_n)) \to 1 \quad \hbox{as}\ n\to\infty\]
whenever the seqeuence $k_1 \leq k_2 \leq \dots$ grows sufficiently slowly. \qed
\end{lem}

In general, it may happen that $\mu_n \lws \mu$ as in Definition~\ref{dfn:localweak}, but condition (ii) of Definition~\ref{dfn:approxs} is not satisfied.  In this case, consider the empirical-distribution maps
\[P^{\s_n}:\X^{V_n} \to \Pr(\X^G):\bf{x}\mapsto P^{\s_n}_{\bf{x}}.\]
Pushing forward through these maps gives a sequence of distributions on measures
\[P^{\s_n}_\ast\mu_n \in \Pr(\Pr(\X^G)).\]
Since the weak$^\ast$ topology on $\Pr(\X^G)$ is compact and metrizable, this space of distributions on measures carries a weak$^\ast$ topology of its own, which is also compact and metrizable.  In the sequel it should always be clear which of these weak$^\ast$ topologies is being referred to.

The distributions on measures $P^{\s_n}_\ast\mu_n$ give the following useful characterization of the difference between Definitions~\ref{dfn:localweak} and~\ref{dfn:approxs}.

\begin{lem}\label{lem:lwa}
If $\mu_n \lws \mu$, then $\mu_n \q \mu$ if and only if
\[P^{\s_n}_\ast\mu_n \stackrel{\rm{weak}^\ast}{\to} \delta_{\mu}.\]
\end{lem}

\begin{proof}
For any w$^\ast$-neighbourhood $\calO$ of $\mu$, we have
\[\mu_n(\O(\calO,\s_n)) = \mu_n\big\{\bf{x} \in \X^{V_n}:\ P^{\s_n}_{\bf{x}} \in \calO\big\} = (P^{\s_n}_\ast\mu_n)(\calO).\]
By the Portmanteau Theorem, the w$^\ast$-convergence of $P^{\s_n}_\ast\mu_n$ to $\delta_\mu$ is equivalent to the convergence $(P^{\s_n}_\ast\mu_n)(\calO) \to 1$ for every such $\calO$.
\end{proof}

Given only that $\mu_n \lws \mu$, the distributions on measures $P^{\s_n}_\ast\mu_n$ converge to a decomposition of $\mu$ into other invariant measures.  The next lemma describes this.

\begin{lem}\label{lem:inv-bary}
If $\theta \in \Pr(\Pr(\X^G))$ is a subsequential w$^\ast$-limit of the sequence of distributions on measures $(P^{\s_n}_\ast\mu_n)_{n\geq 1}$, then $\theta(\Pr^S(\X^G)) = 1$, and the barycentre of $\theta$ is equal to $\mu$, meaning that
\[\int \nu\,\theta(\d\nu) = \mu.\]
\end{lem}

\begin{proof}
By passing to a subsequence we may simply assume that
\begin{eqnarray}\label{eq:randmeasweakstar}
P^{\s_n}_\ast\mu_n \stackrel{\rm{weak}\ast}{\to} \theta.
\end{eqnarray}

\vspace{7pt}

\emph{Support.}\quad For each local function $f \in C(\X^G)$ and each $g \in G$, Lemma~\ref{lem:approx-invar} gives that
\[\sup_{\bf{x} \in \X^{V_n}}\Big|\int f\,\d P^{\s_n}_{\bf{x}} - \int f\circ S^g\,\d P^{\s_n}_{\bf{x}}\Big| \to 0.\]
Since local functions are uniformly dense in $C(\X^G)$, this implies that
\[(P^{\s_n}_\ast\mu_n)\Big\{\nu \in \Pr(\X^G):\ \int f\,\d\nu \approx_\eps \int f\circ S^g\,\d\nu \Big\} \to 1\]
for all $f \in C(X^G)$ and $\eps > 0$.  By the Portmeanteau Theorem, it follows that $\theta$ is supported on $\Pr^S(\X^G)$.

\vspace{7pt}

\emph{Barycentre.}\quad For any $f \in C(\X^G)$, the weak$^\ast$ convergence in~(\ref{eq:randmeasweakstar}) gives
\begin{align*}
\iint f(x)\,\nu(\d x)\,\theta(\d\nu) &= \lim_{n\to\infty} \iint f(x)\,\nu(\d x)\,(P^{\s_n}_\ast\mu_n)(\d\nu)\\ &= \lim_{n\to\infty} \iint f(x)\,P^{\s_n}_{\bf{x}}(\d x)\,\mu_n(\d\bf{x})\\
&= \lim_{n\to\infty} \frac{1}{|V_n|}\sum_{v\in V_n}\int f(\Pi^{\s_n}_v(\bf{x}))\,\mu_n(\d\bf{x})\\ &= \lim_{n\to\infty} \frac{1}{|V_n|}\sum_{v\in V_n}\int f\,\d((\Pi^{\s_n}_v)_\ast\mu_n),
\end{align*}
and this converges to $\int f\,\d\mu$ because $\mu_n \lws \mu$.
\end{proof}

The two previous lemmas have the following immediate consequence.

\begin{cor}\label{cor:lws-lwa}
If $(\X^G,\mu,S,d)$ is ergodic then
\[\mu_n \lws \mu \quad \Longrightarrow \quad \mu_n \q \mu.\]
\end{cor}

\begin{proof}
Since $\mu$ is ergodic, it has no nontrivial representation as a barycentre of other invariant measures.  In view of Lemma~\ref{lem:inv-bary}, it follows that the only possible subsequential limit of $( P^{\s_n}_\ast\mu_n)_{n\geq 1}$ is the Dirac mass at $\mu$ itself.  Now Lemma~\ref{lem:lwa} completes the proof.
\end{proof}

In light of Lemma~\ref{lem:Bern-q}, we can immediately deduce the following.

\begin{cor}\label{cor:Bern-q}
If $(\X^G,\nu^{\times G},S)$ is a Bernoulli process over $G$ and $d$ is any choice of compact generating metric on $\X$, then $\nu^{\times V_n} \q \nu^{\times G}$. \qed
\end{cor}

This result is already implicit in the proofs of~\cite[Theorem 8]{Bowen10} and~\cite[Lemma 2.2]{KerLi11a}, which calculate the sofic entropies of Bernoulli processes.

\begin{rmk}
Another notion of convergence, also introduced in~\cite{MontMosSly12}, is `local-on-average weak$^\ast$ convergence'.  It requires only that
\[\frac{1}{|V_n|}\sum_{v\in V_n}(\Pi^{\s_n}_v)_\ast\mu_n \stackrel{\rm{weak}^\ast}{\to} \mu.\]
This is clearly weaker than local weak$^\ast$ convergence.  Indeed, any sequence of measures which satisfy only condition (ii) in Definition~\ref{dfn:approxs} must have this property, and a variation of the proof of Corollary~\ref{cor:lws-lwa} shows that these are equivalent if $\mu$ is ergodic.  We do not use this kind of convergence in this paper. \fin
\end{rmk}

\subsection{Doubly-quenched convergence}

For any $\mu_n \in \Pr(\X^{V_n})$ we have
\[(\Pi^{\s_n}_v)_\ast(\mu_n \times \mu_n) = (\Pi^{\s_n}_v)_\ast\mu_n \times (\Pi^{\s_n}_v)_\ast\mu_n \quad \forall v \in V_n.\]
Therefore
\[\mu_n \lws \mu \quad \Longrightarrow \quad \mu_n\times \mu_n \lws \mu\times \mu.\]

However, the analogous implication can fail for quenched convergence.

\begin{ex}\label{ex:T}
Let $G$ be a residually finite group and let $G > G_1 > G_2 > \ldots$ be finite-index normal subgroups whose intersection is $\{e\}$.  Let $X$ be the compact inverse limit of the tower of finite groups
\[\dots \onto G/G_2 \onto G/G_1.\]
This $X$ is a compact group. Let $m$ be its Haar measure, and let $T$ be the action of $G$ on $X$ by left rotations.  Let $d$ be a left-invariant compact group metric on $X$.  The map~(\ref{eq:system-process}) gives a metric $G$-process $(X^G,\mu,S,d)$ isomorphic to the Kronecker system $(X,m,T)$.

For each $n$, let $\mu_{(n)}$ be the pushforward of the measure $\mu$ under the coordinate-wise factor map $X^G \to (G/G_n)^G$.  Then $\mu_{(n)}$ is supported on the $G_n$-periodic elements of $(G/G_n)^G$.

Let $V_n := G/G_n$ and let $\s_n:G\to \Sym(V_n)$ be the left-rotation action of $G$ for each $n$.  These together give a sofic  approximation $\S$ to $G$.

For each $n$ we now construct a measure $\mu_n \in \Pr(X^{V_n})$ as follows.  Let $S_n \subseteq G$ be a cross-section of $G_n$ in $G$. The identity mapping $V_n\to G/G_n$ may be regarded as an element $\bf{z}_n \in (G/G_n)^{V_n}$.  Let $\mu_n^\circ \in \Pr((G/G_n)^{V_n})$ be the law of a random rotate of $\bf{z}_n$: that is,
\[\mu_n^\circ := \frac{1}{|S_n|}\sum_{g \in S_n}\delta_{\bf{z}_n\circ \s_n^g}.\]
Finally, let $\mu_n$ be any lift of $\mu_n^\circ$ to a measure on $X^{V_n}$.

For each $n$, the measure $\mu_n^\circ$ is the Haar measure on the $G$-orbit of $\bf{z}_n$.  That orbit is a free and transitive $(G/G_n)$-space.  Therefore the $G$-action on the finitely-supported measure $\mu_n^\circ$ is isomorphic to the left-rotation action on $G/G_n$ with Haar measure.  Composing with the map~(\ref{eq:system-process}), this isomorphism converts the elements $\bf{z}_n\circ \s_n^g \in (G/G_n)^{V_n}$ into the $G_n$-periodic  elements of $(G/G_n)^G$.  Therefore the points in the support of $\mu_n^\circ$ have empirical distribution that actually equals $\mu_{(n)}$, and now the local marginals of $\mu_n^\circ$ are also all equal to $\mu_{(n)}$.  It follows that $\mu_n \q \mu$ as $n\to\infty$.

However, we also have
\[\mu_n^\circ \times \mu_n^\circ = \frac{1}{|S_n|^2}\sum_{g,h \in S_n}\delta_{\bf{z}_n\circ \s_n^g}\times \delta_{\bf{z}_n\circ \s_n^h} = \frac{1}{|S_n|^2}\sum_{g,h \in S_n}\delta_{(\bf{z}_n\circ \s_n^g,\bf{z}_n\circ \s_n^g)\circ (\rm{id}_{V_n}\times \s_n^h)}.\]
From this we can calculate that the distribution
\[P^{\s_n}_\ast(\mu_n^\circ\times \mu_n^\circ) \in \Pr\big(\Pr\big((G/G_n\times G/G_n)^G\big)\big)\]
is the law of the random measure
\[\int \delta_{(x,S^hx)}\,\mu_{(n)}(\d x),\]
where $h$ is a uniform random element of $S_n$.  This shows that $P^{\s_n}_\ast(\mu_n \times \mu_n)$ does not converge weakly$^\ast$ to $\delta_{\mu\times \mu}$, but rather to the disintegration of $\mu\times \mu$ into the ergodic components supported on the cosets of the diagonal subgroup in $X\times X$.  Therefore, by Lemma~\ref{lem:lwa}, $\mu_n\times\mu_n$ does not quenched-converge to $\mu\times \mu$.

This example is particularly striking if $G$ has Kazhdan's property (T).  In that case, the Schreier graphs of the quotients $\s_n$ are expanders (see~\cite[Section 3.3]{Lubot--book} or~\cite[Section 6.1]{BekdelaHVal08}). Therefore the sofic approximation $\S$ cannot support arbitrarily good models for any non-ergodic $G$-system: in particular, $\rmh_\S(\mu\times \mu) = -\infty$.  In this case there can be no sequence of measures $\nu_n \in \Pr(X^{V_n})$ such that $\nu_n\times \nu_n \q \mu\times \mu$. \fin
\end{ex}

To rule out examples like these, we make the following definition.

\begin{dfn}
The sequence $(\mu_n)_{n\geq 1}$ \textbf{doubly-quenched converges} to $\mu$ if
\[\mu_n \times \mu_n \q \mu\times \mu.\]
This is denoted by $\mu_n \dq \mu$.
\end{dfn}

Once again, the obvious positive examples are Bernoulli processes.  Indeed, if $(\X^G,\nu^{\times G},S)$ is a Bernoulli process, then $\nu^{\times G}\times \nu^{\times G}$ may be identified with $(\nu\times \nu)^{\times G}$ and $\nu^{\times V_n}\times \nu^{\times V_n}$ may be identified with $(\nu\times \nu)^{\times V_n}$.  Therefore applying Corollary~\ref{cor:Bern-q} directly to $(\nu\times \nu)^{\times G}$ gives the following.

\begin{lem}\label{lem:Bern-dq}
If $(\X^G,\nu^{\times G},S)$ is a Bernoulli process over $G$ and $d$ is any choice of compact generating metric on $\X$, then $\nu^{\times V_n} \dq \nu^{\times G}$. \qed
\end{lem}

The `quenched' condition in Definition~\ref{dfn:localweak} asserts that, once $n$ is large, $\mu_n$ is mostly supported on individual models $\bf{x}$ whose empirical distribution is close to $\mu$.  This is a kind of `equidistribution' of the points that support $\mu$, and is an analog of ergodicity in the setting of a locally weak$^\ast$ convergent sequence of probability measures.  With this in mind, doubly-quenched convergence is the analog of weak mixing.

The main result of this subsection is Theorem A, which gives two other conditions that are equivalent to $\mu_n \dq \mu$.  It shows that doubly-quenched convergence is preserved by other Cartesian products as well.  This is analogous to some of the classical equivalent conditions for weak mixing of a probability-preserving transformation: see, for instance,~\cite[Theorem 2.6.1]{Pet83}

The proof of Theorem A requires little more than a few applications of the Cauchy--Bunyakowski--Schwartz Inequality, via the following easy consequence.

\begin{lem}\label{lem:CBS}
Let $(X,\mu)$ be a probability space, let $H$ be a real Hilbert space with inner product $\langle \cdot,\cdot\rangle$, let $a:X \to H$ be strongly measurable, and let $b \in H$ have norm at most $1$.  Then
\[\int |\langle a(x),b\rangle|\,\mu(\d x) \leq \sqrt{\iint |\langle a(x),a(x')\rangle|\,\mu(\d x)\,\mu(\d x')}.\]
\qed
\end{lem}

\begin{proof}[Proof of Theorem A]
\emph{(i) $\Longrightarrow$ (ii).}\quad It suffices to prove (ii) for a sub-basis of w$^\ast$-neighbourhoods of $\mu\times \nu$, so suppose that
\[\calN = \Big\{\lambda \in \Pr(\X^G\times \Y^G):\ \int f\otimes h\,\d\l \approx_\eps \int f\,\d\mu \int h\,\d\nu\Big\}\]
for some $f \in C(\X^G)$ and $h \in C(\Y^G)$ with $\|f\|_\infty, \|h\|_\infty \leq 1$ and some $\eps > 0$.

Let
\[\calO := \Big\{\theta \in \Pr(\Y^G):\ \int h\,\d\theta \approx_{\eps/2} \int h\,\d\nu\Big\}.\]
We will show that this $\calO$ suffices.  Suppose that $n_i \uparrow \infty$ and $\O(\calO,\s_{n_i}) \neq \emptyset$ for all $i$.  By relabeling the subsequence, we may assume that $n_i = i$ for all $i$, and so ignore the indexing by $i$.  Now conclusion (ii) will follow if we prove that
\[\mu_n\Big\{\bf{x} \in \X^{V_n}:\ \frac{1}{|V_n|}\sum_{v\in V_n}f(\Pi^{\s_n}_v \bf{x})h(\Pi^{\s_n}_v\bf{y}_n) \approx_\eps \int f\,\d\mu \int h\,\d\nu\Big\} \to 1\]
for any sequence of models $\bf{y}_n \in \O(\calO,\s_n)$.

Since $\|f\|_\infty \leq 1$, we have
\begin{multline*}
\Big|\frac{1}{|V_n|}\sum_{v\in V_n}f(\Pi^{\s_n}_v \bf{x})h(\Pi^{\s_n}_v\bf{y}_n) - \int f\,\d\mu \int h\,\d\nu\Big|\\ \leq \Big|\frac{1}{|V_n|}\sum_{v\in V_n}\Big(f(\Pi^{\s_n}_v \bf{x}) - \int f\,\d\mu\Big)h(\Pi^{\s_n}_v\bf{y}_n)\Big| + \Big|\frac{1}{|V_n|}\sum_{v\in V_n}h(\Pi^{\s_n}_v\bf{y}_n) - \int h\,\d\nu\Big|.
\end{multline*}
The last term here is less than $\eps/2$ by the definition of $\calO$.  It therefore remains to prove that
\[\mu_n\Big\{\bf{x} \in \X^{V_n}:\ \Big|\frac{1}{|V_n|}\sum_{v\in V_n}f(\Pi^{\s_n}_v \bf{x})h(\Pi^{\s_n}_v\bf{y}_n)\Big| < \eps/2\Big\} \to 1\]
under the extra assumption that $\int f\,\d\mu = 0$. By Chebyshev's Inequality, this will follow if we prove that
\[\int \Big|\frac{1}{|V_n|}\sum_{v\in V_n}f(\Pi^{\s_n}_v \bf{x})h(\Pi^{\s_n}_v\bf{y}_n)\Big|\,\mu_n(\d\bf{x}) \to 0\]
under that extra assumption.

To do this, consider the Hilbert spaces $H_n := \ell^2(V_n)$ with the inner products
\[\langle a,b\rangle_n := \frac{1}{|V_n|}\sum_{v\in V_n} a_vb_v.\]
Let
\[b:= (h(\Pi^{\s_n}_v\bf{y}_n))_{v\in V_n}\]
and define
\[a:\X^{V_n} \to H_n:\bf{x} \mapsto (f(\Pi^{\s_n}_v\bf{x}))_{v\in V_n}.\]
Now Lemma~\ref{lem:CBS} gives
\begin{multline*}
\int\Big|\frac{1}{|V_n|}\sum_{v\in V_n}f(\Pi^{\s_n}_v\bf{x})h(\Pi^{\s_n}_v\bf{y}_n)\Big|\,\mu_n(\d\bf{x}) = \int |\langle a(\bf{x}),b\rangle_n|\,\mu_n(\d\bf{x})\\
\leq \sqrt{\iint |\langle a(\bf{x}),a(\bf{x}')\rangle_n|\,\mu_n(\d\bf{x})\,\mu_n(\d\bf{x}')}.
\end{multline*}

However,
\[\iint |\langle a(\bf{x}),a(\bf{x}')\rangle_n|\,\mu_n(\d\bf{x})\,\mu_n(\d\bf{x}') = \iint \Big|\int f\otimes f\,\d P^{\s_n}_{(\bf{x},\bf{x}')}\Big|\,\mu_n(\d\bf{x})\,\mu_n(\d\bf{x}'),\]
and this converges to
\[\int f\otimes f\,\d(\mu\times\mu) = \Big(\int f\,\d\mu\Big)^2 = 0\]
as $n\to\infty$, by assumption (i).

\vspace{7pt}

\emph{(ii) $\Longrightarrow$ (iii).}\quad We have seen that $\mu_{n_i}\times \nu_i \lws \mu\times \nu$; the only issue is to show that $\mu_{n_i} \times \nu_i$ is asymptotically supported on good models for $\mu\times \nu$.  This now follows from conclusion (ii) and Fubini's Theorem, since $\nu_i$ is asymptotically supported on good models for $\nu$.

\vspace{7pt}

\emph{(iii) $\Longrightarrow$ (i).}\quad Clearly (i) is a special case of (iii).
\end{proof}

Conclusion (ii) of Theorem A has a corollary whose conclusion does not involve measures on model spaces.  It asserts that, asymptotically as $n\to\infty$, every sufficiently good model for $\nu$ in $\Y^{V_n}$ can be lifted to a good model for $\mu\times \nu$ in $(\X\times \Y)^{V_n}$.

\begin{cor}
Let $(\X^G,\mu,S,d_\X)$ and $(\Y^G,\mu,S,d_\Y)$ be metric $G$-processes, and suppose that $\mu_n \dq \mu$ over $\S$.  Let $\pi:(\X\times \Y)^G\to \Y$ be the projection onto the $\Y$-component of the identity coordinate, so that $\pi^{\s_n}:(\X\times \Y)^{V_n} \to \Y^{V_n}$ is the coordinate-projection for each $n$.  Then for every w$^\ast$-neighbourhood $\calN$ of $\mu\times \nu$ there is a w$^\ast$-neighbourhood $\calO$ of $\nu$ such that
\[\pi^{\s_n}\big(\O(\calN,\s_n)\big) \supseteq \O(\calO,\s_n)\]
for all sufficiently large $n$.  \qed
\end{cor}

The next theorem gives another equivalent characterization of the convergence $\mu_n \dq \mu$.  We separate it from Theorem A because it is not involved in the rest of this paper.

\begin{thm}\label{thm:autocorr}
Assume that $\mu_n \q \mu$.  Then $\mu_n \dq \mu$ if and only if
\[(\Pi_v^{\s_n},\Pi^{\s_n}_{v'})_\ast\mu_n \stackrel{\rm{weak}^\ast}{\to} \mu\times \mu \quad \hbox{w.h.p. in}\ (v,v') \in V_n\times V_n \ \hbox{as}\ n\to \infty.\]
\end{thm}

\begin{proof}
 \emph{($\Longrightarrow$).}\quad Let $f_1,f_2 \in C(\X^G)$, and suppose that $\int f_1\,\d\mu = 0$.  Then
\begin{eqnarray*}
&&\frac{1}{|V_n|^2}\sum_{v,v' \in V_n}\Big(\int f_1(\Pi^{\s_n}_v\bf{x})f_2(\Pi^{\s_n}_{v'}\bf{x})\,\mu_n(\d\bf{x})\Big)^2\\
&&= \frac{1}{|V_n|^2}\sum_{v,v' \in V_n}\iint f_1(\Pi^{\s_n}_v\bf{x})f_1(\Pi^{\s_n}_v\bf{x}')f_2(\Pi^{\s_n}_{v'}\bf{x})f_2(\Pi^{\s_n}_{v'}\bf{x}')\,\mu_n(\d\bf{x})\,\mu_n(\d\bf{x}')\\
&&= \frac{1}{|V_n|}\sum_{v\in V_n}\iint f_1(\Pi^{\s_n}_v\bf{x})f_1(\Pi^{\s_n}_v\bf{x}')\,\mu_n(\d\bf{x})\,\mu_n(\d\bf{x}')\\
&&\qquad \qquad \qquad \qquad \cdot\frac{1}{|V_n|}\sum_{v'\in V_n}\iint f_2(\Pi^{\s_n}_{v'}\bf{x})f_2(\Pi^{\s_n}_{v'}\bf{x}')\,\mu_n(\d\bf{x})\,\mu_n(\d\bf{x}')\\
&&= \iint \Big(\int f_1\otimes f_1\,\d P^{\s_n}_{(\bf{x},\bf{x}')}\Big)\,\mu_n(\d\bf{x})\,\mu_n(\d\bf{x}')\\
&&\qquad \qquad \qquad \qquad \cdot \iint \Big(\int f_2\otimes f_2\,\d P^{\s_n}_{(\bf{x},\bf{x}')}\Big)\,\mu_n(\d\bf{x})\,\mu_n(\d\bf{x}').
\end{eqnarray*}
Doubly-quenched convergence implies that the first integral in this product tends to
\[\int f_1\otimes f_1\,\d(\mu\times \mu) = \Big(\int f_1\,\d\mu\Big)^2 = 0.\]
Therefore Chebyshev's Inequality gives that
\[\int f_1(\Pi^{\s_n}_v\bf{x})f_2(\Pi^{\s_n}_{v'}\bf{x})\,\mu_n(\d\bf{x}) \to 0 \quad \hbox{w.h.p. in}\ (v,v').\]
Finally, adjusting by constants as in the proof of (i) $\Longrightarrow$ (ii) in Theorem A, it follows that
\[\int f_1(\Pi^{\s_n}_v\bf{x})f_2(\Pi^{\s_n}_{v'}\bf{x})\,\mu_n(\d\bf{x}) \to \int f_1\otimes f_2\,\d(\mu\times \mu) \quad \hbox{w.h.p. in}\ (v,v')\]
for arbitrary $f_1,f_2 \in C(\X^G)$.

\vspace{7pt}

\emph{($\Longleftarrow$).}\quad Let $f_1 := f_2 := f \in C(\X^G)$ have mean zero according to $\mu$.  Reversing the chain of equalities in the previous step, we see that the assumed weak$^\ast$ convergence implies that
\[\iint \Big(\int f\otimes f\,\d P^{\s_n}_{(\bf{x},\bf{x}')}\Big)\,\mu_n(\d\bf{x})\,\mu_n(\d\bf{x}') \to 0 \quad \hbox{as}\ n\to\infty.\]
Also, it is clear that
\begin{multline*}
\iint \Big(\int f_1\otimes f_2\,\d P^{\s_n}_{(\bf{x},\bf{x}')}\Big)\,\mu_n(\d\bf{x})\,\mu_n(\d\bf{x}')\\
 = \iint \Big(\int f_2\otimes f_1\,\d P^{\s_n}_{(\bf{x},\bf{x}')}\Big)\,\mu_n(\d\bf{x})\,\mu_n(\d\bf{x}')
 \end{multline*}
for any other $f_1,f_2 \in C(\X^G)$, and so the Polarization Identity gives that these integrals also tend to $0$ if $\int f_1\,\d\mu = \int f_2\,\d\mu = 0$.

Finally, this convergence generalizes to
\[\iint \Big(\int f_1\otimes f_2\,\d P^{\s_n}_{(\bf{x},\bf{x}')}\Big)\,\mu_n(\d\bf{x})\,\mu_n(\d\bf{x}') \to \int f_1\,\d\mu \int f_2\,\d\mu\]
for arbitrary $f_1,f_2 \in C(\X^G)$, because the assumption that $\mu_n \q \mu$ handles the case when either $f_1$ or $f_2$ is constant.
\end{proof}

Theorem~\ref{thm:autocorr} continues the analogy between doubly-quenched convergence and weak mixing: it corresponds to the classical fact that an ergodic probability-preserving transformation $(X,\mu,T)$ is weakly mixing if and only if
\[ \frac{1}{n}\sum_{m=1}^n\Big|\int f(x)g(T^mx)\,\mu(\d x) - \int f\,\d\mu \int g\,\d\mu\Big| \to 0 \quad \forall f,g \in L^2(\mu)\]
as $n\to\infty$ (see again~\cite[Theorem 2.6.1]{Pet83}).

Doubly-quenched convergence is not only analogous to weak mixing, but also logically related to it.

\begin{lem}\label{lem:q-dq-and-wm}
 If $(\X^G,\mu,S)$ is weakly mixing, then
\[\mu_n \q \mu \quad \Longrightarrow \quad \mu_n \dq \mu.\]
\end{lem}

\begin{proof}
 This follows from Corollary~\ref{cor:lws-lwa}, since if $\mu$ is weakly mixing then $\mu\times \mu$ is ergodic.
\end{proof}

\begin{rmk}
Another property which seems related to doubly-quenched convergence is `replica symmetry'.  In the study of spin glasses and other disorded systems in statistical physics, the term `replica symmetry' (or its negation, `replica symmetry breaking') is used for a variety of phenomena that are expected to occur or fail together in most models of interest.

Often such a model consists of a special sequence of probability measures $\mu_n$ on $\{\pm 1\}^{V_n}$ for some finite sets $V_n$, for instance given by a particular Hamiltonian. In spin glasses, the measures (and the Hamiltonians) are usually random themselves.  For such a sequence of measures, one popular meaning of `replica symmetry' is that the sequence of `overlaps'
\[R(\s,\s') := \frac{1}{n}\sum_{i=1}^n \s_i\s'_i, \quad (\s,\s') \in \{\pm 1\}^n\times \{\pm 1\}^n,\]
regarded as a sequence of random variables for the probabilities $\mu_n\times \mu_n$, should concentrate as $n\to\infty$.

Clearly this holds in case the sets $V_n$ are associated to some sofic approximation of a group $G$ and $\mu_n \times \mu_n \q \mu \times \mu$ for some shift-invariant measure $\mu$ on $\{\pm 1\}^G$.  But doubly-quenched convergence could be stronger in general.  We should be careful about bringing the term `replica symmetry' into ergodic theory, since it does have several meanings for the physicists and it is not yet clear under what conditions they coincide.  However, it would be very interesting to know whether such ideas or models can shed further light on doubly-quenched convergence.

An introduction to replica symmetry and replica symmetry breaking from a physical point of view can be found in Chapters 8, 12 and 19 of~\cite{MezMon09}. \fin
\end{rmk}

\subsection{Behaviour under factor maps}

This subsection considers how local weak$^\ast$ and quenched convergence behave under applying AL approximants to factor maps.

\begin{prop}\label{prop:approx-meas-and-factor}
Let $\Phi = \phi^G:(\X^G,\mu,S,d_\X) \to (\Y^G,\nu,S,d_\Y)$ be a factor map of metric $G$-processes.  Suppose that $\mu_n \lws \mu$, and let $\psi_k \aL \phi$ rel $(\mu,d_\X,d_\Y)$.  Then
\[(\psi_{k_n}^{\s_n})_\ast\mu_n \lws \nu\]
whenever the sequence $k_1 \leq k_2 \leq \dots$ grows sufficiently slowly.  The same holds if both instances of `$\lws$ ' are replaced with `$\q$' or if both are replaced with `$\dq$'.
\end{prop}

\begin{proof}
First consider the case of quenched convergence. Suppose $\mu_n \q \mu$.

Let $\calN_1 \supseteq \calN_2 \supseteq \dots$ be a basis of w$^\ast$-neighbourhoods at $\nu$.  By shrinking each $\calN_m$ if necessary, we may assume that each of them has the form
\[\calN_m = \big\{\theta \in \Pr(\Y^G):\ \theta_{F_m} \in \calN'_m\big\}\]
for some finite set $F_m \subseteq G$ and w$^\ast$-neighbourhood $\calN'_m$ of $\nu_{F_m}$.

Using Lemma~\ref{lem:approxs-reform}, we can choose a basis $\calO_1 \supseteq \calO_2 \supseteq \dots$ of w$^\ast$-neighbourhoods at $\mu$ for which
\begin{equation}\label{eq:lwtick}
(\Pi_v^{\s_n})_\ast\mu_n \in \calO_n \quad \hbox{w.h.p. in}\ v\ \hbox{as}\ n\to\infty
\end{equation}
and
\begin{equation}\label{eq:qtick}
\mu_n(\O(\calO_n,\s_n))\to 1 \quad \hbox{as}\ n\to\infty.
\end{equation}

By the combination of Corollaries~\ref{cor:pushfwd-approx} and~\ref{cor:approx-by-Lip-maps}, if we now choose $(k_n)_{n\geq 1}$ growing sufficiently slowly, then any choice of $(m_n)_{n\geq 1}$ which grows sufficiently slowly (depending on $(k_n)_{n\geq 1}$) gives
\begin{equation}\label{eq:containment}
(\psi_{k_n}^G)_\ast(\calO_n) \subseteq \calN_{m_n} \quad \hbox{and} \quad \psi_{k_n}^{\s_n}\big(\O(\calO_n,\s_n)\big) \subseteq \O(\calN_{m_n},\s_n)
\end{equation}
for all sufficiently large $n$.  In addition, by Corollary~\ref{cor:psi-sig-compatible}, if we choose $(k_n)_{n\geq 1}$ and $(m_n)_{n\geq 1}$ both growing slowly enough, then
\[\big((\Pi^{\s_n}_v)_\ast(\psi^{\s_n}_{k_n})_\ast\mu_n\big)_{F_{m_n}} = (\psi^{F_{m_n}}_{k_n})_\ast(\Pi^{\s_n}_v)_\ast\mu_n \quad \hbox{w.h.p. in}\ v\ \hbox{as}\ n\to\infty.\]

Having chosen such sequences $(k_n)_{n\geq 1}$ and $(m_n)_{n\geq 1}$, we obtain from~(\ref{eq:lwtick}) and~(\ref{eq:containment}) that the following all hold w.h.p. in $v$ as $n\to\infty$:
\begin{multline*}
\big((\Pi^{\s_n}_v)_\ast(\psi^{\s_n}_{k_n})_\ast\mu_n\big)_{F_{m_n}} = \big((\psi^G_{k_n})_\ast(\Pi^{\s_n}_v)_\ast\mu_n\big)_{F_{m_n}} \in \big((\psi_{k_n}^G)_\ast(\calO_n)\big)_{F_{m_n}} \\
\subseteq \{\theta_{F_{m_n}}:\ \theta \in \calN_{F_{m_n}}\} = \calN_{F_{m_n}}',
\end{multline*}
and hence $(\psi^{\s_n}_{k_n})_\ast\mu_n \lws \nu$.  Similarly, from~(\ref{eq:qtick}) and~(\ref{eq:containment}) we obtain that
\[((\psi^{\s_n}_{k_n})_\ast\mu_n)\big(\O(\calN_{m_n},\s_n)\big) \geq ((\psi^{\s_n}_{k_n})_\ast\mu_n)\big(\psi^{\s_n}_{k_n}\big(\O(\calO_n,\s_n)\big)\big) \geq \mu_n(\O(\calO_n,\s_n)) \to 1\]
as $n \to\infty$, so in fact $(\psi^{\s_n}_{k_n})_\ast\mu_n \q \nu$, by Lemma~\ref{lem:approxs-reform}.

For the case of local weak$^\ast$ convergence, we argue in the same way, except ignoring the lower bounds on $((\psi_{k_n}^{\s_n})_\ast\mu_n)(\O(\calN_{m_n},\s_n))$ and omitting the appeal to Corollary~\ref{cor:approx-by-Lip-maps}.

Finally, Lemma~\ref{lem:combine} gives that $\psi_k \times \psi_k \aL \phi\times \phi$ rel $\mu\times \mu$.  Therefore, if $\mu_n \dq \mu$, then we may apply the argument for the quenched case to $\mu_n \times \mu_n$.
\end{proof}

If $\phi:\X^G\to \Y$ is itself local and continuous, then by Lemma~\ref{lem:cts-alm-Lip} we may regard $\phi$ as an $\eta$-AL approximation to itself for every $\eta > 0$.  Therefore Proposition~\ref{prop:approx-meas-and-factor} has the following special case.

\begin{cor}\label{cor:approx-meas-and-factor}
If $\Phi = \phi^G$ is as in Proposition~\ref{prop:approx-meas-and-factor} with $\phi$ local and continuous, and $\mu_n \lws \mu$, then
\[\phi^{\s_n}_\ast\mu_n \lws \nu,\]
and similarly for quenched and doubly-quenched convergence.

In particular, suppose that $\l$ is a joining of the metric $G$-processes $(\X^G,\mu,S,d_\X)$ and $(\Y^G,\nu,S,d_\Y)$, and that $\l_n \in \Pr((\X\times \Y)^{V_n})$ satisfies
\[\l_n \lws \l.\]
If $\mu_n$ and $\nu_n$ are the marginals of $\l_n$ on $\X^{V_n}$ and $\Y^{V_n}$ respectively, then
\[\mu_n \lws \mu \quad \hbox{and} \quad \nu_n \lws \nu.\]

The same conclusions hold if locally weak$^\ast$ convergence is replaced with quenched or doubly-quenched convergence throughout. \qed
\end{cor}

The mere existence of a convergent sequence of measures on model spaces is a feature of a process that can be useful.  Proposition~\ref{prop:approx-meas-and-factor} has the following important consequence for this feature.

\begin{cor}
For a given metric $G$-process $(\X^G,\mu,S,d_\X)$ and sofic approximation $\S$, the property that there exists a sequence $\mu_n \in \Pr(\X^{V_n})$ which locally weak$^\ast$ (respectively, quenched or doubly-quenched) converges to $\mu$ is preserved by all factor maps, including all isomorphisms. In particular, it is a property of the process $(\X^G,\mu,S)$, not depending on the choice of the metric $d_\X$.

Therefore the definitions of these properties may be extended unambiguously to abstract $G$-systems.  \qed
\end{cor}

Beware of the following distinction: this corollary tells us that the \emph{existence} of a sequence $\mu_n \lws \mu$ is independent of $d_\X$, but whether a \emph{given} sequence $\mu_n$ satisfies this certainly does depend on $d_\X$.

\section{Model-measure sofic entropies}\label{sec:mod-meas-sof-ent}

Let $(\X^G,\mu,S)$ be a $G$-process and $d$ a compact generating metric on $\X$.  As in the Introduction, we define the \textbf{quenched model-measure sofic entropy} to be
\begin{multline*}
\rmh^{\rm{q}}_\S(\mu) := \sup\Big\{\sup_{\delta,\eps > 0}\limsup_{i \to\infty}\frac{1}{|V_{n_i}|}\log\rm{cov}_{\eps,\delta}\big(\mu_i,d^{(V_{n_i})}\big):\\ n_i \uparrow \infty\ \hbox{and}\ \mu_i \q \mu\ \hbox{over}\ (\s_{n_i})_{i\geq 1}\Big\},
\end{multline*}
and the \textbf{doubly-quenched model-measure sofic entropy} to be
\begin{multline*}
\rmh^{\rm{dq}}_\S(\mu) := \sup\Big\{\sup_{\delta,\eps > 0}\limsup_{i \to\infty}\frac{1}{|V_{n_i}|}\log\rm{cov}_{\eps,\delta}\big(\mu_i,d^{(V_{n_i})}\big):\\ n_i \uparrow \infty\ \hbox{and}\ \mu_i \dq \mu\ \hbox{over}\ (\s_{n_i})_{i\geq 1}\Big\}.
\end{multline*}

As for sofic entropy, it turns out that these do not depend on the choice of compact generating metric, and so it is omitted from the notation.

In general, it could happen that there are
\[n_1 < m_1 < n_2 < m_2 < \dots,\]
a w$^\ast$-neighbourhood $\calO$ of $\mu$, and a sequence of measure $\mu_i \in \Pr(\X^{V_{n_i}})$ such that
\[\mu_{n_i} \q \mu \quad \hbox{over}\ (\s_{n_i})_{i\geq 1},\]
but on the other hand
\begin{equation}\label{eq:some-empty}
\O(\calO,\s_{m_j}) = \emptyset \quad \forall j \geq 1.
\end{equation}

The first of these conditions implies that $\O(\calO,\s_{n_i}) \neq \emptyset$ for all sufficiently large $i$, and hence $\rmh_\S(\mu)\geq 0$.  Thus, the presence of any sofic sub-approximation along which one can find good models gives a lower bound on the sofic entropy.  We wish to define the model-measure sofic entropies so that they have the analogous property.

However, if~(\ref{eq:some-empty}) holds then there is no way to insert $(\mu_i)_{i\geq 1}$ into a sequence of measures which quenched-converges to $\mu$ over the whole of the original sofic approximation.  This is why we must explicitly allow a supremum over arbitrary sub-sequences $(n_i)_{i\geq 1}$ in the definitions of $\rmh^\rm{q}_\S$ and $\rmh^\rm{dq}_\S$, in addition to taking a limit supremum along those sub-approximations.

Just as for sofic entropy, it is sometimes important to know whether a restriction to sofic sub-approximations is really necessary in computing $\rmh^\rm{dq}_\S$.  This can be expressed by comparing with the \textbf{lower doubly-quenched model-measure sofic entropy}:
\[\ul{\rmh}_\S^\rm{dq}(\mu) := \sup\Big\{\sup_{\delta,\eps > 0}\liminf_{n \to\infty}\frac{1}{|V_n|}\log\rm{cov}_{\eps,\delta}\big(\mu_n,d^{(V_n)}\big):\ \mu_n \dq \mu\ \hbox{over}\ \S\Big\}.\]
This time we do insist that $\mu_n \dq \mu$ over the whole of the original sofic approximation $\S$, and then we take a limit infimum as $n\to\infty$, rather than a limit supremum.

The simplest relationship between sofic entropy and its model-measure variants is as follows.

\begin{lem}\label{lem:simple-ineq}
For any $G$-system $(\X^G,\mu,S,d)$ we have
\[\rmh_\S(\mu) \geq \rmh^\rm{q}_\S(\mu) \geq \rmh^{\rm{dq}}_\S(\mu).\]
\end{lem}

\begin{proof}
Firstly, suppose that $n_i \uparrow \infty$ and that $\mu_i \q \mu$ over $(\s_{n_i})_{i\geq 1}$.   For any $\eps,\delta > 0$ and any w$^\ast$-neighbourhood $\calO$ of $\mu$, condition (ii) in Definition~\ref{dfn:approxs} gives that
\[\mu_i(\O(\calO,\s_{n_i})) > 1 - \eps \quad \hbox{and hence} \quad \cov_{\eps,\delta}(\mu_i,d^{(V_{n_i})}) \leq \cov_\delta\big(\O(\calO,\s_{n_i}),d^{(V_{n_i})}\big)\]
for all sufficiently large $i$.  This shows that $\rmh^\rm{q}_\S(\mu) \leq \rmh_\S(\mu)$.

Second, since doubly-quenched convergence implies quenched convergence, the supremum defining $\rmh^\rm{dq}_\S(\mu)$ is over a subset of that defining $\rmh^\rm{q}_\S(\mu)$, so $\rmh^\rm{dq}_\S(\mu) \leq \rmh^\rm{q}_\S(\mu)$.
\end{proof}

\begin{ex}\label{ex:h-neq-hq}
Let $\X = \{0,1\}$, let $G$ be the free group on four generators, and let $(\X^G,\mu,S)$ be the $G$-process constructed in Example~\ref{ex:coind-ex}.  We saw in that example that $\rmh_\S(\mu) = 0$, but also that, once $n$ is large, any sufficiently good model for $\mu$ in $\X^{V_n}$ must be very close to the particular model $\bf{x}_n = 1_{W_n}$.  Therefore, if the measures $\mu_n \in \Pr(\X^{V_n})$ are asymptotically supported on good models of $\mu$, then they must be mostly supported on smaller and smaller Hamming balls around $\bf{x}_n$.  Once $n$ is large this has the following consequence: for most vertices $v \in V_n$, the marginal of $\mu_n$ around $v$ is close to either the Dirac mass at $0$ or the Dirac mass at $1$.  This violates the definition of local weak$^\ast$ convergence, and so there are no subsequence $n_1 < n_2 < \dots$ and measures satisfying $\mu_i \q \mu$ over $(\s_{n_i})_{i\geq 1}$.  Thus $\rmh^\rm{q}_\S(\mu) = -\infty$. \fin
\end{ex}

Of course, Lemma~\ref{lem:q-dq-and-wm} gives immediately that $\rmh^\rm{q}_\S(\mu) = \rmh^\rm{dq}_\S(\mu)$ if $(\X^G,\mu,S)$ is weakly mixing.  Kronecker systems can be examples in which $\rmh^\rm{q}_\S(\mu) > \rmh^\rm{dq}_\S(\mu)$.

\begin{ex}\label{ex:hq-neq-hdq}
If $G$ is a group with Kazhdan's property (T) in Example~\ref{ex:T}, then we produced a sequence $\mu_n \q \mu$ over the given sofic approximation $\S$, but also showed that there can be no sequence $\nu_n \dq \mu$.  The latter argument still holds over any sofic sub-approximation.  Therefore $\rmh_\S^\rm{q}(\mu) \geq 0$ (indeed, simple estimates show that it equals $0$) but $\rmh_\S^\rm{dq}(\mu) = -\infty$ in that example.

We also saw that $\rmh_\S(\mu\times \mu) = -\infty$ for this system.  Therefore this is an example in which one cannot replace $\rmh_\S^\rm{dq}$ with $\rmh_\S^\rm{q}$ in the conclusion of Theorem B. \fin
\end{ex}

\subsection{Invariance under isomorphism}\label{subs:iso-invar}

\begin{thm}\label{thm:iso-invar}
For a fixed sofic approximation $\S$, the quantities $\rmh^\rm{q}_\S(\mu)$ and $\rmh_\S^{\rm{dq}}(\mu)$ are isomorphism-invariants of the metric $G$-process $(\X^G,\mu,S,d)$.  In particular, they do not depend on the choice of $d$.
\end{thm}

Proposition~\ref{prop:approx-meas-and-factor} is the key to this result, together with the following.

\begin{prop}\label{prop:control-cov-of-img}
Suppose that $\Phi = \phi^G:(\X^G,\mu,S) \to (\Y^G,\nu,S)$ is a factor map, that $d_\X$ and $d_\Y$ are compact generating metrics of diameter at most $1$, that $\mu_n \q \mu$, and that $\psi_m \aL \phi$ rel $(\mu,d_\X,d_\Y)$.  For any $\eps,\delta \in (0,1)$ there is a $\delta' > 0$ for which the following holds.  Provided $m_1 \leq m_2 \leq \dots$ grows sufficiently slowly, we have
\[\cov_{\eps,\delta}\big((\psi^{\s_n}_{m_n})_\ast\mu_n,d_\Y^{(V_n)}\big) \leq \cov_{\eps/4,\delta'}(\mu_n,d_\X^{(V_n)})\]
for all sufficiently large $n$.
\end{prop}

\begin{proof}
Suppose that $\psi_m$ is an $\eta_m$-AL approximation to $\phi$ for each $m$, where $\eta_m \downarrow 0$.  We may assume that $(\eta_m)_{m\geq 1}$ is non-increasing. Then there are finite sets $D_m\subseteq G$, $D_m$-local open sets $U_m\subseteq \X^G$, and constants $L_m < \infty$ such that each $\psi_m:\X^G \to \Y$ is $\eta_m$-almost $L_m$-Lipschitz from $d_\X^{(D_m)}$ to $d_\Y$.

For each $m$, Lemma~\ref{lem:Lip-still-Lip} gives w$^\ast$-neighbourhoods $\calN_m$ of $\mu$ such that
\[\psi^{\s_n}_m|\O(\calN_m,\s_n)\]
is $(3\eta_m)$-almost $(L_m|D_m|)$-Lipschitz for all sufficiently large $n$.  Choose $m'$ large enough that $3\eta_{m'} < \delta/1000$.  Now choose $\delta'$ small enough that
\begin{equation}\label{eq:delta-prime-bound}
3\eta_{m'} + L_{m'}|D_{m'}|\delta' < \delta/2.
\end{equation}

Once $n$ is sufficiently large, we have
\[\mu_n(\O(\calN_{m'},\s_n)) > 1 - \eps/4 > 3/4.\]
This implies that
\begin{equation}\label{eq:put-cond-in}
\cov_{3\eps/4,\delta/2}\big((\psi^{\s_n}_{m'})_\ast\mu_n,d_\Y^{(V_n)}\big) \leq \cov_{\eps/2,\delta/2}\Big((\psi^{\s_n}_{m'})_\ast\big(\mu_{n|\O(\calN_{m'},\s_n)}\big),\ d_\Y^{(V_n)}\Big),
\end{equation}
where $\mu_{n|\O(\calN_{m'},\s_n)}$ is the measure $\mu_n$ conditioned on the subset $\O(\calN_{m'},\s_n)$.  It also implies that
\begin{equation}\label{eq:cov-cond-uncond}
\cov_{\eps/2,\delta'}(\mu_{n|\O(\calN_{m'},\s_n)},d_\X^{(V_n)}) \leq \cov_{\eps/4,\delta'}(\mu_n,d_\X^{(V_n)}),
\end{equation}
because if $\mu_n(F) < \eps/4$ then
\[\mu_{n|\O(\calN_{m'},\s_n)}(F) \leq \frac{\mu_n(F)}{\mu_n(\O(\calN_{m'},\s_n))} < \frac{\eps/4}{1 - \eps/4} < \eps/2.\]

For sufficiently large $n$, the fact that $\psi_{m'}^{\s_n}|\O(\calN_{m'},\s_n)$ is $(3\eta_{m'})$-almost $(L_{m'}|D_{m'}|)$-Lipschitz may be combined with Lemma~\ref{lem:img-covnos} to conclude that
\[\cov_{\eps/2,3\eta_{m'} + L_{m'}|D_{m'}|\delta'}\Big((\psi^{\s_n}_{m'})_\ast\nu,d_\Y^{(V_n)}\Big) \leq \cov_{\eps/2,\delta'}(\nu,d_\X^{(V_n)})\]
for any Borel probability measure $\nu$ supported on $\O(\calN_{m'},\s_n)$. Applying this with $\nu = \mu_{n|\O(\calN_{m'},\s_n)}$ and using~(\ref{eq:delta-prime-bound}) and~(\ref{eq:cov-cond-uncond}), we obtain
\begin{equation}\label{eq:img-control}
\cov_{\eps/2,\delta/2}\Big((\psi^{\s_n}_{m'})_\ast\big(\mu_{n|\O(\calN_{m'},\s_n)}\big),\ d_\Y^{(V_n)}\Big) \leq \cov_{\eps/4,\delta'}(\mu_n,d_\X^{(V_n)})
\end{equation}
for all sufficiently large $n$.

Finally, since $\eta_m$ decreases as $m\to \infty$, Lemma~\ref{lem:comparing-approxs} gives w$^\ast$-neighbourhoods $\calO_m$ of $\mu$ such that
\[d_\Y^{(V_n)}\big(\psi^{\s_n}_m(\bf{x}),\psi^{\s_n}_{m'}(\bf{x})\big) < 10 \eta_{m'} < \delta/2 \quad \forall \bf{x} \in \calO_m\ \forall m \geq m'.\]
By forming running intersections, we may assume that $\calO_1 \supseteq \calO_2 \supseteq \dots$.  Since $\mu_n(\O(\calO_m,\s_n))$ tends to $1$ for every $m$, it follows that, if $m_1 \leq m_2 \leq \dots$ grows sufficiently slowly, then
\[\mu_n\big\{\bf{x}:\ d_\Y^{(V_n)}\big(\psi^{\s_n}_{m_n}(\bf{x}),\psi^{\s_n}_{m'}(\bf{x})\big) > \delta/2\big\} \to 0 \quad \hbox{as}\ n\to\infty.\]
Therefore, for such a slowly-growing sequence, Lemma~\ref{lem:comparing-nearby-covnos} gives
\[\cov_{\eps,\delta}\big((\psi^{\s_n}_{m_n})_\ast\mu_n,d_\Y^{(V_n)}\big) \leq \cov_{3\eps/4,\delta/2}\big((\psi^{\s_n}_{m'})_\ast\mu_n,d_\Y^{(V_n)}\big)\]
for all sufficiently large $n$.  This completes the proof in combination with inequalities~(\ref{eq:put-cond-in}) and~(\ref{eq:img-control}).
\end{proof}

\begin{proof}[Proof of Theorem~\ref{thm:iso-invar}]
Let $\Phi = \phi^G:(\X^G,\mu,S)\to (\Y^G,\nu,S)$ be an isomorphism, choose compact generating metrics $d_\X$ for $\X$ and $d_\Y$ for $\Y$ of diameter at most $1$, and suppose that $\Phi^{-1} = \t{\phi}^G$.  Let $d$ be the Hamming average of $d_\X$ and $d_\Y$, and let $d_\X^{(2)}$ be the Hamming average of two copies of $d_\X$.

We will show that $\rmh_\S^{\rm{q}}(\mu) \leq \rmh_\S^{\rm{q}}(\nu)$; the reverse must also hold by symmetry.  An exactly analogous argument gives the proof for $\rmh_\S^{\rm{dq}}$.

If the left-hand side is $-\infty$ then there is nothing to prove. So let $n_i \uparrow \infty$ be a subsequence, suppose that $\mu_{n_i} \q \mu$ over $(\s_{n_i})_{i\geq 1}$, and let $\eps > 0$ and $\delta > 0$. We will produce measures $\nu_{n_i} \q \nu$ over $(\s_{n_i})_{i\geq 1}$ and a $\delta' > 0$ such that
\[\cov_{\eps,\delta}\big(\mu_{n_i},d_\X^{(V_{n_i})}\big) \leq \cov_{\eps/8,\delta'}\big(\nu_{n_i},d_\Y^{(V_{n_i})}\big)\]
for all sufficiently large $i$.  To this end, it suffices to consider only the sofic sub-approximation $(\s_{n_i})_{i\geq 1}$, so we may relabel this sub-approximation and simply assume that it equals $(\s_n)_{n\geq 1}$.

Now let $\psi_k \aL \phi$ rel $(\mu,d_\X,d_\Y)$ and $\t{\psi}_m \aL \t{\phi}$ rel $(\nu,d_\Y,d_\X)$.  In addition, let $\xi:\X^G\to \X$ be the projection onto the $e$-indexed coordinate, so $\xi^G = \rm{id}_{\X^G}$. Finally, let
\[\l := \int_{\X^G} \delta_{(x,\Phi(x))}\,\mu(\d x) \quad \hbox{and} \quad \hat{\l} := \int_{\X^G} \delta_{(x,x)}\,\mu(\d x).\]
These are the graphical joining of $\mu$ and $\nu$ corresponding to the factor map $\Phi$, and the diagonal joining of $\mu$ with itself, respectively.

Since $\xi$ may be regarded as a constant AL approximating sequence to itself, Lemma~\ref{lem:combine} and Corollary~\ref{cor:combine} give that
\begin{equation}\label{eq:join-still-AL-approx1}
(\xi,\psi_k) \aL (\xi,\phi) \quad \hbox{rel}\ (\mu,d_\X,d)
\end{equation}
and
\begin{equation}\label{eq:join-still-AL-approx2}
\xi\times \t{\psi}_m \aL \xi\times \t{\phi} \quad \hbox{rel}\ (\l,d,d^{(2)}_\X).
\end{equation}

By Proposition~\ref{prop:approx-meas-and-factor} and~(\ref{eq:join-still-AL-approx1}), if we choose $k_1 \leq k_2 \leq \dots$ growing sufficiently slowly, then
\begin{equation}\label{eq:lambdan}
\l_n := \big((\xi,\psi_{k_n})^{\s_n}\big)_\ast\mu_n \q \l.
\end{equation}
Fix such a sequence $(k_n)_{n\geq 1}$, and define $\nu_n := (\psi_{k_n}^{\s_n})_\ast\mu_n$, so this is the marginal of $\l_n$ on the space $\Y^{V_n}$.  By Corollary~\ref{cor:approx-meas-and-factor} we also have $\nu_n \q \nu$.

We now apply Proposition~\ref{prop:approx-meas-and-factor} with~(\ref{eq:join-still-AL-approx2}) and~(\ref{eq:lambdan}), and also Proposition~\ref{prop:control-cov-of-img}.  According to those propositions, if we choose $m_1 \leq m_2 \leq \dots$ growing sufficiently slowly, then we have
\begin{equation}\label{eq:conv-to-lhat}
((\xi \times \t{\psi}_{m_n})^{\s_n})_\ast\l_n \q (\xi^G\times \t{\Phi})_\ast\l = \hat{\l},
\end{equation}
and also there is a $\delta' > 0$ such that
\begin{equation}\label{eq:delta-prime}
\rm{cov}_{\eps/2,\delta/2}\big((\t{\psi}^{\s_n}_{m_n})_\ast\nu_n,d_\X^{(V_n)}\big) \leq \cov_{\eps/8,\delta'}(\nu_n,d_\Y^{(V_n)})
\end{equation}
for all sufficiently large $n$.

We finish the proof by comparing $\mu_n$ with its image measure
\[(\t{\psi}^{\s_n}_{m_n})_\ast\nu_n = (\t{\psi}^{\s_n}_{m_n})_\ast(\psi^{\s_n}_{k_n})_\ast\mu_n.\]
This is where we need the joining $\hat{\l}$.  On $\X^G\times \X^G$, let $F$ be the continuous function
\[F\big((x_g)_{g\in G},(x'_g)_{g\in G}\big) := d_\X(x_e,x'_e).\]
Then by~(\ref{eq:conv-to-lhat}) and a simple calculation we have
\begin{multline*}
\int d_\X^{(V_n)}\big(\bf{x},\t{\psi}^{\s_n}_{m_n}(\psi^{\s_n}_{k_n}(\bf{x}))\big)\,\mu_n(\d \bf{x}) = \iint F\,\d P^{\s_n}_{\left(\bf{x},\t{\psi}^{\s_n}_{m_n}\left(\psi^{\s_n}_{k_n}(\bf{x})\right)\right)}\,\mu_n(\d \bf{x})\\
= \iint F\,\d P^{\s_n}_{\left(\bf{x},\t{\psi}^{\s_n}_{m_n}(\bf{y})\right)}\,\l_n(\d \bf{x},\d \bf{y})
\to \int F\,\d\hat{\l} = 0.
\end{multline*}
Therefore Lemma~\ref{lem:comparing-nearby-covnos} gives
\[\cov_{\eps,\delta}\big(\mu_n,d_\X^{(V_n)}\big) \leq \cov_{\eps/2,\delta/2}\big((\t{\psi}^{\s_n}_{m_n})_\ast\nu_n,d_\X^{(V_n)}\big)\]
for all sufficiently large $n$.  Combining with~(\ref{eq:delta-prime}) completes the proof.
\end{proof}

\section{Entropy of Cartesian products}\label{sec:prod}

\subsection{Proof of Theorem B}

\begin{proof}[Proof of Theorem B]
Let $(\X^G,\mu,S,d_\X)$ and $(\Y^G,\nu,S,d_\Y)$ be metric $G$-processes, and suppose that $\eta > 0$. Let $d$ be the Hamming average of $d_\X$ and $d_\Y$ on $\X\times \Y$.  Let $h_1 := \rmh^\rm{dq}_\S(\mu)$ and $h_2 = \rmh_\S(\nu)$. We assume the latter is equal to $\ul{\rmh}_\S(\nu)$, so there is a $\delta > 0$ such that for any w$^\ast$-neighbourhood $\calO$ of $\nu$ we have
\begin{equation}\label{eq:cov-lower-bd}
\rm{cov}_\delta(\O(\calO,\s_n),d_\Y^{(V_n)}) \geq \rme^{(h_2 - \eta)|V_n|}
\end{equation}
for all sufficiently large $n$.

By shrinking $\delta$ further if necessary, and choosing $\eps > 0$ sufficiently small, we may also assume that there are a sequence $n_i \uparrow \infty$ and a sequence of measures $\mu_i \in \Pr(\X^{V_{n_i}})$ such that $\mu_{n_i} \dq \mu$ over $(\s_{n_i})_{i\geq 1}$ and
\[\cov_{\eps,\delta}(\mu_i,d_\X^{(V_{n_i})}) \geq \rme^{(h_1 - \eta)|V_{n_i}|}\]
for all sufficiently large $i$.

Let $\calN$ be any w$^\ast$-neighbourhood of $\mu\times \nu$.  We will prove that
\[\limsup_{n\to\infty}\frac{1}{|V_n|}\log \pack_{\delta/2}\big(\O(\calN,\s_n),d^{(V_n)}\big) \geq h_1 + h_2 - 2\eta.\]
Since $\eta > 0$ is arbitrary, and recalling the inequalities~(\ref{eq:pre-cov-and-pack}), this will complete the proof.

By conclusion (ii) of Theorem A, there is a w$^\ast$-neighbourhood $\calO$ of $\nu$ such that
\begin{equation}\label{eq:meas-lower-bd}
\inf_{\bf{y} \in \O(\calO,\s_{n_i})}\mu_i\big\{\bf{x}:\ (\bf{x},\bf{y}) \in \O(\calN,\s_{n_i})\big\} > 1 - \eps
\end{equation}
for all sufficiently large $i$ (note that~(\ref{eq:cov-lower-bd}) guarantees that $\O(\calO,\s_{n_i})$ is nonempty for all sufficiently large $i$).

Having chosen $\calO$, the inequalities~(\ref{eq:pre-cov-and-pack}) and the lower bound~(\ref{eq:cov-lower-bd}) let us choose subsets $F_n \subseteq \O(\calO,\s_n)$ which are $\delta$-separated according to the metrics $d_\Y^{(V_n)}$ and such that
\[|F_n| \geq \rme^{(h_2 - \eta)|V_n|}\]
for all sufficiently large $n$.

For each $\bf{y} \in F_n$, let
\[G_{n,\bf{y}} := \big\{\bf{x} \in \X^{V_n}:\ (\bf{x},\bf{y}) \in \O(\calN,\s_n)\big\},\]
so~(\ref{eq:meas-lower-bd}) implies that once $i$ is sufficiently large we have $\mu_i(G_{n_i,\bf{y}}) > 1 - \eps$ for all $\bf{y} \in F_{n_i}$.  This requires that
\[\cov_\delta(G_{n_i,\bf{y}},d_\X^{(V_{n_i})}) \geq \cov_{\eps,\delta}(\mu_i,d_\X^{(V_{n_i})}) \quad \forall \bf{y} \in F_{n_i}\]
once $i$ is sufficiently large.

Using the inequalities~(\ref{eq:pre-cov-and-pack}) again, we may therefore find further subsets $H_{n,\bf{y}} \subseteq G_{n,\bf{y}}$ for each $\bf{y} \in F_n$ which are $\delta$-separated according to the metrics $d_\X^{(V_n)}$ and such that
\[|H_{n_i,\bf{y}}| \geq \rme^{(h_1 - \eta)|V_{n_i}|} \quad \forall \bf{y} \in F_{n_i}\]
for all sufficiently large $i$.

Finally, defining
\[K_n := \{(\bf{x},\bf{y}):\ \bf{y} \in F_n,\ \bf{x} \in H_{n,\bf{y}}\}\]
for each $n$, it follows that these sets are $(\delta/2)$-separated according to $d^{(V_n)}$; that they are contained in $\O(\calN,\s_n)$; and that
\[\limsup_{n\to\infty}\frac{1}{|V_n|}\log|K_n| \geq \limsup_{i \to\infty}\frac{1}{|V_{n_i}|}\log|K_{n_i}| \geq h_1 + h_2 - 2\eta.\]
\end{proof}

\begin{rmk}
The above proof really shows that any $G$-systems $(X,\mu,T)$ and $(Y,\nu,S)$ satisfy
\[\rmh_\S(\mu\times \nu,T\times S) \geq \rmh^\rm{dq}_\S(\mu,T) + \ul{\rmh}_\S(\nu,S).\]
This conclusion may also be deduced formally from the statement of Theorem B, by first restricting to a sofic sub-approximation which nearly realizes the value $\rmh^\rm{dq}_\S(\mu,T)$, and then restricting to a further sub-approximation along which the sofic entropy and lower sofic entropy of $(Y,\nu,S)$ agree. \fin
\end{rmk}

\subsection{Proof of Theorem C}

The first assertion of Theorem C, that model-measure sofic entropy is subadditive, holds for arbitrary joinings, similarly to sofic entropy itself (see Proposition~\ref{prop:subadd}).

\begin{prop}\label{prop:h-q-subadd}
Let $(X,\mu,T)$ and $(Y,\nu,S)$ be $G$-systems and let $\l$ be a joining of them.  Then
\[\rmh^\rm{q}_\S(\l,T\times S) \leq \rmh^\rm{q}_\S(\mu,T) + \rmh^\rm{q}_\S(\nu,S),\]
and similarly if $\rmh^\rm{q}_\S$ is replaced with $\rmh^\rm{dq}_\S$ throughout.
\end{prop}

\begin{proof}
It suffices to consider two metric $G$-processes $(\X^G,\mu,S,d_\X)$ and $(\Y^G,\nu,S,d_\Y)$.  Let $d$ be the Hamming average of $d_\X$ and $d_\Y$ on $\X\times \Y$.

Suppose $(\s_{n_i})_{i\geq 1}$ is a sofic sub-approximation, that $\eps,\delta > 0$, and that $\l_i \q \l$ over $(\s_{n_i})_{i\geq 1}$.  By relabeling the sub-approximation, we may simply assume that it equals $(\s_n)_{n\geq 1}$, and write $\l_n \q \l$.

Let $\mu_n$ and $\nu_n$ be the marginals of $\l_n$ on $\X^{V_n}$ and $\Y^{V_n}$, so Corollary~\ref{cor:approx-meas-and-factor} gives that $\mu_n \q \mu$ and $\nu_n \q \nu$.  Part (i) of Lemma~\ref{lem:sum-cov} gives
\[\cov_{\eps,\delta}(\l_n,d^{(V_n)}) \leq \cov_{\eps/2,\delta}(\mu_n,d_\X^{(V_n)})\cdot \cov_{\eps/2,\delta}(\nu_n,d_\Y^{(V_n)}).\]
Taking logarithms and normalizing by $|V_n|$, this completes the proof, since $\eps$, $\delta$ and $\l_n$ were arbitrary.

The argument is the same if quenched convergence is replaced with doubly-quenched convergence.
\end{proof}

The reverse inequality required for Theorem C relies on the following lemma, which is another manifestation of the difference between quenched and doubly-quenched convergence.

\begin{lem}\label{lem:product-still-dqs}
If $\mu_n \in \Pr(\X^{V_n})$ and $\nu_n \in \Pr(\Y^{V_n})$ satisfy
\[\hbox{both} \quad \mu_n \dq \mu \quad \hbox{and} \quad \nu_n \dq \nu \quad \hbox{over}\ \S,\]
then
\[\mu_n \times \nu_n \dq \mu\times \nu \quad \hbox{over}\ \S.\]
\end{lem}

\begin{proof}
Repeatedly applying the implication (i $\Longrightarrow$ iii) in Theorem A gives
\begin{align*}
\nu_n \times \nu_n \q \nu\times \nu \quad &\Longrightarrow \quad \mu_n \times \nu_n \times \nu_n \q \mu\times \nu\times \nu\\ &\Longrightarrow \quad \mu_n \times \mu_n \times \nu_n \times \nu_n \q \mu \times \mu\times \nu\times \nu,
\end{align*}
and hence $\mu_n \times \nu_n \dq \mu\times \nu$.
\end{proof}

\begin{proof}[Proof of Theorem C]
Let $(\X^G,\mu,S,d_\X)$, $(\Y^G,\nu,S,d_\Y)$ and $d$ be as in the proof of Proposition~\ref{prop:h-q-subadd}.  That proposition has given subadditivity, so it remains to prove the reverse inequality.

We need only consider the case in which $\rmh_\S^\rm{dq}(\mu),\rmh_\S^\rm{dq}(\nu) \geq 0$, for otherwise subadditivity gives also that $\rmh_\S^\rm{dq}(\mu\times \nu) = -\infty$.

Since we assume that $\rmh^\rm{dq}_\S(\nu) = \ul{\rmh}_\S^\rm{dq}(\nu)$, for any $\eta > 0$ there are $\eps,\delta > 0$ and a sequence $\nu_n \in \Pr(\Y^{V_n})$ such that $\nu_n \dq \nu$ over $\S$ and
\[\cov_{\sqrt{\eps},\delta}(\nu_n,d_\Y^{(V_n)}) \geq \rme^{(\rmh^\rm{dq}_\S(\nu) - \eta)|V_n|}\]
for all sufficiently large $n$.  Of course, this remains true if we pass to any sofic sub-approximation.

After shrinking $\eps$ and $\delta$ if necessary, now let $(\s_{n_i})_{i\geq 1}$ be a sofic sub-approximation and let $\mu_i \in \Pr(\X^{(V_{n_i})})$ be a sequence such that $\mu_i \dq \mu$ over $(\s_{n_i})_{i\geq 1}$ and
\[\cov_{\sqrt{\eps},\delta}(\mu_i,d_\Y^{(V_{n_i})}) \geq \rme^{(\rmh^\rm{dq}_\S(\mu) - \eta)|V_{n_i}|}\]
for all sufficiently large $i$.

Lemma~\ref{lem:product-still-dqs} gives that $\mu_i \times \nu_{n_i} \dq \mu\times \nu$ over $(\s_{n_i})_{i\geq 1}$, and Corollary~\ref{cor:sum-cov} gives
\[\cov_{\eps,\delta/4}(\mu_i \times \nu_{n_i},d^{(V_{n_i})}) \geq \cov_{\sqrt{\eps},\delta}(\mu_i,d_\X^{(V_{n_i})})\cdot \cov_{\sqrt{\eps},\delta}(\nu_{n_i},d_\Y^{(V_{n_i})}).\]
Taking logarithms and normalizing by $|V_{n_i}|$, this completes the proof.

If $\mu\times \nu$ is ergodic then we may argue in just the same way using only quenched convergence, because in that case
\[\mu_i \q \mu \quad \hbox{and} \quad \nu_{n_i} \q \nu \quad \hbox{imply} \quad \mu_i \times \nu_{n_i} \q \mu\times \nu.\]
\end{proof}

Example~\ref{ex:hq-neq-hdq} shows that the assumption that $\mu \times \nu$ is ergodic cannot be dropped.

The proof of Theorem C gives the following important special case.

\begin{cor}\label{cor:dq-is-stable}
For any $G$-system $(X,\mu,T)$ we have
\[\rmh^\rm{dq}_\S(\mu^{\times k},T^{\times k}) = k\cdot \rmh^\rm{dq}_\S(\mu,T) \quad \forall k \geq 1.\]
\end{cor}

Here we do not need the assumption that $\rmh^\rm{dq}_\S = \ul{\rmh}_\S^\rm{dq}$.  This is because we are now combining $(X,\mu,T)$ with itself, so there is no risk that two different sofic sub-approximations are needed to obtain the entropies of the ingredient systems.

\begin{proof}
If $\rmh^\rm{dq}_\S(\mu,T) = -\infty$ then the result is trivial, so suppose otherwise.

Let $(\X^G,\mu,S,d)$ be a metric $G$-process. For any $\eta > 0$, there are $\eps,\delta > 0$, a sofic sub-approximation $\S' = (\s_{n_i})_{i\geq 1}$, and a sequence of measures $\mu_{n_i} \dq \mu$ over $\S'$ such that
\[\cov_{\eps,\delta}(\mu_i,d^{(V_{n_i})}) \geq \rme^{(\rmh^\rm{dq}_\S(\mu) - \eta)|V_{n_i}|}\]
for all sufficiently large $i$.  It follows that
\[\rmh^\rm{dq}_\S(\mu) \geq \rmh^\rm{dq}_{\S'}(\mu) \geq \ul{\rmh}^\rm{dq}_{\S'}(\mu) \geq \rmh^\rm{dq}_\S(\mu) - \eta.\]

Since $\eta$ was arbitrary, a simple diagonal argument now gives a sofic sub-approximation $\S''$ such that in fact
\[\rmh^\rm{dq}_{\S''}(\mu) = \ul{\rmh}^\rm{dq}_{\S''}(\mu) \geq \rmh^\rm{dq}_\S(\mu) - \eta.\]
A $k$-fold application of Theorem C with this sofic sub-approximation gives that
\[\rmh^\rm{dq}_\S(\mu^{\times k}) \geq \rmh^\rm{dq}_{\S''}(\mu^{\times k}) = k\cdot \rmh^\rm{dq}_{\S''}(\mu) \geq k\cdot (\rmh^\rm{dq}_\S(\mu) - \eta).\]
Since $\eta$ was arbitrary, this completes the proof.
\end{proof}

\section{Processes with finite state spaces}\label{sec:finite-state-spaces}

\subsection{Alternative formulae for the entropies}

Now let $(\X^G,\mu,S)$ be a $G$-process with finite state space $\X$.  Let $d$ be the discrete metric on $\X$ (all distances are zero or one). For this process there are alternative, simpler formulae for the sofic entropy and model-measure sofic entropy.  These will be essential in the proof of Theorem D.

\begin{prop}\label{prop:finite-X-formulae}
For a $G$-process $(\X^G,\mu,S)$ with $|\X| < \infty$, we have
\[\rmh_\S(\mu) = \inf_\calO \limsup_{n\to\infty}\frac{1}{|V_n|}\log |\O(\calO,\s_n)|,\]
where $\calO$ ranges over w$^\ast$-neighbourhoods of $\mu$,
and
\[\rmh^{\rm{dq}}_\S(\mu) := \sup\Big\{\sup_{\eps > 0}\limsup_{i\to\infty}\frac{1}{|V_{n_i}|}\log\rm{cov}_\eps(\mu_i): n_i \uparrow \infty\ \hbox{and}\ \mu_i \dq \mu\ \hbox{over}\ (\s_{n_i})_{i\geq 1}\Big\},\]
where
\[\rm{cov}_\eps(\mu_i) = \min\{|F|:\ F\subseteq \X^{V_{n_i}}\ \hbox{with}\ \mu_i(F) > 1 - \eps\}.\]
\end{prop}

This proposition is a consequence of the following bound on the volumes of Hamming balls, which is implied by standard estimates in Information Theory: see, for instance,~\cite[Section 5.4]{GolPin--book}.

\begin{lem}\label{lem:ball-small}
If $(\X,d)$ is a nonempty finite set with its discrete metric, then for every $\eta > 0$ there is a $\delta > 0$ such that the following holds.  If $V$ is a nonempty finite set and $\bf{x} \in \X^V$, then
\[|B_\delta(\bf{x})| \leq \rme^{\eta|V|},\]
where $B_\delta(\bf{x})$ is the $\delta$-ball around $\bf{x}$ for the metric $d^{(V)}$. \qed
\end{lem}

\begin{proof}[Proof of Proposition~\ref{prop:finite-X-formulae}]
The formula for $\rmh_\S(\mu)$ is easily seen to be equivalent to Bowen's original definition of sofic entropy for systems with finite generating partitions~(\cite{Bowen10}).  Its agreement with $\rmh_\S(\mu)$ is therefore contained in~\cite{KerLi11a}.

The reasoning for model-measure sofic entropy is very similar. Clearly we always have
\[\cov_{\eps,\delta}(\mu_i,d^{(V_{n_i})}) \leq \cov_\eps(\mu_i),\]
and so $\rmh^\rm{dq}_\S(\mu)$ is bounded from above by the right-hand side of the desired formula.  On the other hand, for any $\eta > 0$, Lemma~\ref{lem:ball-small} gives a $\delta > 0$ such that all $\delta$-balls in the metric $d^{(V_{n_i})}$ have size at most $\rme^{\eta|V_{n_i}|}$.  This implies that
\[|B_\delta(F)| \leq \rme^{\eta|V_{n_i}|}|F| \quad \forall F\subseteq \X^{V_{n_i}},\]
and hence that
\[\frac{1}{|V_{n_i}|}\log\cov_{\eps,\delta}(\mu_i,d^{(V_{n_i})}) \geq \frac{1}{|V_{n_i}|}\log\cov_\eps(\mu_i) - \eta \quad \forall i\geq 1.\]
Since $\eta$ can be made arbitrarily small, $\rmh^\rm{dq}_\S(\mu)$ is also bounded from below by the right-hand side of the desired formula.
\end{proof}

\subsection{Proof of Theorem D}

One half of Theorem D is true for general systems.

\begin{lem}\label{lem:half-ThmD}
Any $G$-system $(X,\mu,T)$ satisfies
\[\rmh^\rm{dq}_\S(\mu,T) \leq \rmh^\ps_\S(\mu,T).\]
\end{lem}

\begin{proof}
Combining Lemma~\ref{lem:simple-ineq} and Corollary~\ref{cor:dq-is-stable} gives
\[\rmh^\rm{dq}_\S(\mu,T) = \frac{1}{k}\rmh^\rm{dq}_\S(\mu^{\times k},T^{\times k}) \leq \frac{1}{k}\rmh_\S(\mu^{\times k},T^{\times k}) \quad \forall k\geq 1,\]
and hence $\rmh^{\rm{dq}}_\S(\mu) \leq \rmh^\ps_\S(\mu)$.
\end{proof}

We prove the reverse half of Theorem D using the assumption that $|\X|$ is finite.  In this case we always endow $\X$ with the discrete metric $d$.  As remarked in the Introduction, this gives the result for any $G$-system that has a finite generating partition, including any ergodic $G$-system for which $\rmh^{\rm{Rok}}(\mu,T) < \infty$ (see~\cite{Seward--KriI}).

The proof relies on producing model-measures for $\mu$ out of good models for $\mu^{\times k}$ for large values of $k$.  This is done using the following proposition.

\begin{prop}\label{prop:making-model-measures}
Let $V$ be any finite set and $\s:G\to\rm{Sym}(V)$ any map.

For any $\eps > 0$ and w$^\ast$-neighbourhood $\calO$ of $\mu\times \mu$, the following holds for any sufficiently large $k\geq 1$: there is a w$^\ast$-neighbourhood $\calN$ of $\mu^{\times k}$ such that, if
\[\bf{x} = (\bf{x}_1,\dots,\bf{x}_k)\in \O(\calN,\s),\]
then the measure
\[\varrho := \frac{1}{k}\sum_{i=1}^k \delta_{\bf{x}_i} \in \Pr(\X^{V_n})\]
satisfies
\[|\{v:\ (\Pi^\s_v)_\ast (\varrho \times \varrho) \in \calO\}| > (1-\eps)|V| \quad \hbox{and} \quad (\varrho \times \varrho)(\O(\calO,\s)) > 1- \eps.\]
\end{prop}

We consider a neighbourhood $\calO$ of $\mu\times \mu$, rather than $\mu$, to ensure that the resulting model measures not only quenched converge but doubly-quenched converge.

\begin{proof}
It suffices to prove this for a sub-basic family of neighbourhoods $\calO$, so we may assume that
\[\calO = \Big\{\theta \in \Pr(\X^G\times \X^G):\ \int f\otimes h\,\d\theta \approx_{2\k} \int f\,\d\mu \int h\,\d\mu\Big\}\]
for some $h,f \in C(\X^G)$ with $\|f\|_\infty, \|h\|_\infty \leq 1$ and some $\k > 0$.

\vspace{7pt}

\emph{Part 1.} \quad For each $k$, define two continuous functions $(\X^G)^k \to \bbR$ by
\[F_k(x_1,\dots,x_k) = \frac{1}{k}\sum_{i=1}^k f(x_i) \quad \hbox{and} \quad H_k(x_1,\dots,x_k) = \frac{1}{k}\sum_{i=1}^k h(x_i),\]
and let
\begin{multline*}
U_k := \Big\{(x_1,\dots,x_k) \in (\X^G)^k:\ F_k(x_1,\dots,x_k) \approx_\k \int f\,\d\mu \\ \hbox{and} \ H_k(x_1,\dots,x_k) \approx_\k \int h\,\d\mu\Big\}.
\end{multline*}
This $U_k$ is open, and the Law of Large Numbers gives that $\mu^{\times k}(U_k) \to 1$ as $k\to\infty$.  Letting
\[\calN_{1,k} := \big\{\nu \in \Pr((\X^G)^k):\ \nu(U_k) > 1 - \eps\big\},\]
it follows that this is a w$^\ast$-neighbourhood of $\mu^{\times k}$ for all sufficiently large $k$.

If $(\bf{x}_1,\dots,\bf{x}_k) \in \O(\calN_{1,k},\s)$, then this asserts that
\[P^\s_{(\bf{x}_1,\dots,\bf{x}_k)}(U_k) > 1 - \eps,\]
and hence that the set
\begin{multline*}
V_k^\rm{good} := \Big\{v\in V:\ F_k(\Pi^\s_v(\bf{x}_1),\dots,\Pi^\s_v(\bf{x}_k)) \approx_\k \int f\,\d\mu \\ \hbox{and} \ H_k(\Pi^\s_v(\bf{x}_1),\dots,\Pi^\s_v(\bf{x}_k)) \approx_\k \int h\,\d\mu\Big\}
\end{multline*}
has $|V_k^\rm{good}| > (1-\eps)|V|$.

Now let $\varrho$ be as in the statement of the proposition, and observe that
\[\int f\otimes h\ \d((\Pi^\s_v)_\ast (\varrho\times \varrho)) = \int (f\circ \Pi^\s_v)\,\d\varrho \cdot \int (h\circ \Pi^\s_v)\,\d\varrho.\]
The first of these right-hand factors is equal to
\[\int (f\circ \Pi^\s_v)\,\d\varrho = \frac{1}{k}\sum_{i=1}^k f(\Pi^\s_v(\bf{x}_i)) = F_k(\Pi^\s_v(\bf{x}_1),\dots,\Pi^\s_v(\bf{x}_k)),\]
and similarly the second is equal to $H_k(\Pi^\s_v(\bf{x}_1),\dots,\Pi^\s_v(\bf{x}_k))$.  Therefore any $v \in V_k^\rm{good}$ satisfies
\begin{multline*}
\int f\otimes h\ \d((\Pi^\s_v)_\ast (\varrho \times \varrho))\\ = F_k(\Pi^\s_v(\bf{x}_1),\dots,\Pi^\s_v(\bf{x}_k))\cdot H_k(\Pi^\s_v(\bf{x}_1),\dots,\Pi^\s_v(\bf{x}_k)) \approx_{2\k} \int f\,\d\mu \int h\,\d\mu
\end{multline*}
where the last estimate by $2\k$ uses the fact that both of the factors here lie in $[-1,1]$.

Hence
\[|\{v:\ (\Pi^\s_v)_\ast (\varrho \times \varrho) \in \calO\}| \geq |V_k^\rm{good}| > (1-\eps)|V|\]
once $k$ is sufficiently large.

\vspace{7pt}

\emph{Part 2.}\quad The second part is simpler.  The set
\begin{multline*}
\calN_{2,k} := \Big\{\nu \in \Pr((\X^G)^k):\ \int f(x_i)h(x_j)\,\nu(\d x_1,\dots,\d x_k) \approx_{2\k} \int f\,\d\mu \int h\,\d\mu \\ \hbox{whenever}\ i,j \in\{1,\dots,k\}\ \hbox{are distinct}\Big\}
\end{multline*}
is another a w$^\ast$-neighbourhood of $\mu^{\times k}$ for every $k$.  If $(\bf{x}_1,\dots,\bf{x}_k) \in \O(\calN_{2,k},\s)$ and $i \neq j$, then
\[\int f\otimes h\,\d P^\s_{(\bf{x}_i,\bf{x}_j)} = \int f(x_i)h(x_j)\,P^\s_{(\bf{x}_1,\dots,\bf{x}_k)}(\d x_1,\dots,\d x_k) \approx_{2\k} \int f\,\d\mu \int h\,\d\mu:\]
that is, $P^\s_{(\bf{x}_i,\bf{x}_j)} \in \calO$.  Therefore
\[(\rho\times \rho)(\O(\calO,\s)) = \frac{|\{(i,j) \in \{1,\dots,k\}^2:\ P^\s_{(\bf{x}_i,\bf{x}_j)} \in \calO\}|}{k^2} \geq \frac{k(k-1)}{k^2},\]
and this is greater than $1-\eps$ once $k$ is large enough.

\vspace{7pt}

\emph{Completion.}\quad Choose $k$ large enough to satisfy both parts above, and set
\[\calN := \calN_{1,k}\cap \calN_{2,k}.\]
\end{proof}

\begin{proof}[Proof of Theorem D]
One inequality is already given by Lemma~\ref{lem:half-ThmD}, so we focus on the other.

Let $\eps > 0$.  Let $\calO_1 \supseteq \calO_2 \supseteq \dots$ be a basis of w$^\ast$-neighbourhoods of $\mu$.  By Proposition~\ref{prop:making-model-measures}, there are integers $1 \leq k_1 \leq k_2 \leq \dots$ tending to $\infty$ and w$^\ast$-neighbourhoods $\calN_j$ of $\mu^{\times k_j}$ for every $j$ such that the following holds.  If
\[\vec{\bf{x}} := (\bf{x}_1,\dots,\bf{x}_{k_j}) \in \O(\calN_j,\s_n) \quad \hbox{for some}\ j\ \hbox{and}\ n,\]
then the measure
\[\varrho^{n,j}_{\vec{\bf{x}}} := \frac{1}{k_j}\sum_{i=1}^{k_j}\delta_{\bf{x}_i}\]
satisfies
\begin{multline}\label{eq:indivs-conv}
|\{v \in V_n:\ (\Pi^{\s_n}_v)_\ast(\varrho^{n,j}_{\vec{\bf{x}}}\times \varrho^{n,j}_{\vec{\bf{x}}}) \in \calO_j\}| > (1-2^{-j})|V_n| \\ \hbox{and} \quad (\varrho^{n,j}_{\vec{\bf{x}}}\times \varrho^{n,j}_{\vec{\bf{x}}})(\O(\calO_j,\s_n)) > 1 - 2^{-j}.
\end{multline}

Next, by the definition of $\rmh^\ps_\S$, we may also choose a subsequence $n_1 < n_2 < \dots$ such that
\begin{equation}\label{eq:OmegaNlarge}
|\O(\calN_j,\s_{n_j})| \geq \exp\big(k_j(\rmh^\ps_\S(\mu) - \eps)|V_{n_j}|\big) \geq 1
\end{equation}
for all $j$.  Now set
\[\mu_j := \frac{1}{|\O(\calN_j,\s_{n_j})|}\sum_{\vec{\bf{x}} \in \O(\calN_j,\s_{n_j})} \varrho^{n_j,j}_{\vec{\bf{x}}} \quad \hbox{for}\ j\geq 1.\]

Since the sets $\calO_j$ form a basis of neighbourhoods around $\mu\times \mu$, the bounds~(\ref{eq:indivs-conv}) imply that for any sequence of single $k_j$-tuples $\vec{\bf{x}}_j \in \O(\calN_j,\s_{n_j})$, we have
\[\rho^{n_j,j}_{\vec{\bf{x}}_j} \dq \mu \quad \hbox{along}\ (\s_{n_j})_{j\geq 1}.\]
We can quickly strengthen this conclusion as follows: if $\vec{\bf{y}}_j \in \O(\calN_j,\s_{n_j})$ is any other sequence of $k_j$-tuples, then the implication (i $\Longrightarrow$ iii) of Theorem A gives that
\[\rho^{n_j,j}_{\vec{\bf{x}}_j} \times \rho^{n_j,j}_{\vec{\bf{y}}_j} \q \mu\times \mu.\]
By simply averaging this last assertion, it follows that
\[\mu_j\times \mu_j = \frac{1}{|\O(\calN_j,\s_{n_j})|^2}\sum_{\vec{\bf{x}},\vec{\bf{y}} \in \O(\calN_j,\s_{n_j})} \varrho^{n_j,j}_{\vec{\bf{x}}} \times \varrho^{n_j,j}_{\vec{\bf{y}}} \ \q \ \mu\times \mu,\]
and hence $\mu_j \dq \mu$.

Finally, let
\[\rmH(2\eps',1-2\eps') := -2\eps'\log (2\eps') - (1-2\eps')\log(1-2\eps') \quad \hbox{for}\ \eps' \in (0,1/2),\]
and choose $\eps'$ so small that
\[2\eps'\log|\X| + \rmH(2\eps',1-2\eps') < \eps.\]
Consider the covering numbers $\cov_{\eps'}(\mu_j)$.  For each $j$, let $F_j \subseteq \X^{V_{n_j}}$ be a minimum-size subset for which $\mu_j(F_j) > 1 - \eps'$.  By the definition of $\mu_j$ and Chebyshev's Inequality, this implies that at least half of the tuples $(\bf{x}_1,\dots,\bf{x}_{k_j}) \in \O(\calN_{n_j},\s_{n_j})$ satisfy
\[\varrho^{n_j,j}_{(\bf{x}_1,\dots,\bf{x}_{k_j})}(F_j) = \frac{|\{i \in \{1,\dots,k_j\}:\ \bf{x}_i \in F_j\}|}{k_j} > 1 - 2\eps'.\]

On the other hand, a simple estimate using volumes of Hamming-balls (c.f.~\cite[Section 5.4]{GolPin--book}) gives
\begin{multline*}
\Big|\Big\{(\bf{x}_1,\dots,\bf{x}_{k_j}) \in (\X^{V_{n_j}})^{k_j}:\ \frac{|\{i \in \{1,\dots,k_j\}:\ \bf{x}_i \in F_j\}|}{k_j} > 1 - 2\eps' \Big\}\Big|\\
\leq |F_j|^{k_j}\cdot |\X|^{2\eps' \cdot k_j\cdot |V_{n_j}|}\cdot 2^{\rmH(2\eps',1 - 2\eps')k_j},
\end{multline*}
where the last factor estimates the number of ways of choosing at most $2\eps' k_j$ coordinates $i \in \{1,\dots,k_j\}$ at which to allow $\bf{x}_i \not\in F_j$.

Therefore
\begin{multline*}
\frac{1}{2}|\O(\calN_{n_j},\s_{n_j})| \leq |F_j|^{k_j}\cdot |\X|^{2\eps' \cdot k_j\cdot |V_{n_j}|}\cdot 2^{\rmH(2\eps',1 - 2\eps')k_j}\\ = (\cov_{\eps'}(\mu_j))^{k_j}\cdot |\X|^{2\eps'\cdot k_j\cdot |V_{n_j}|}\cdot 2^{\rmH(2\eps',1 - 2\eps')k_j}.
\end{multline*}
Combining this with~(\ref{eq:OmegaNlarge}) and re-arranging, we obtain
\[\frac{1}{|V_{n_j}|}\log \cov_{\eps'}(\mu_j) \geq \rmh^\ps_\S(\mu) -\eps - 2\eps'\log|\X| - \frac{\rmH(2\eps',1-2\eps')\cdot \log 2}{|V_{n_j}|} - O\Big(\frac{1}{|k_j|}\Big),\]
and this lower bound is greater than $\rmh^\ps_\S(\mu) - 2\eps$ for all sufficiently large $j$.  Since $\eps$ was arbitrary, this shows that $\rmh^\rm{dq}_\S(\mu) \geq \rmh^\ps_\S(\mu)$.
\end{proof}

\begin{proof}[Proof of Corollary D$'$]
Theorem B has already proved that
\[\rmh^\rm{dq}_\S(\mu,T) \leq \rmh_\S(\mu\times \nu,T\times S) - \rmh_\S(\nu,S)\]
for any other $G$-system $(Y,\nu,S)$ satisfying $\rmh_\S(\nu,S) = \ul{\rmh}_\S(\nu,S)$.

On the other hand, if $(X,\mu,T)$ has a finite generating partition, then Theorem D shows that
\[\frac{1}{k}\rmh_\S(\mu^{\times k},T^{\times k}) \to \rmh^\rm{dq}_\S(\mu,T).\]
In particular, for any $\eps > 0$, there must be infinitely many $k$ for which
\[\rmh_\S(\mu^{\times (k+1)},T^{\times (k+1)}) \leq \rmh_\S(\mu^{\times k},T^{\times k}) + \rmh^\rm{dq}_\S(\mu,T) + \eps.\]
Letting $(Y,\nu,S) := (X^k,\mu^{\times k},T^{\times k})$ and re-arranging, this becomes
\[\rmh_\S(\mu\times \nu,T\times S) - \rmh_\S(\nu,S) \leq \rmh^\rm{dq}_\S(\mu,T) + \eps.\]
Since we also have $\rmh_\S(\nu,S) = \ul{\rmh}_\S(\nu,S)$ for this system $(Y,\nu,S)$ by assumption, these examples complete the proof.
\end{proof}

Another corollary seems worth including at this point.  For any $G$-system $(X,\mu,T)$, the definition of $\rmh^\ps_\S$ gives that
\[\rmh_\S^\ps(\mu^{\times k},T^{\times k}) = \lim_{n\to\infty} \frac{1}{n}\rmh_\S(\mu^{\times kn},T^{\times kn}) = k\cdot \rmh^\ps_\S(\mu,T) \quad \forall k\geq 1.\]
If $(Y,\nu,S)$ is another system with the property that
\begin{equation}\label{eq:upper=lower-strong}
\rmh_\S(\nu^{\times k},S^{\times k}) = \ul{\rmh}_\S(\nu^{\times k},S^{\times k}) \quad \forall k \geq 1,
\end{equation}
then we may take $k^\rm{th}$ Cartesian powers of both systems in Theorem B to obtain
\[\frac{1}{k}\rmh_\S((\mu\times\nu)^{\times k},(T\times S)^{\times k}) \geq \rmh^\ps_\S(\mu,T) + \frac{1}{k}\rmh_\S(\nu^{\times k},S^{\times k}) \quad \forall k \geq 1.\]
Letting $k \to \infty$, we conclude that $\rmh^\ps_\S$ \emph{is} strictly additive.

\begin{cor}\label{cor:strictly-add}
If $(X,\mu,T)$ has a finite generating partition and~(\ref{eq:upper=lower-strong}) is satisfied then
\[\rmh^\ps_\S(\mu\times\nu,T\times S) = \rmh^\ps_\S(\mu,T) + \rmh^\ps_\S(\nu,S).\]
\qed
\end{cor}

\vspace{7pt}

\subsection{Some remarks on systems without finite generating partitions}

I do not know whether $\rmh^\rm{ps}_\S = \rmh^\rm{dq}_\S$ for arbitrary systems.  If this is so, then Corollary D$'$ and Corollary~\ref{cor:strictly-add} can be extended to them.  Also, Theorem B could be re-written with $\rmh^\rm{dq}_\S$ replaced by $\rmh^\rm{ps}_\S$, so that model measures are not needed to give a meaningful lower bound on the sofic entropy of a Cartesian product.

The difficulty in the general case seems to be the following.  Written out in full, the power-stabilized entropy for a metric $G$-process $(\X^G,\mu,S,d)$ is
\[\rmh^\ps_\S(\mu) := \limsup_{k\to\infty}\frac{1}{k}\sup_{\delta > 0}\ \inf_{\rm{Int}(\calO) \ni \mu^{\times k}}\ \limsup_{n\to\infty}\frac{1}{|V_n|}\log \rm{cov}_\delta\Big(\O(\calO,\s_n),(d^{(k)})^{(V_n)}\Big).\]
If this is non-negative, then for every $k\geq 1$ and $\eps > 0$ there is a $\delta_k > 0$ such that
\[\inf_{\rm{Int}(\calO) \ni \mu^{\times k}}\ \limsup_{n\to\infty}\frac{1}{|V_n|}\log \rm{cov}_{\delta_k}\Big(\O(\calO,\s_n),(d^{(k)})^{(V_n)}\Big) \geq k(\rmh^\rm{ps}_\S(\mu) - \eps).\]

However, as far as I know, it could happen that we must choose smaller and smaller values of $\delta_k$ as $k\to\infty$.  On the other hand, in order to make contact with $\rmh_\S^\rm{dq}$, we must find \emph{fixed} $\eps,\delta > 0$ and measures satisfying $\mu_i \dq \mu$ along $(\s_{n_i})_{i\geq 1}$ such that $\cov_{\eps,\delta}(\mu_i,d^{(V_{n_i})})$ grows fast enough as $i\to\infty$.  If the choice of $\delta_k$ tends to $0$ as $k\to\infty$, and then we construct the measures $\mu_j$ as in the proof of Theorem D, we do not obtain control over $\cov_{\eps,\delta}(\mu_j,d^{(V_{n_j})})$ for any fixed $\delta > 0$.

To get around this problem, one could simply re-define $\rmh^\rm{ps}_\S$ so that the supremum over $\delta$ appears on the outside: let us set
\[\t{\rmh}^\ps_\S(\mu) := \sup_{\delta > 0}\Big[\limsup_{k\to\infty}\frac{1}{k}\ \inf_{\rm{Int}(\calO) \ni \mu^{\times k}}\ \limsup_{n\to\infty}\frac{1}{|V_n|}\log \rm{cov}_\delta\Big(\O(\calO,\s_n),(d^{(k)})^{(V_n)}\Big)\Big].\]
Using this quantity, the construction used to prove Theorem D does generalize quite easily, leading to the inequality
\[\t{\rmh}^\rm{ps}_\S(\mu) \leq \rmh^\rm{dq}_\S(\mu).\]
However, now the argument of Lemma~\ref{lem:half-ThmD} seems to run into difficulty, and I cannot show that
\[\t{\rmh}^\rm{ps}_\S(\mu) \geq \rmh^\rm{dq}_\S(\mu).\]

Thus, Theorem D will hold for arbitrary systems if the supremum over $\delta$ may be exchanged with the limit supremum over $k$ in the formula for $\rmh^\ps_\S(\mu)$.  If $\X$ is finite and $d$ is the discrete metric, then Proposition~\ref{prop:finite-X-formulae} lets one switch to counting individual models, so that $\delta$ disappears altogether from $\rmh^\ps_\S(\mu)$ and $\rmh^\rm{dq}_\S(\mu)$.  This is why the proofs above could be completed when $\X$ is finite.

\section{Co-induced systems}\label{sec:when-equal}

Consider the setting of Theorem E.  In view of Lemma~\ref{lem:simple-ineq}, that theorem will follow if we show that
\begin{equation}\label{eq:remaining-ineq}
\rmh_{\S\times \rm{T}}\big(\mu^{\times H},\rm{CInd}_G^{G\times H}T\big) \leq \rmh^\rm{dq}_{\S\times \rm{T}}\big(\mu^{\times H},\rm{CInd}_G^{G\times H}T\big).
\end{equation}

As usual, we can assume that we start with a metric $G$-process $(\X^G,\mu,S,d)$.  After co-induction this simply becomes $(\X^{G\times H},\mu^{\times H},\t{S},d)$, where $\t{S}$ is the right-shift action of $G\times H$.

To prove Theorem E, we also need to use the left-action $\t{T}$ of $H$ on $\X^{G\times H}$:
\[\t{T}^h\big((x_{g,k})_{(g,k) \in G\times H}\big) = (x_{g,h^{-1}k})_{(g,k) \in G\times H}.\]
The product measure $\mu^{\times H}$ is invariant under this action, as well as under $\t{S}$: this special feature of the measure is crucial for the proof. Since $H$ is infinite, the $H$-system $(\X^{G\times H},\mu^{\times H},\t{T})$ is weakly mixing.

The action $\t{T}$ commutes with $\t{S}$, and so each transformation $\t{T}^h$ is an automorphism of $(X^H,\mu^{\times H},\t{S})$. Therefore the results of Section~\ref{sec:factors} apply to each of these transformations.  Each $\t{T}^h$ is already defined coordinate-wise by the $\{(e_G,h^{-1})\}$-local map
\[\X^{G\times H}\to \X: (x_{g,k})_{g,k} \mapsto x_{e_G,h^{-1}},\] which is $1$-Lipschitz from $d^{(\{(e_G,h^{-1})\})}$ to $d$.  Therefore there is no need to introduce other AL approximations to these maps.

For each $\s_n$, $\tau_n$ and $h \in H$, the map
\[\rm{id}_{V_n}\times \tau_n^h:V_n\times W_n\to V_n\times W_n\]
has an `approximate adjoint' defined by setting
\[\rho_n^h:\X^{V_n\times W_n}\to\X^{V_n\times W_n}:(x_{v,w})_{v\in V_n,\,w\in W_n} \mapsto (x_{v,\s_n^{h^{-1}}(w)})_{v\in V_n,\,w \in W_n}.\]
Such maps were already discussed in the remark following Lemma~\ref{lem:approx-equiv}: as explained there, they become useful only now that our measure $\mu^{\times H}$ is also left-shift-invariant.  In terms of these, a special case of Lemma~\ref{lem:psi-sig-compatible} translates as follows.

\begin{lem}\label{lem:approx-centralizer}
If $F\subseteq G\times H$ and $E \subseteq H$ are finite, then the following holds w.h.p in $(v,w) \in V_n\times W_n$:
\[\Pi^{\s_n\times \tau_n}_{(v,w)}(\rho_n^h(\cdot))|_F = \big(\t{T}^h(\Pi^{\s_n\times \tau_n}_{(v,w)}(\cdot))\big)\big|_F \quad \forall h \in E.\]
\qed
\end{lem}

\begin{lem}\label{lem:make-dq}
Let $E_1$, $E_2$, \dots be any sequence of finite subsets of $H$ satisfying $|E_m|\to \infty$.  Also, suppose that $\theta_n \in \Pr(\X^{V_n\times W_n})$ is any sequence of measures satisfying
\begin{equation}\label{eq:prelim-conv}
\theta_n(\O(\calO,\s_n\times \tau_n)) \to 1 \quad \hbox{as}\ n\to\infty
\end{equation}
for any w$^\ast$-neighbourhood $\calO$ of $\mu^{\times H}$.

Provided the sequence $m_1 \leq m_2 \leq \dots$ grows sufficiently slowly, the sequence of measures
\[\mu_n := \frac{1}{|E_{m_n}|}\sum_{h \in E_{m_n}}(\rho_n^h)_\ast\theta_n\]
doubly-quenched converges to $\mu^{\times H}$ over $\S\times \rm{T}$.
\end{lem}

\begin{proof}
It suffices to show that $\mu_n \lws \mu^{\times H}$; since the co-induced system is weakly mixing, this implies doubly-quenched convergence.

Also, it suffices to consider a sub-basic w$^\ast$-neighbourhood of $\mu^{\times H}$, so let
\[\calO := \Big\{\nu \in \Pr(\X^{G\times H}):\ \int f\,\d\nu \approx_\k \int f\,\d\mu^{\times H}\Big\}\]
for some local function $f \in C(\X^{G\times H})$ and $\k > 0$.

Now for each $m$ let
\[U_m = \Big\{x \in \X^{G\times H}:\ \frac{1}{|E_m|}\sum_{h \in E_m}f(\t{T}^hx) \approx_\k \int f\,\d\mu\Big\}.\]
Since $|E_m|\to \infty$, the Law of Large Numbers gives that $\mu^{\times H}(U_m) \to 1$ as $m\to\infty$ (note that this works even if there is no ergodic theorem over the sets $E_m$ for general $H$-systems).  Choose real values $\a_m < \mu^{\times H}(U_m)$ which still tend to $1$, and for each $m$ let
\[\calN_m := \big\{\nu \in \Pr(\X^{G\times H}):\ \nu(U_m) > \a_m \big\}.\]
Each $\calN_m$ is a w$^\ast$-neighbourhood of $\mu^{\times H}$, and so~(\ref{eq:prelim-conv}) implies that
\[\theta_n(\O(\calN_{m_n},\s_n\times \tau_n)) \to 1 \quad \hbox{as}\ n\to\infty\]
provided $m_1 \leq m_2 \leq \dots$ grows sufficiently slowly.  In terms of empirical distributions, this implies that
\begin{multline}\label{eq:mostly-good}
\int \frac{\big|\big\{(v,w) \in V_n\times W_n:\ \Pi^{\s_n\times \tau_n}_{(v,w)}(\bf{x}) \in U_{m_n}\big\}\big|}{|V_n\times W_n|}\,\theta_n(\d\bf{x}) \\
= \frac{1}{|V_n\times W_n|}\sum_{(v,w) \in W_n\times W_n}\theta_n\big\{\bf{x}:\ \Pi^{\s_n\times \tau_n}_{(v,w)}(\bf{x}) \in U_{m_n}\big\} \to 1
\end{multline}
as $n\to\infty$.

Finally, we have
\[\int f\ \d((\Pi^{\s_n\times \tau_n}_{(v,w)})_\ast\mu_n) = \int_{\X^{V_n\times W_n}}\Big(\frac{1}{|E_{m_n}|}\sum_{h \in E_{m_n}}f\big(\Pi^{\s_n\times \tau_n}_{(v,w)}(\rho^h_n(\bf{x}))\big)\Big)\ \theta_n(\d \bf{x}).\]
Since $f$ is a local function, another appeal to Lemma~\ref{lem:approx-centralizer} gives that, w.h.p. in $(v,w)$, this is equal to
\[\int_{\X^{V_n\times W_n}}\Big(\frac{1}{|E_{m_n}|}\sum_{h \in E_{m_n}}f\big(\t{T}^h(\Pi^{\s_n\times \tau_n}_{(v,w)}(\bf{x}))\big)\Big)\ \theta_n(\d \bf{x}).\]
Recalling the definition of $U_{m_n}$ and the inequality~(\ref{eq:mostly-good}), this, in turn, lies within $\k$ of $\int f\,\d\mu$ w.h.p. in $(v,w)$.  That is,
\[(\Pi^{\s_n\times \tau_n}_{(v,w)})_\ast\mu_n \in \calO \quad \hbox{w.h.p. in}\ (v,w),\]
as required.
\end{proof}

\begin{proof}[Proof of Theorem E]
We need only prove the inequality~(\ref{eq:remaining-ineq}) for the co-induced process. If $\rmh_{\S\times \rm{T}}(\mu^{\times H}) = -\infty$ then the result is trivial, so suppose otherwise, let $h_1 < \rmh_{\S\times \rm{T}}(\mu^{\times H})$ be arbitrary, and let $h_2$ lie strictly between these two values.

Let $d$ be a compact generating metric on $\X$, and let $E_1$, $E_2$, \dots be finite subsets of $H$ with $|E_m|\to \infty$.

From the definition of $\rmh_{\S\times \rm{T}}(\mu^{\times H})$, it follows that there are $\delta > 0$ and a sequence of subsets $A_n \subseteq \X^{V_n\times W_n}$ such that
\begin{itemize}
\item[i)] each $A_n$ is $\delta$-separated according to $d^{(V_n\times W_n)}$,
\item[ii)] for every w$^\ast$-neighbourhood $\calO$ of $\mu^{\times H}$ we have
\[A_n \subseteq \O(\calO,\s_n\times \tau_n)\]
for all sufficiently large $n$, and
\item[iii)] $|A_n| \geq \exp(h_2|V_n||W_n|)$ for infinitely many $n$.
\end{itemize}
By passing to a subsequence $n_1 < n_2 < \dots$, we may now assume that (iii) holds for all sufficiently large $n$, and in particular that $A_n \neq \emptyset$ for every $n$.

Let $\theta_n$ be the uniform measure on $A_n$ for each $n$.  Then condition (ii) above shows that these satisfy the hypotheses of Lemma~\ref{lem:make-dq}, and so
\[\mu_n := \frac{1}{|E_{m_n}|}\sum_{h \in E_{m_n}}(\rho_n^h)_\ast\theta_n \dq \mu^{\times H}\]
for some $m_1 \leq m_2 \leq \dots$ tending slowly to $\infty$.

Finally, for each $n$ let $B_n \subseteq \X^{V_n\times W_n}$ be a subset of minimal cardinality such that
\[\mu_n(B_{\delta/2}(B_n)) > 1/2.\]
From the definition of $\mu_n$, this requires that
\[\theta_n\big((\rho^h_n)^{-1}(B_n)\big) = \frac{|A_n \cap (\rho^h_n)^{-1}(B_n)|}{|A_n|} > \frac{1}{2} \quad \hbox{for some}\ h \in F_{m_n}.\]
Since each $\rho^h_n$ is an isometry of the metric $d^{(V_n\times W_n)}$, this and property (i) require that
\[|B_n| \geq |A_n|/2.\]
Therefore, since property (iii) now holds for all sufficiently large $n$, we have
\[\rm{cov}_{1/2,\delta/2}(\mu_n,d^{(V_n\times W_n)}) = |B_n| > \exp(h_1|V_n||W_n|)\]
for all sufficiently large $n$, and thus
\[\rmh_{\S\times \rm{T}}^{\rm{dq}}(\mu^{\times H}) \geq h_1.\]
Since $h_1 < \rmh_{\S\times\rm{T}}(\mu^{\times H})$ was arbitrary, this completes the proof.
\end{proof}

In case $G$ is trivial, Theorem E just asserts that all our sofic entropy-notions coincide for Bernoulli $H$-systems.  However, in that case one could give a much simpler proof: if $(\X^H,\nu^{\times H},S)$ is a Bernoulli $H$-process and $d$ a compact generating metric on $\X$, then Lemma~\ref{lem:Bern-dq} gives that $\nu^{\times W_n} \dq \nu^{\times H}$, and this sequence of measures achieves the full sofic entropy of the process, which just equals the Shannon entropy of $\nu$.

It is worth comparing Theorem E with~\cite[Theorem 4.1]{Bowen11}, which gives other sufficient conditions for a $G$-system $(X,\mu,T)$ to satisfy $\rmh_\S(\mu,T) = \rmh^\rm{q}_\S(\mu,T)$. Bowen's assumptions are that $G$ is residually finite and that $\S$ consists of quotients by finite-index normal subgroups of $G$.  This looks quite different from Theorem E: in the first place, Theorem E assumes that the system is of a special kind, whereas Bowen's condition is mostly about the sofic approximation. However, the proofs do have a similar flavour.  A key point in the proof of Theorem E is that the \emph{left}-shift action of $H$ on $X^H$ commutes with the co-induced $(G\times H)$-action and is weakly mixing.  Bowen's proof also requires that there be a `sufficiently large' group commuting with a given action: in his case, that the $G$-actions on the sofic approximations $V_n$ commute with some transitive actions from the other side.  In both cases, the proof uses this large centralizer for some auxiliary averaging, which converts single good models into measures.  It would be interesting to find some way of unifying these two sufficient conditions.  It would also worth knowing whether Bowen's condition can be generalized in such a way that every sofic group has some sofic approximation which satisfies it.

In light of the role played by the left-shift $H$-action in the proof of Theorem E, I suspect it might have a far-reaching generalization as follows.  Let $(X,\mu,T)$ be a $G$-system.  Let $\rm{Aut}(X,\mu)$ denote the group of all measure-preserving automorphisms of the probability space $(X,\mu)$, up to agreement $\mu$-almost everywhere.  This is a Polish group in its coarse topology.  The $G$-action $T$ defines a homomorphism $G\to \rm{Aut}(X,\mu)$, and we define the \textbf{centralizer} of $T$ to be the subgroup of elements of $\rm{Aut}(X,\mu)$ which commute with the image of that homomorphism.

\begin{conj}\label{conj:centralizer}
If the centralizer of $T$ is ergodic, then $\rmh_\S(\mu,T) = \rmh^{\rm{q}}_\S(\mu,T)$.  If the centralizer is weakly mixing, then $\rmh_\S(\mu,T) = \rmh^{\rm{dq}}_\S(\mu,T)$. 
\end{conj}

For example, the centralizer of the co-induced $(G\times H)$-system $(X^H,\mu^{\times H},\rm{CInd}_G^{G\times H}T)$ includes the left-shift action of $H$, which is Bernoulli and therefore weakly mixing.

On the other hand, I do not believe that the conditions in Conjecture~\ref{conj:centralizer} are necessary.  On the contrary, if $G$ is amenable, then one should always have
\[\rmh_\S(\mu,T) = \rmh_\S^\rm{q}(\mu,T) = \rmh_\S^\rm{dq}(\mu,T) = \rmh_{\rm{KS}}(\mu,T)\]
for any sofic approximation $\S$.  The equality of the first and last values here is shown in~\cite{Bowen12}, and I think similar methods are able to prove the others.

\section{Some directions for further study}\label{sec:open}

Just as for sofic entropy itself, the following basic fact is not known for the model-measure sofic entropies.  It suggests a major gap in our present understanding.

\begin{ques}
Are there a sofic group $G$, a $G$-system $(X,\mu,T)$, and two different sofic approximations $\S$ and $\S'$ to $G$ such that
\[0 \leq \rmh^\rm{q}_\S(\mu,T) < \rmh^\rm{q}_{\S'}(\mu,T),\]
or similarly with $\rmh^\rm{q}$ replaced by $\rmh^\rm{dq}$?
What if $G$ is a free group?
\end{ques}

There are cases in which some sofic approximations give a non-negative value, while others give $-\infty$, just as there are for sofic entropy itself.

The following related questions are also open.

\begin{ques}
Are there a sofic group $G$, a $G$-system $(X,\mu,T)$, and a sofic approximation $\S$ such that at least two of the quantities
\[\rmh_\S(\mu,T), \quad \rmh^\rm{q}_\S(\mu,T) \quad \hbox{and} \quad \rmh^\rm{dq}_\S(\mu,T)\]
are non-negative, but are not equal? What if $G$ is a free group?
\end{ques}

\begin{ques}
Are there a sofic group $G$, a $G$-system $(X,\mu,T)$, and a sofic approximation $\S$ such that the sequence
\[\frac{1}{k}\rmh_\S(\mu^{\times k},T^{\times k}), \quad k\geq 1,\]
contains at least two distinct non-negative values? How about for $\rmh_\S^\rm{q}$? Are there examples in which
\[0 \leq \rmh_\S^\rm{ps}(\mu,T) < \rmh_\S(\mu,T)?\]
What if $G$ is a free group?
\end{ques}

Another possibility that might be worth pursuing is that certain choices of sofic approximation give some simplification of the entropy theories.

\begin{ques}
Let $G$ be a sofic group.  Is there a sofic approximation $\S$ to $G$ such that
\[\rmh_\S(\mu,T) = \rmh_\S^\rm{dq}(\mu,T)\]
for all $G$-systems $(X,\mu,T)$?
\end{ques}

In this case, I think an obvious candidate is to start with an arbitrary sofic approximation $\S_0 = (\s_n:G\to \rm{Sym}(V_n))_{n\geq 1}$, and then let $\S$ be the sequence
\[\s_n^{\times m_n}:G\to \rm{Sym}(V_n^{m_n})\]
for some slowly-growing sequence $m_1 \leq m_2 \leq \dots$.  It might be that for this $\S$, some variation of the averaging argument used to prove Theorem E would answer the above question positively.  I have not pursued this idea very far.

In case $G$ is a free group, Bowen introduced another entropy-like invariant called the `f-invariant' in~\cite{Bowen10free}, and denoted it by $\rm{f}(\mu,T)$.  In~\cite{Bowen10c}, he then showed that $\rm{f}(\mu,T)$ may be expressed as a kind of average of sofic entropies over random sofic approximations.  As a result, the f-invariant may have better behaviour than the sofic entropy along any give sofic approximation.  It would be interesting to study its additivity properties using the method of the present paper.

\bibliographystyle{abbrv}
\bibliography{bibfile}

\def\cprime{$'$} \def\cprime{$'$}
\begin{thebibliography}{10}

\bibitem{BekdelaHVal08}
B.~Bekka, P.~de~la Harpe, and A.~Valette.
\newblock {\em Kazhdan's property ({T})}, volume~11 of {\em New Mathematical
  Monographs}.
\newblock Cambridge University Press, Cambridge, 2008.

\bibitem{BenSch01}
I.~Benjamini and O.~Schramm.
\newblock Recurrence of distributional limits of finite planar graphs.
\newblock {\em Electron. J. Probab.}, 6:no. 23, 13 pp. (electronic), 2001.

\bibitem{Bowen10c}
L.~Bowen.
\newblock The ergodic theory of free group actions: entropy and the
  {$f$}-invariant.
\newblock {\em Groups Geom. Dyn.}, 4(3):419--432, 2010.

\bibitem{Bowen10}
L.~Bowen.
\newblock Measure conjugacy invariants for actions of countable sofic groups.
\newblock {\em J. Amer. Math. Soc.}, 23(1):217--245, 2010.

\bibitem{Bowen10b}
L.~Bowen.
\newblock Non-abelian free group actions: {M}arkov processes, the
  {A}bramov-{R}ohlin formula and {Y}uzvinskii's formula.
\newblock {\em Ergodic Theory Dynam. Systems}, 30(6):1629--1663, 2010.

\bibitem{Bowen11}
L.~Bowen.
\newblock Entropy for expansive algebraic actions of residually finite groups.
\newblock {\em Ergodic Theory Dynam. Systems}, 31(3):703--718, 2011.

\bibitem{Bowen12}
L.~Bowen.
\newblock Sofic entropy and amenable groups.
\newblock {\em Ergodic Theory Dynam. Systems}, 32(2):427--466, 2012.

\bibitem{Bowen10free}
L.~P. Bowen.
\newblock A measure-conjugacy invariant for free group actions.
\newblock {\em Ann. of Math. (2)}, 171(2):1387--1400, 2010.

\bibitem{DooZha12}
A.~H. Dooley and G.~Zhang.
\newblock Co-induction in dynamical systems.
\newblock {\em Ergodic Theory Dynam. Systems}, 32(3):919--940, 2012.

\bibitem{GolPin--book}
C.~M. Goldie and R.~G.~E. Pinch.
\newblock {\em Communication theory}, volume~20 of {\em London Mathematical
  Society Student Texts}.
\newblock Cambridge University Press, Cambridge, 1991.

\bibitem{Gro99}
M.~Gromov.
\newblock Endomorphisms of symbolic algebraic varieties.
\newblock {\em J. Eur. Math. Soc. (JEMS)}, 1(2):109--197, 1999.

\bibitem{Hay14}
B.~Hayes.
\newblock Polish models and sofic entropy.
\newblock Preprint, available online at \verb|arXi.org|: 1411.1510.

\bibitem{Kec10}
A.~S. Kechris.
\newblock {\em Global aspects of ergodic group actions}, volume 160 of {\em
  Mathematical Surveys and Monographs}.
\newblock American Mathematical Society, Providence, RI, 2010.

\bibitem{KerLi11a}
D.~Kerr and H.~Li.
\newblock Bernoulli actions and infinite entropy.
\newblock {\em Groups Geom. Dyn.}, 5(3):663--672, 2011.

\bibitem{KerLi11b}
D.~Kerr and H.~Li.
\newblock Entropy and the variational principle for actions of sofic groups.
\newblock {\em Invent. Math.}, 186(3):501--558, 2011.

\bibitem{KerLi13}
D.~Kerr and H.~Li.
\newblock Soficity, amenability, and dynamical entropy.
\newblock {\em Amer. J. Math.}, 135(3):721--761, 2013.

\bibitem{Li12}
H.~Li.
\newblock Compact group automorphisms, addition formulas and
  {F}uglede-{K}adison determinants.
\newblock {\em Ann. of Math. (2)}, 176(1):303--347, 2012.

\bibitem{Lubot--book}
A.~Lubotzky.
\newblock {\em Discrete groups, expanding graphs and invariant measures}.
\newblock Modern Birkh\"auser Classics. Birkh\"auser Verlag, Basel, 2010.
\newblock With an appendix by Jonathan D. Rogawski, Reprint of the 1994
  edition.

\bibitem{MezMon09}
M.~M{\'e}zard and A.~Montanari.
\newblock {\em Information, physics, and computation}.
\newblock Oxford Graduate Texts. Oxford University Press, Oxford, 2009.

\bibitem{MontMosSly12}
A.~Montanari, E.~Mossel, and A.~Sly.
\newblock The weak limit of {I}sing models on locally tree-like graphs.
\newblock {\em Probab. Theory Related Fields}, 152(1-2):31--51, 2012.

\bibitem{MosNeeSly15}
E.~Mossel, J.~Neeman, and A.~Sly.
\newblock Reconstruction and estimation in the planted partition model.
\newblock {\em Probab. Theory Related Fields}, 162(3-4):431--461, 2015.

\bibitem{Pet83}
K.~E. Petersen.
\newblock {\em Ergodic {T}heory}.
\newblock Cambridge University Press, Cambridge, 1983.

\bibitem{Seward--KriI}
B.~Seward.
\newblock Krieger's finite generator theorem for actions of countable groups
  {I}.
\newblock Preprint, available online at \verb|arXiv.org|: 1405.3604.

\bibitem{Weiss00}
B.~Weiss.
\newblock Sofic groups and dynamical systems.
\newblock {\em Sankhy\=a Ser. A}, 62(3):350--359, 2000.
\newblock Ergodic theory and harmonic analysis (Mumbai, 1999).

\end{thebibliography}

\parskip 0pt
\parindent 0pt

\vspace{7pt}

\small{\textsc{Courant Institute of Mathematical Sciences, New York University, 251 Mercer St, New York NY 10012, USA}

\vspace{7pt}

Email: \verb|tim@cims.nyu.edu|2

URL: \verb|cims.nyu.edu/~tim|}

\enddocument